\documentclass[]{article}
\usepackage[utf8]{inputenc}
\usepackage[a4paper,top=2cm,bottom=2.4cm,left=2cm,right=2cm]{geometry}
\usepackage{graphicx} 
\usepackage{float}
\usepackage{textcomp}
\usepackage{xcolor}
\usepackage{amsmath}
\usepackage{amssymb}
\usepackage{amsthm}
\usepackage{amsfonts}
\usepackage[english]{babel}
\usepackage{verbatim}
\usepackage{array}
\usepackage{arydshln}
\usepackage{hyperref}
\usepackage{cases}
\usepackage{multirow}
\usepackage{enumerate}
\usepackage{cite}
\usepackage{afterpage}
\usepackage{url}    
\usepackage{esint}    
\usepackage{caption}
\usepackage{subcaption}    
\usepackage[title]{appendix}

\newcommand{\authoraddress}[2]{%
	\textsc{#1} \textit{E-mail address:} \protect\url{#2}%
}                                                                                                                     
\newcommand{\AuthorAddressone}{                                                                                          
	\authoraddress{Institute for Applied Mathematics, University of Bonn, 53115 Bonn, Germany.}{dematte@iam.uni-bonn.de}%
} 
\newcommand{\AuthorAddresstwo}{                                                                                          
	\authoraddress{Institute for Applied Mathematics, University of Bonn, 53115 Bonn, Germany.}{velazquez@iam.uni-bonn.de}%
} 
\numberwithin{equation}{section}

\newtheorem{lemma}{Lemma}[section]

\newtheorem{prop}{Proposition}[section]

\theoremstyle{definition}

\theoremstyle{remark}
\newtheorem*{remark}{Remark}

\newcommand{\rchi}{\protect\raisebox{2pt}{$ \chi $}}

\newcommand{\Ss}{{\mathbb{S}^2}}
\newcommand{\RR}{\mathbb{R}}
\newcommand{\Rot}{\mathcal{R}}
\newcommand{\eps}{\varepsilon}

\newcommand{\bnd}{\partial\Omega}

\DeclareMathOperator*{\Div}{div}
\newcommand{\bigslant}[2]{{\raisebox{.2em}{$#1$}\left/\raisebox{-.2em}{$#2$}\right.}}
\title{Equilibrium and Non-Equilibrium diffusion approximation for the radiative transfer equation}
\author{Elena Demattè\thanks{\AuthorAddressone}, Juan J.L. Velázquez\thanks{\AuthorAddresstwo}}

\begin{document}
	
	\maketitle
	\begin{abstract}
	In this paper we study the distribution of the temperature within a body where the heat is transported only by radiation. Specifically, we consider the situation where both emission-absorption and scattering processes take place. We study the initial boundary value problem given by the coupling of the radiative transfer equation with the energy balance equation on a convex domain $ \Omega \subset \RR^3 $ in the diffusion approximation regime, i.e. when the mean free path of the photons tends to zero. Using the method of matched asymptotic expansions we will derive the limit initial boundary value problems for all different possible scaling limit regimes and we will classify them as equilibrium or non-equilibrium diffusion approximation. Moreover, we will observe the formation  of boundary and initial layers for which suitable equations are obtained. We will consider both stationary and time dependent problems as well as different situations in which the light is assumed to propagate either instantaneously or with finite speed.

	\end{abstract}
	\textbf{Acknowledgments:} The authors gratefully acknowledge the financial support of the collaborative research centre \textit{The mathematics of emerging effects} (CRC 1060, Project-ID 211504053) and Bonn International Graduate School of Mathematics (BIGS) at the Hausdorff Center for Mathematics founded through the Deutsche Forschungsgemeinschaft (DFG, German Research Foundation) under Germany’s Excellence Strategy – EXC-2047/1 – 390685813. \\
	
	
	\textbf{Statements and Declarations:} The authors have no relevant financial or non-financial interests to disclose.\\
	
	\textbf{Data availability:} Data sharing not applicable to this article as no datasets were generated or analysed during the current study.
	\tableofcontents
	
	\section{Introduction}
	The kinetic equation which describes the interaction of matter with photons is the radiative transfer equation. The radiative transfer equation can be written including absorption-emission processes and scattering processes in a rather general setting as 
	\begin{equation}\label{RTE}
		\frac{1}{c}\partial_t I_\nu(t,x,n)+n\cdot \nabla_x I_\nu(t,x,n)=\alpha_\nu^e-\alpha_\nu^aI_\nu(t,x,n)+\alpha_\nu^s\left(\int_\Ss K(n,n')I_\nu(t,x,n')\;dn'-I_\nu(t,x,n)\right).
	\end{equation}
We denote by $ I_\nu(t,x,n) $ the radiation intensity, i.e. the distribution of energy of photons moving at time $ t>0 $, at position $ x\in\Omega\subset \RR^3 $ and in direction $ n\in\Ss $ with frequency $ \nu>0 $. Moreover, $ c $ is the speed of light in the medium that will be assumed to be constant. The parameters $ \alpha^e $, $ \alpha^a $ and $ \alpha^s $ are respectively the emission, absorption and scattering coefficients. These are functions that can depend on the frequency $ \nu $, on the position $ x $ or in the case of local thermal equilibrium on the local temperature $ T(x) $. The function $ K $ is the scattering kernel. It can be considered as the probability rate of a photon to be deflected from an incident direction $ n'\in\Ss $ to a new direction $ n\in\Ss $. The scattering kernel $ K $ can be assumed also to depend on the frequency $ \nu\in\RR_+ $, however in this paper we omit the dependence on $ \nu $ in order to simplify the notation. Notice that all the results that we will present in this paper holds also in the case where $ K $ is also a function of $ \nu $.

In this paper we will study the heat transfer by means of radiation under some assumptions. First of all we consider only the case of local thermal equilibrium in which the temperature $ T(t,x) $ is well-defined at any point $ x\in\Omega $ and for any time $ t>0 $. This is not necessarily the case in situations where the microscopic processes driving the system towards equilibrium are slow. Such problems arise in applications to astrophysics (cf. \cite{oxenius}). Under this assumption the emission coefficient takes a particular form. Indeed it is given by $ \alpha_\nu^e= \alpha_\nu^a B_\nu(T(t,x)) $, where $ B_\nu(T)= \frac{2h\nu^3}{c^2}\frac{1}{e^{\frac{h\nu}{kT}}-1} $ is the Planck distribution of a black body. We assume also that the considered material is isotropic without a preferred direction of scattering, hence the scattering kernel $ K $ is invariant under rotations.

We couple the radiative transfer equation with the energy balance equation
\begin{equation}\label{div-free}
	C\partial_t T(t,x)+\frac{1}{c}\partial_t\left(\int_0^\infty d\nu\int_\Ss dn\;I_\nu(t,x,n)\right)+\Div\left(\int_0^\infty d\nu\int_\Ss dn\;nI_\nu(t,x,n)\right)=0,
\end{equation}
where $ C>0 $ is the volumetric heat capacity of the material. The combined system \eqref{RTE} and \eqref{div-free} allows to determine the temperature of the system at any point when the heat is transferred only by means of radiation. Notice that in \eqref{div-free} we are not considering other heat transport processes such as conduction or convection. After a suitable time rescaling we can assume $ C=1 $.
As boundary condition we consider a source of radiation placed at infinity. Mathematically we impose
\begin{equation}\label{bnd}
	I_\nu(t,x,n)=g_\nu(t,n)\;\;\text{ if } x\in\bnd\text{ and }n\cdot n_x<0,
\end{equation}
where $ n_x\in\Ss $ is the outer normal to the boundary at point $ x $. However, we could consider a more general setting with the incoming boundary profile $ g_\nu(t,x,n) $ depending also on $ x\in\bnd $.\\

In this paper we will consider both time dependent and stationary cases. Assuming $ \Omega\subset \RR^3 $ bounded and convex and as initial values the bounded functions $ I_0(x,n,\nu) $ and $ T_0(x) $ we consider the following initial-boundary value problem
	\begin{equation}\label{rtetime}
		\begin{cases}
		\frac{1}{c}\partial_tI_\nu(t,x,n)+n\cdot \nabla_x I_\nu(t,x,n)=\alpha_\nu^a(x)\left(B_\nu(T(t,x))-I_\nu(t,x,n)\right)\\\text{\phantom{nel mezzo del cammin di nostra}}+\alpha_\nu^s(x)\left(\int_\Ss K(n,n')I_\nu(t,x,n')\;dn'-I_\nu(t,x,n)\right)& x\in\Omega,n\in\Ss,t>0\\
		\partial_tT+\partial_t \left(\int_0^\infty d\nu\int_\Ss dn \;I_\nu(t,n,x)\right) +\Div\left(\int_0^\infty d\nu\int_\Ss dn \;nI_\nu(t,n,x)\right)=0& x\in\Omega,n\in\Ss,t>0\\
		I_\nu(0,x,n)=I_0(x,n,\nu)& x\in\Omega,n\in\Ss\\
		T(0,x)=T_0(x)& x\in\Omega\\
		I_\nu(t,n,x)=g_\nu(t,n)& x\in\bnd,n\cdot n_x<0,t>0
		\end{cases}
	\end{equation}
	and the following stationary boundary value problem
		\begin{equation}\label{rtestationary}
	\begin{cases}
	n\cdot \nabla_x I_\nu(x,n)=\alpha_\nu^a(x)\left(B_\nu(T(x))-I_\nu(x,n)\right)\\\text{\phantom{nel mezzo del cammin di nostra}}+\alpha_\nu^s(x)\left(\int_\Ss K(n,n')I_\nu(x,n')\;dn'-I_\nu(x,n)\right)& x\in\Omega,n\in\Ss\\
	\Div\left(\int_0^\infty d\nu\int_\Ss dn \;nI_\nu(n,x)\right)=0& x\in\Omega,n\in\Ss\\
	I_\nu(n,x)=g_\nu(n)& x\in\bnd,n\cdot n_x<0.		
	\end{cases}
	\end{equation}

Problems like \eqref{rtetime} and \eqref{rtestationary} or similar equations related to radiative transfer are often studied in the framework of the so-called diffusion approximation (see \cite{mihalas,Zeldovic}). This approximation is valid when the mean free path of the photons is much smaller than the macroscopic size of the system. However, the mean free path of the photons can be small because, either the scattering mean free path or the absorption mean free path is smaller than the size of the system. The main consequence for that is that, depending on the ratio between the different mean free paths, the radiation intensity can be approximated by the Planck distribution, i.e. $ B_\nu(T) $, or it cannot be. The first case is denoted as equilibrium diffusion approximation while the second one is referred to as non-equilibrium diffusion approximation. These concepts have been extensively discussed in the physical literature on radiation (cf. \cite{mihalas,Zeldovic}). The goal of this paper is to obtain a precise mathematical characterization of these concepts, specifically to derive an accurate mathematical condition for the validity of the equilibrium diffusion approximation and to determine the regions where the equilibrium or non-equilibrium diffusion approximation holds for the specific problems \eqref{rtetime} and \eqref{rtestationary}. To this end, we will use perturbative methods and  matched asymptotic expansions in order to study different scaling limits for the scattering and absorption mean free paths.




\subsection{Scaling lengths and results}
	We study the solutions of the time dependent and stationary radiative transfer equations \eqref{rtetime} and \eqref{rtestationary} under different scaling limits and we obtain suitable problems satisfied by the limit of the solutions of the original problems. For these problems we will obtain either the equilibrium or the non-equilibrium diffusion approximation. To this end we start defining some characteristic lengths.
	
We consider a convex domain $ \Omega\subset\RR^3 $ with diameter of order $ 1 $ and such that the size of the domain is comparable in all directions of the space. Moreover, the characteristic macroscopic length $ L $ is assumed to be $ L=1 $. We remark that many of the results obtained in this paper are valid also in non-convex domain. However, in non-convex domains we should take into account also the consequences of incoming radiation into cavities, an issue that we will not consider in this paper (see \cite{jang} for more details).

We will replace the absorption coefficient $ \alpha_\nu^a(x) $ by
\begin{equation}\label{Labs}
  \frac{\alpha_\nu^a(x)}{\ell_A}
\end{equation}
and the scattering coefficient $ \alpha_\nu^s(x) $ by \begin{equation}\label{Lsca}
\frac{\alpha_\nu^s(x)}{\ell_S},
\end{equation}
where now  $ \alpha_\nu^a(x)=\mathcal{O}(1) $ and $ \alpha_\nu^s(x)=\mathcal{O}(1) $ are bounded by a constant of order one in both variables. We denote by $ \ell_A $ the absorption length and by $ \ell_S $ the scattering length. These are also the mean free paths of the absorption/emission processes and the scattering processes, respectively. In some physical applications it is convenient to assume $ \alpha_\nu^a(x) $ or $ \alpha_\nu^s(x) $ to tend to zero for large or small frequencies $ \nu $. The exact dependence of these functions on $ \nu $ will be made after. Roughly speaking, we have to assume that they have to decay not too fast in order to obtain that some integrals arising in the analysis are convergent.

In many technological applications it can be assumed that $ \alpha_\nu^s\ll \alpha_\nu^a $ (cf. \cite{Zeldovic}), however there are also applications where the scattering plays a more important role than the absorption/emission process. This is the case for example in the analysis of planetary atmospheres, see \cite{friedlander2000smoke,oxenius}.

Another important scaling length that we should consider is the Milne length, which is given by the minimum between absorption and scattering length,
\begin{equation}\label{Lmilne}
	\ell_M=\min\{\ell_A,\ell_S\}.
\end{equation}
The Milne length can be considered to be the effective mean free path of the whole radiative process. The key feature of the Milne length is that at distances of order $ \ell_M $ to the boundary the radiation intensity becomes isotropic, i.e. independent of the direction $ n\in\Ss $. Since we are interested in the diffusion approximation we assume in the rest of this paper $ \ell_M\ll L= 1 $.

Another length which plays a crucial role in the analysis of this paper is the quantity that we will denote as thermalization length which is the geometrical mean of the absorption and the Milne length
\begin{equation}\label{Ltherm}
	\ell_T=\sqrt{\ell_A\ell_M}.
\end{equation}
The thermalization length is the characteristic distance from the boundary in which the radiation intensity $ I_\nu $ approaches the Planck equilibrium distribution of the temperature. 

We now replace in \eqref{rtetime} and \eqref{rtestationary} the absorption and scattering coefficients by the expression in \eqref{Labs} and \eqref{Lsca}. The changes of the temperature take place in times of order \begin{equation*}\label{Theat}
	\tau_{\mathit{h}}=\frac{\ell_A}{\min\{\ell_T^2,1\}}\gg 1,
\end{equation*} which will be denoted as heat parameter. Therefore, in order to obtain an equation that changes in times $ t $ of order $ 1 $ we will replace $ t $ by $ \tau_{\mathit{h}}t $. Notice that, after this change of variable, the changes of times $ t $ of order $ 1 $ are associated to relevant changes of the temperature of order $ 1 $. We will use this notation throughout the paper, i.e. we will denote by $ t $ the time after the change of variable. Hence, \eqref{rtetime} writes using $ L=1 $
\begin{equation}\label{rtetimeps}
	\begin{cases}
		\frac{1}{c}\partial_tI_\nu(t,x,n)+\tau_{\mathit{h}} n\cdot \nabla_x I_\nu(t,x,n)=\frac{\alpha_\nu^a(x)\tau_{\mathit{h}}}{\ell_A}\left(B_\nu(T(t,x))-I_\nu(t,x,n)\right)\\\text{\phantom{nel mezzo del cammin di no}}+\frac{\alpha_\nu^s(x)\tau_{\mathit{h}}}{\ell_S}\left(\int_\Ss K(n,n')I_\nu(t,x,n')\;dn'-I_\nu(t,x,n)\right)& x\in\Omega,n\in\Ss,t>0\\
		\partial_tT+\partial_t \left(\int_0^\infty d\nu\int_\Ss dn \;I_\nu(t,n,x)\right)+\tau_{\mathit{h}}\Div\left(\int_0^\infty d\nu\int_\Ss dn \;nI_\nu(t,n,x)\right)=0& x\in\Omega,n\in\Ss,t>0\\
		I_\nu(0,x,n)=I_0(x,n,\nu)& x\in\Omega,n\in\Ss\\
		T(0,x)=T_0(x)& x\in\Omega\\
		I_\nu(t,n,x)=g_\nu(t,n)&x\in\bnd,n\cdot n_x<0,t>0.		
	\end{cases}
\end{equation}
We will also consider the case where the speed of light is infinite, i.e. $ c= \infty $. This approximation is justified if the characteristic time for the temperature to change is much smaller than the time required for the light to cross the domain. In this case the equation will be
\begin{equation}\label{rtetimepsc}
\hspace{-0.27cm}	\begin{cases}
		n\cdot \nabla_x I_\nu(t,x,n)=\frac{\alpha_\nu^a(x)}{\ell_A}\left(B_\nu(T(t,x))-I_\nu(t,x,n)\right)\\\text{\phantom{nel mezzo del cammin}}+\frac{\alpha_\nu^s(x)}{\ell_S}\left(\int_\Ss K(n,n')I_\nu(t,x,n')\;dn'-I_\nu(t,x,n)\right)& x\in\Omega,n\in\Ss,t>0\\
		\partial_tT+\tau_{\mathit{h}}\Div\left(\int_0^\infty d\nu\int_\Ss dn \;nI_\nu(t,n,x)\right)=0& x\in\Omega,n\in\Ss,t>0\\
		T(0,x)=T_0(x)& x\in\Omega\\
		I_\nu(t,n,x)=g_\nu(t,n)&x\in\bnd,n\cdot n_x<0,t>0.
	\end{cases}
\end{equation}
Similarly the stationary problem \eqref{rtestationary} can be written as
\begin{equation}\label{rtestationaryeps}
	\begin{cases}
		n\cdot \nabla_x I_\nu(x,n)=\frac{\alpha_\nu^a(x)}{\ell_A}\left(B_\nu(T(x))-I_\nu(x,n)\right)\\\text{\phantom{nel mezzo del cammin di}}+\frac{\alpha_\nu^s(x)}{\ell_S}\left(\int_\Ss K(n,n')I_\nu(x,n')\;dn'-I_\nu(x,n)\right)& x\in\Omega,n\in\Ss\\
		\Div\left(\int_0^\infty d\nu\int_\Ss dn \;nI_\nu(n,x)\right)=0& x\in\Omega,n\in\Ss\\
		I_\nu(n,x)=g_\nu(n)&x\in\bnd,n\cdot n_x<0.		
	\end{cases}
\end{equation}
It is important to remark that we assume $ g_\nu(t,n) $ in \eqref{rtetimeps} and \eqref{rtetimepsc} to change in times of order $ 1 $ after rescaling the time, i.e. we assume the incoming radiation $ g_\nu $ to change in the same time scale as the one for meaningful changes of the temperature.

Notice that at the first glance the time $ \tau_{\mathit{h}} $ does not seem to have units of time. However, we must take into account that since $ L=1 $, omitted in all the equations, all quantities $ \ell_A$, $\ell_S $, $ \ell_M $ and $ \ell_T $ are non-dimensional parameters that have to be understood as $ \frac{\ell_A}{L} $, $ \frac{\ell_S}{L} $, $ \frac{\ell_M}{L} $ and $ \frac{\ell_T}{L} $. In addition we recall that we have chosen a particular unit of time for which the heat capacity is $ C=1 $. Hence, all the space and time variables appearing in  \eqref{rtetimeps}- \eqref{rtestationaryeps} are non-dimensional. We will see in Sections \ref{Sec.5} to \ref{Sec.7} that the definition of the heat parameter, namely $ \tau_{\mathit{h}} $, is motivated by the behavior of the radiation intensity in the bulk and it is the order of time in which the temperature changes.

There are three characteristic lengths in \eqref{rtetimeps}- \eqref{rtestationaryeps}, namely $ \ell_A $, $ \ell_S $ and $ L=1 $, and we can consider several relative scalings between them. Since $ \ell_M\ll 1 $ in the case of the diffusion approximation, the solutions can be described by means of different boundary layers. It turns out that the relative size and the structure of these boundary layers can be characterized using the relative scaling of $ \ell_M $ (cf. \eqref{Lmilne}), $ \ell_S $ (cf. \eqref{Ltherm}) and $ L=1 $. In order to consider these different scalings in the following sections we will set for the equations \eqref{rtetimeps}, \eqref{rtetimepsc} and \eqref{rtestationaryeps} $\ell_M=\eps\ll 1 $ and we will choose $ \ell_A $, $ \ell_S $ and $ c $ as power of $ \eps $.

Notice that the incoming radiation $ g_\nu $ to the boundary of $ \Omega $ is not necessarily isotropic and in general it is different from the Planck distribution, i.e. it is not in thermal equilibrium. This implies the onset (in principle) of two nested boundary layers near the boundary where the intensity $ I_\nu $ changes its behavior. The thickness of these layers is $ \ell_M $ and $ \ell_T $ respectively. In the first layer, which we call Milne layer, the radiative intensity $ I_\nu $ becomes isotropic. In the latter, which we denote as thermalization layer, $ I_\nu $ approaches the Planck distribution for a suitable temperature that has to be determined and it is one of the unknowns of the problem. Notice moreover that since by definition $ \ell_M\leq\ell_T $, the Milne layer appears always before the thermalization layer. On the other hand if $ \ell_M $ is comparable to $ \ell_T $ both layers can coincide. It is worth to notice that beyond the thermalization layer the radiative intensity $ I_\nu $ is given by a Planck distribution. In the time dependent problem besides the formation of boundary layers we observe the formation of initial layers in which the radiation intensity becomes isotropic or the equilibrium distribution, respectively.


Table \ref{table:1} summarizes the behavior of the solution $ (T,I_\nu) $ to the equations \eqref{rtetimeps}-\eqref{rtestationaryeps} for different scaling limits yielding equilibrium or non-equilibrium diffusion approximation. Moreover, for any considered regime we observe the onset or not of Milne layers or thermalization layers. Finally, when $ \ell_T $ is of the same order of the characteristic length $ L $ the thermalization, i.e. the transition of $ I_\nu $ to the equilibrium distribution $ B_\nu(T) $, takes place in the bulk of the domain $ \Omega $.
\begin{table}[h!]\centering
\begin{tabular}{ |c|c|c|c|c|} 
	\hline
	& $ \ell_M=\ell_T\ll L $ &  $ \ell_M\ll\ell_T\ll L $& $ \ell_M\ll\ell_T= L $& $ \ell_M\ll L\ll \ell_T $ \\
	\hline
	Milne layer & Milne = & Yes & Yes & Yes \\
	\cline{1-1}\cdashline{2-2}\cline{3-5}
	\multirow{2}{7em}{Thermalization layer } & \multirow{2}{7em}{Thermalization} &\multirow{2}{.5em}{Yes} &\multirow{2}{4em}{ $ \approx  $ Bulk} & \multirow{2}{1em}{No} \\ 
	&&&&\\
	\hline
	\multirow{4}{3em}{Bulk } &  \multirow{4}{7em}{Equilibrium diffusion approximation } & \multirow{4}{7em}{Equilibrium diffusion approximation }	& \multirow{4}{7em}{Transition from equilibrium to non-equilibrium approximation } & \multirow{4}{7em}{Non-equilibrium diffusion approximation } \\ 
	&&&&\\
	&&&&\\
	&&&&\\
	\hline
\end{tabular}
\caption{Main results.}	
\label{table:1}
\end{table}
 \subsection{Revision of the literature}
	
	The problem concerning the distribution of temperature of a material interacting with electromagnetic waves is not only a relevant question in many physical applications but also it is the source of several interesting mathematical problems. The radiative transfer equation is the kinetic equation describing the interaction of photons with matter. Its derivation and its main properties are explained in \cite{Chandrasekhar, mihalas,oxenius,Rutten,Zeldovic}. In particular, the validity of the diffusion approximation and a discussion of the situations where the radiation intensity is expected to be or not to be given approximately by the Planck distribution are considered in \cite{mihalas,Zeldovic}.
	
	Starting from the seminal work of Compton \cite{compton} the interaction of matter and radiation has been widely studied both in the physical and mathematical literature. Some of the early results can be found in the paper of Milne \cite{Milne}, who considered a simplified model of monochromatic radiation depending only on one space variable.
	
	 When considering the diffusion approximation of the radiative transfer equation a boundary layer near the boundary appears in which the distribution of radiation becomes isotropic. The specific equation describing this layer involves a radiative transfer equation depending on one space variable, whose details depend on the problem under consideration. This class of problems is known in the mathematical literature as Milne problems and they have been extensively studied at least for some particular choices of $ \alpha_\nu^a $ and $ \alpha_\nu^s $.
	
	While it is difficult to find explicit solutions of the radiative transfer equation, in the case of small photon's mean free path (i.e. in the diffusion approximation) this problem reduces to an elliptic (in the stationary case) or parabolic (in the time dependent case) problem. The mathematical properties of these problems are much better understood than the properties of the non-local radiative transfer equation \eqref{RTE}. Due to this the diffusion approximation of the radiative transfer equation has been studied in great detail.
	
	Before discussing the currently available mathematical results about the diffusion approximation and the Milne problems it is worth to introduce an equation which is closely related to the radiative transfer equation \eqref{RTE}. In the absence of emission-absorption processes, i.e. when $ \alpha_\nu^a=0 $, and when $ \alpha_\nu^s $ is independent of the frequency $ \nu $ the radiative transfer equation \eqref{RTE} reduces to
	\begin{equation}\label{NTE}
		\partial_t u(t,x,n)+n\cdot \nabla_x u(t,x,n)=\alpha(x)\left(\int_\Ss K(n,n')u(t,x,n')\;dn'-u(t,x,n)\right),
	\end{equation}
where $ u=\int_0^\infty I_\nu(t,x,n)\;d\nu $. This equation is mathematically identical to the one-speed neutron transport equation. Moreover, in the stationary case the radiative transfer equation reduces to \eqref{NTE} also in the presence of absorption-emission processes if both $ \alpha^a $ and $ \alpha^s $ are independent of the frequency. The case where both absorption and scattering coefficients are independent of the frequency is usually denoted in the literature as the Grey approximation. Therefore the one-speed neutron transport equation and the radiative transfer equation for the Grey approximation are mathematically equivalent. See \cite{Davison} for more details.
As a matter of fact, the neutron transport equation, especially its diffusion approximation, was largely studied in the late 70's. The reason is that this problem is important in order to determine the critical size for neutron transport, i.e. the smallest size of the system for which the scattering eigenvalue problem has a stable solution. This is relevant in nuclear reactor engineering. For more details about this issue we refer to \cite{Davison}.

In several articles \cite{Larsen7,Larsen3,Larsen2, Larsen9,Larsen10,LarsenKeller,Larsen8} Larsen and several coauthors studied many properties of the neutron transport equation and its diffusion approximation. Moreover in \cite{Larsen6} the authors studied via asymptotic analysis the diffusion approximation of the radiative transfer equation for both absorption and scattering taking as initial and boundary value the Planck distribution. This choice of boundary data simplifies the treatment of the problem because no boundary layers or initial transport problems arise at least to the leading order.

To the best of our knowledge the first mathematically rigorous article about the diffusion approximation for the neutron transport equation is \cite{papanicolaou}. In that article the authors studied equation \eqref{NTE} under different boundary conditions including also the absorbing boundary condition that we are considering in \eqref{bnd}. In particular using probabilistic methods they studied the Milne problem arising for the boundary layers and proved the convergence of the solution of the original neutron transport equation to the solution of a diffusive problem. Moreover, the scattering kernel considered is assumed to be strictly positive, bounded and rotationally symmetric.

More recently Guo and Wu studied in a series of papers \cite{GuoLei2d,Leiunsteady,Lei3d,wuguo,annulus} both the stationary and time dependent diffusion approximation for the neutron transport equation with a constant scattering kernel and a constant scattering coefficient. They proved rigorously the convergence to such diffusion problem computing also a geometric correction for the boundary layer. Their method is based on the derivation of suitable $ L^2-L^p-L^\infty $ estimates, a method that has been extensively used in the study of kinetic equations (cf. \cite{GuoL2,GuoL22}).

The mathematical theory of the radiative transfer equation has been also extensively studied. The well-posedness and the diffusion approximation for the time dependent problem without scattering has been studied using the theory of $ m $-accretive operators in \cite{Golse3,Golse6,Erratum}. 

In a recent paper \cite{dematte2023diffusion} we developed an alternative method to derive the equilibrium diffusion approximation starting with the stationary radiative transfer equation. Specifically, in \cite{dematte2023diffusion} the Grey approximation and the case of absence of scattering are considered. The procedure developed in \cite{dematte2023diffusion} consists in reformulating the problem \eqref{rtestationaryeps} as a non-local elliptic equation for the temperature for which maximum principles techniques are applicable.

As indicated before an important class of problems which need to be studied in order to derive the boundary condition for the diffusion approximation are the Milne problems.

In the case of pure absorption, namely when $ \alpha_\nu^s=0 $, the well-posedness for the Milne problem can be found for instance in \cite{Golse4} and also in \cite{dematte2023diffusion} using different methods. In particular in \cite{Golse4} well-posedness is shown for a very large class of absorption coefficients.

In the case of pure scattering radiative transfer equation for the Grey approximation (equivalently the neutron transport equation), the well-posedness of the Milne problem has been studied in \cite{BardosSantosSentis,papanicolaou}. More recently, geometric corrections to the solution of the Milne problem have been obtained in \cite{GuoLei2d,Leiunsteady,Lei3d,wuguo,annulus}.

To our knowledge the only example of Milne problem involving both emission/absorption and scattering has been studied in \cite{Sentis1}. The case considered in this paper is the one of constant scattering kernel and constant scattering coefficient and more general absorption coefficient. The proof relies on the accretiveness of the operators used similarly to the Perron method applied to solve boundary value problem for elliptic equations.

It is finally worth to mention that also for other kinetic equations, such as for example the Boltzmann equation, the diffusion limit and hence the boundary layer equations have been studied. The equations describing the boundary layers are also often denoted in the literature by Milne problems, see for instance \cite{MilneBoltzmann4,MilneBoltzmann2,MilneBoltzmann1,MilneBoltzmann3}.\\

Besides the studies about the diffusion approximation the radiative transfer equation has been analyzed in numerous works. In recent times there has been a graving interest of the study of problems involving the radiative transfer equation in different contexts. The well-posedness of the stationary equation \eqref{rtestationary} has been considered in \cite{dematt42024compactness,jang}. The authors proved the existence of solutions to the stationary radiative transfer equation with or without scattering in the cases of constant coefficients, coefficients depending on the frequency but not on the temperature of the system and finally coefficients depending on both the frequency and the temperature of the particular form $ \alpha_\nu(T)=Q(\nu)\alpha(T) $.

Finally, the radiative transfer equation has been considered also for more complicated interactions between matter and photons. We refer to \cite{Golse1,Golse2,mihalas,Zeldovic} for problems concerning the interaction of matter with radiation in a moving fluid. For the study of interaction of electromagnetic waves with a Boltzmann gas whose molecules have different energy levels we refer to \cite{dematte,paper,oxenius,rossani}. Several authors considered problems where the heat is transported in a body by means of both radiation and conduction, we refer to \cite{masmoudi2,masmoudi1,finnland2,finnland1,Pouso,finnland3,Tiihonen1}. Finally, homogenization problems in porous and perforated domains where the heat is transported by conduction, radiation and possibly also convection are studied in \cite{Allaire1,Allaire2,Allaire3,Englishpeople}. Specifically, in \cite{Englishpeople} the authors applied the method of multiple scales to a homogenization problem describing the heat transport in a porous medium. The heat transport is assumed to be due to the conduction in the solid part of the material and due to the radiation in the gas filled cavities. 

Derivations of the scattering kernel for the radiative transfer equation taking as starting point the Maxwell equations has been also extensively studied in \cite{mischchenko}.

\subsection{Structure of the paper}
The paper is organized as follows. In Section \ref{Sec.preliminary} we will study some of the mathematical properties of the scattering operator and of the absorption-emission process appearing in the radiative transfer equation. We will then proceed to the derivation of the limit problems in the diffusion approximation under different scaling limits. In Section \ref{Sec.4} we consider the stationary diffusion approximation for the radiative transfer equation and we derive using the method of matched asymptotic expansions the new limit boundary value problem as well as the boundary layer equations. Moreover, we will see for which choice of characteristic lengths the equilibrium diffusion approximation holds and for which ones it fails. We will then proceed with the study of the time dependent diffusion approximation, for which we will use again the method of matched asymptotic expansions. In Section \ref{Sec.5} the focus is on the case of infinite speed of light (i.e. instantaneously transport of the radiation in the domain), namely on the problem \eqref{rtetimepsc}. Besides the construction of the limit problems and their classification as equilibrium and non-equilibrium diffusion approximations we will also derive the initial layer and initial-boundary layer equations. In Section \ref{Sec.6} and in Section \ref{Sec.7} we proceed similarly to Section \ref{Sec.5} studying first the time dependent diffusion approximation in the case of finite speed of light, i.e. speed of light of order one, (cf. Section \ref{Sec.6}) and later in the case where the speed of light is assumed to scale like a power law $ c=\eps^{-\kappa} $ for $ \kappa>0 $ and $ \eps=\ell_M $ (cf. Section \ref{Sec.7}).

	\section{Preliminary results}\label{Sec.preliminary}
	In this section we collect some properties of the scattering operator and absorption operator that will be used later in the analysis of the diffusion approximation.

	\subsection{Properties of the scattering operator}
	Before deriving suitable diffusion approximations according to the different values of $ \ell_M $ and $ \ell_T $ we describe some properties of the scattering kernel and of the scattering operator.
	
	We consider throughout the paper the kernel $ K\in C\left(\Ss\times\Ss\right) $ to be non-negative and 
	satisfying \\$ \int_\Ss K(n,n')dn=1 $. We also assume in the whole article that the kernel $ K $ is invariant under rotations, i.e.
	\begin{equation*}
		K(n,n')=K(\Rot n, \Rot n')\;\;\;\text{for all }n,n'\in \Ss \text{ and for any }\Rot\in SO(3).
	\end{equation*}  Moreover, for any $ n,\omega\in\Ss $ we define by $\Rot_{n,\omega}\in SO(3)$ the rotation of $ \pi $ around the bisectrix of the angle between $ n $ and $ \omega $ lying in the plane containing both vectors. This rotation satisfies $ \Rot_{n,\omega}(n)=\omega $ and $ \Rot_{n,\omega}(\omega)=n $. As shown in \cite{dematt42024compactness}, this implies that the scattering kernel $ K $ is symmetric. Notice that this is not true in two dimensions unless we assume $ K $ to be invariant also under reflections.

	We define the scattering operator as the bounded linear operator given by
	\begin{equation}\label{defH}
		\begin{split}
			H:L^\infty\left(\Ss\right)& \to L^\infty\left(\Ss\right)\\\varphi&\mapsto H[\varphi]=\int_\Ss K(\cdot,n')\varphi(n')\ dn'.
		\end{split}
	\end{equation}
	With this notation we can formulate the following Proposition which contains the most important properties of the scattering operator.
	
	\begin{prop}\label{prop.constant}
		Let $ K\in C\left(\Ss\times\Ss\right) $, invariant under rotations, non-negative and satisfying $$ \int_\Ss K(n,n')dn=1 .$$ Assume $ \varphi\in L^\infty\left(\Ss\right) $ satisfies $ H[\varphi]=\varphi $. Then
		\begin{enumerate}
			\item[(i)] $ \varphi $ is continuous,
			\item[(ii)] $ \varphi $ is constant,
			\item[(iii)] $ \text{Ran}(Id-H)=\left\{\varphi\in L^\infty(\Ss): \int_\Ss \varphi=0\right\} $.
		\end{enumerate}
	\end{prop}
The proof of Proposition \ref{prop.constant} can be found in the Appendix \ref{Appendix}. A direct consequence of Proposition \ref{prop.constant} is the following Proposition for a continuous scattering kernel $ K\in C\left(\Ss\times\Ss\times \Omega\times \RR_+\right) $ invariant under rotations for each pair $ (x,\nu) $.
\begin{prop}\label{prop.full.k}
	Let  $ K\in C\left(\Ss\times\Ss\times \Omega\times \RR_+\right) $. For any $ x,\nu\in\Omega\times\RR_+ $ we define $ K_{x,\nu}(n,n')=K(n,n',x,\nu) $. Assume that for any $ x,\nu\in\Omega\times\RR_+ $ the kernel $ K_{x,\nu} $ is invariant under rotations, non-negative and satisfies $ \int_\Ss K_{x,\nu}(n,n')dn=1 .$ Then the following holds.
	\begin{enumerate}
		\item[(i)] For any  $ x,\nu\in\Omega\times\RR_+ $ and $ n,\omega\in\Ss $ there exist finitely many $ n_1,\cdots,n_N\in\Ss $ such that \eqref{ergodicity.eq} holds for $ K_{x,\nu}$;
		\item[(ii)] if $ \varphi\in L^\infty\left(\Ss\times\Omega\times\RR_+\right) $ satisfies $ H[\varphi]=\varphi $, then $ \varphi $ is continuous and it is constant for every $ x,\nu\in\Omega\times\RR_+ $,
			\item[(iii)] $ \text{Ran}(Id-H)=\left\{\varphi(\cdot,x,\nu)\in L^\infty(\Ss): \int_\Ss \varphi(n,x,\nu)\;dn=0\right\} $ for every $ x,\nu\in\Omega\times\RR_+ $.
	\end{enumerate}
\begin{proof}
	Apply Proposition \ref{prop.constant} to the continuous kernel $ K_{x,\nu} $.
\end{proof}
\end{prop}
\begin{remark}
In the following Sections we will consider the diffusion approximation for scattering kernels $ K $ independent of $ x\in\Omega $ and $ \nu\geq 0 $. However, under the assumptions of Proposition \ref{prop.full.k} the same results would apply for more general kernels depending continuously on $ x $ and $ \nu $.
\end{remark}
\begin{remark}
The assumption of $ K $ being invariant under rotations is crucial for the validity of Proposition \ref{prop.constant} and Proposition \ref{prop.full.k}. Consider for example the following continuous function $$ k(n)=\frac{2}{3\pi}\left(\rchi_{\{|n\cdot e_3|\leq \frac{1}{4}\}}(n)+(2-4|n\cdot e_3|)\rchi_{\{\frac{1}{4}<|n\cdot e_3|< \frac{1}{2}\}}(n)\right). $$ Then the kernel $ K(n,n')=k(n)\rchi_{\Ss}(n') $ is continuous in both variables, is non-negative and satisfies $$ \int_\Ss K(n,n')\ dn=\int_\Ss k(n)\ dn=1.$$ However, $ K $ is not invariant under rotations. This kernel describes the scattering properties of a non-isotropic medium. It is easy to see that in this case $ H[c](n)=ck(n) $, for $ c\in\RR $. Hence, the constant functions are not a solution to $ H[\varphi]=\varphi $. Actually, all solutions of $ H[\varphi]=\varphi $ satisfy $ \varphi(n)=k(n)\int_\Ss \varphi(n')\;dn' $ and have hence the form $ \varphi=\lambda k $ where $ \lambda\in\RR $ is an arbitrary constant. Therefore the subspace of eigenvectors of $ H $ with eigenvalue $ 1 $ is one-dimensional. 
\end{remark}
\begin{remark}
As we noticed above, in two dimensions the invariance under rotations of $ K $ does not imply directly its symmetry under reflections. However, it is still possible to show that the only eigenfunctions of $ H $ with eigenvalue $ 1$ are the constants. To check this we recall the well-known fact that the one-dimensional sphere $ \mathbb{S}^1 $ can be parameterized by $ \theta\in[0,2\pi) $. Moreover, we can assume without loss of generality that any scattering kernel $ K $ invariant under rotations has the form $ K(n,n')=K(\theta(n)-\theta(n')) $. 
Let now $ f\in L^\infty(\mathbb{S}^1) $ an eigenfunction with eigenvalue $ 1 $ for $ H $. We then see 
$$ \int_0^{2\pi} K(\theta-\varphi)f(\varphi)\;d\varphi=f(\theta). $$ This equation can be solved using Fourier series. We hence obtain the following identity for the Fourier coefficients
\begin{equation}\label{1-dim}
\hat{f}(n)\left(1-2\pi \hat{K}(n)\right)=0.
\end{equation}
For $ n=0 $ we have $ \hat{K}(0)=\frac{1}{2\pi}\int_0^{2\pi}K(\theta) \;d\theta=\frac{1}{2\pi}$. On the other hand, we obtain for $ n\ne 0 $
$$ \left|\hat{K}(n)\right|<\frac{1}{2\pi}\int_0^{2\pi}K(\theta)\;d\theta=\frac{1}{2\pi}. $$
Therefore, the identity \eqref{1-dim} is satisfied if and only if $ \hat{f}(n)=0 $ for all $ n\ne 0 $. This implies that $ f $ is constant. 
\end{remark}
	\subsection{Relation between the temperature and the radiation intensity}\label{subsect.2.2}
We derive here an identity that relates temperature and radiation intensity and that will be repeatedly used in the stationary problem, for instance in the stationary boundary layer equations.

Using the identity $ \Div\left(\int_0^\infty d\nu\int_\Ss dn\;nI_\nu(x,n)\right)=\int_0^\infty d\nu\int_\Ss dn\;n\cdot \nabla_xI_\nu(x,n) $ and plugging the first equation of \eqref{rtestationary} into the second one we see that we have
\begin{equation}\label{stationary.temp}
	\int_0^\infty d\nu\int_\Ss dn\; \alpha_\nu^a(x)\left(B_\nu(T(x))-I_\nu(x,n)\right)	=0,
\end{equation}
where we used also that the integral over the sphere $ \Ss $ of the scattering term is $ 0 $ due to the symmetry of the kernel $ K $. With this identity we can recover the value of the temperature given the radiation intensity. Let us define by $ F:\RR_+\times \Omega\to \RR_+ $ the following function
\begin{equation}\label{F}
	F(T,x)=\int_0^\infty \alpha_\nu^a(x)B_\nu(T)\;d\nu.
\end{equation}
Since $ B_\nu $ is monotone in $ T $, the function $ F(\cdot, x) $ is invertible. Hence, \eqref{stationary.temp} implies that
\begin{equation}\label{def.T}
	T(x)=F^{-1}\left(\left(\int_0^\infty d\nu \fint_\Ss dn\; \alpha_\nu^a(x)I_\nu(x,n)\right),x\right),
\end{equation}
where $ F^{-1} $ is the inverse with respect to the first variable, i.e. $ F(T,x)=\xi $ implies $ T=F^{-1}(\xi,x) $.
Equations \eqref{stationary.temp} and \eqref{def.T} will appear often in the following sections, in particular in the study of the boundary layers. 
\section{The stationary diffusion approximation: different scales}\label{Sec.4}
We first study the stationary diffusion regime for different scalings. We consider \eqref{rtestationaryeps} for $ \alpha_\nu^a $ and $ \alpha_\nu^s $ strictly positive and bounded. Moreover, in the diffusion regime we have $ \ell_M\ll 1 $. Hence, in \eqref{rtestationaryeps} we assume $ \ell_M=\min\{\ell_A,\ell_S\}=\eps $. Moreover, we impose $ \ell_A=\eps^{-\beta} $ and $ \ell_S=\eps^{-\gamma} $, for suitable choices of $ \gamma,\beta \geq -1 $ with $ \min\{\gamma, \beta\}=-1 $. Notice that at least one of $ \beta $ and $ \gamma $ is negative. This choice of $ \ell_A $ and $ \ell_S $ as an inverse power law of $ \eps>0 $ for $ \beta, \gamma\geq -1 $ will be convenient in order to make the computations simpler in the following subsections. Under these assumptions we rewrite equation \eqref{rtestationaryeps} as 
\begin{equation}\label{rtestationaryeps1}
	\begin{cases}
		n\cdot \nabla_x I_\nu(x,n)=\eps^\beta\alpha_\nu^a(x)\left(B_\nu(T(x))-I_\nu(x,n)\right)\\\text{\phantom{nel mezzo del cammin }}+\eps^\gamma\alpha_\nu^s(x)\left(\int_\Ss K(n,n')I_\nu(x,n')\;dn'-I_\nu(x,n)\right)&x\in\Omega,n\in\Ss\\
		\Div\left(\int_0^\infty d\nu\int_\Ss dn \;nI_\nu(n,x)\right)=0&x\in\Omega,n\in\Ss\\
		I_\nu(n,x)=g_\nu(n)&x\in\bnd,n\cdot n_x<0.		
	\end{cases}
\end{equation}
Moreover, we assume the scattering kernel $ K\in C(\Ss\times\Ss) $ to be invariant under rotations, non-negative and with $ \int_\Ss K(n,n')dn=1 $. We consider also $ \Omega\subset \RR^3 $ to be a bounded convex domain with $ C^1 $-boundary. For $ x\in\bnd $ we denote by $ n_x\in\Ss $ the outer normal to the boundary at $ x $. 

Before describing in details the limit diffusion problems for the different choices of scaling parameters we shortly explain how we will use the method of matched asymptotic expansions to derive the limit problems for each case. In order to find the limit problem valid in the bulk, the so-called outer problem, we expand the radiation intensity as
\begin{equation}\label{expansion}
	I_\nu(x,n)=\phi_0(x,n,\nu)+\eps^\delta \phi_1(x,n,\nu)+\eps\phi_2(x,n,\nu)+\eps^{\delta+1}\phi_3(x,n,\nu)+...
\end{equation}
for a suitable $ \delta>0 $ depending on the choice of the scaling parameters. We plug  expansion \eqref{expansion} into the boundary value problem \eqref{rtestationaryeps} and we compare all terms of the same order of magnitude. In this way we will obtain different diffusive equations solved by $ \phi_0 $ in the interior of $ \Omega $ that will yield the leading order of the radiation intensity $ I_\nu $. 

However, to solve the resulting equation for $ \phi_0 $ we need some boundary condition whose derivation requires to analyze boundary layer equations for \eqref{rtestationaryeps1}. The resulting boundary layer problems are related to the description of the radiation intensity in the regions close to the boundary. The thickness of these layers is given by the Milne length and the thermalization length. Therefore, we will rescale the space variable according to $ \ell_M $ and to $ \ell_T $ and we will analyze the resulting one-dimensional problems. 

The matching between the outer and the inner solutions will provide the boundary condition for the equation satisfied in the bulk.
\subsection{Case 1.1: $ \ell_M=\ell_T\ll\ell_S $ and $ L=1 $. Equilibrium approximation}\label{subsec.1}
Since we set $ \ell_M=\eps\ll 1 $, the case $ \ell_M=\ell_T\ll\ell_S $ arises when $ \ell_A=\eps $ (i.e. $ \beta=-1 $) and $ \ell_S=\eps^{-\gamma} $ for $ \gamma>-1 $. Notice that in this case $ \ell_S $ could be small, namely $ \ell_S\ll L=1 $, but also large, e.g. if $ \gamma>0 $.

In order to find the outer problem we choose $ \delta=|\gamma| $ and we substitute \eqref{expansion} into the first equation in \eqref{rtestationaryeps1} and we identify all terms with the same power of $ \eps $, i.e. $ \eps^{-1} $, $ \eps^{|\gamma|-1} $ (if $ 0<|\gamma|<1 $) and $ \eps^0 $.
The terms of order $ \eps^{-1} $ give
\begin{equation*}\label{epsbeta1}
	0=\alpha_\nu^a(x)(B_\nu(T(x))-\phi_0(x,n,\nu)).
\end{equation*} 
Hence, the leading order satisfies $ \phi_0(x,n,\nu)=B_\nu(T(x)) $, where $ B_\nu $ is the Planck distribution which is independent of $ n\in\Ss $. This corresponds to the diffusion equilibrium approximation, since in the interior the radiation intensity is at the leading order the equilibrium Planck distribution.

The terms of order $ \eps^{|\gamma|-1} $ imply $ \phi_1=0 $. Indeed, if  $ \gamma\ne -\frac{1}{2} $, there is only one term of order $ \eps^{|\gamma|-1} $ which is\linebreak$ 0=-\alpha_\nu^a(x)\phi_1 $. On the other hand, if $ \gamma=-\frac{1}{2} $ using $  \int_\Ss K(n,n')\phi_0(x,\nu)dn'-\phi_0(x,\nu)=0 $ we obtain the same result, since $ \phi_0 $ is independent of $ n\in\Ss $. 

Finally, we compare all terms of order $ \eps^0 $. In this case we have 
\begin{equation*}\label{epszero1}
	n\cdot \nabla_x B_\nu(T(x))=-\alpha_\nu^a(x)\phi_2(x,n,\nu),
\end{equation*} 
where in the case $ \gamma=0 $ we used again that $ \left(H-Id\right)B_\nu(T)=0 $. 

Therefore we obtain the following expansion for $ I_\nu $\begin{equation}\label{case1expansion2}
	I_\nu(x,n)=B_\nu(T(x))-\eps\frac{1}{\alpha_\nu(x)}n\cdot\nabla_x B_\nu(T(x))+\cdots,
\end{equation} where $ T(x) $ is a function which is at this stage still unknown.

We now plug  \eqref{case1expansion2} into the second equation of \eqref{rtestationaryeps1}. The term of order $ \eps^0 $ cancels out because $ B_\nu(T) $ is isotropic, hence
$$ \Div\left(\int_0^\infty d\nu\int_{\Ss} dn\;n B_\nu(T(x))\right)=0.$$
We find that the leading term is the one of order $ \eps^1 $ and we obtain
\begin{equation*}\label{case1divstat}
	\Div \left(\int_0^\infty d\nu\frac{1}{\alpha_\nu(x)}\left(\int_\Ss dn\; n\otimes n\right)\nabla_x B_\nu(T(x))\right)=0.
\end{equation*}  
Finally, using that $ \int_\Ss n\otimes n\;dn=\frac{4\pi}{3}Id $ we conclude that the limit problem solved at the interior by $ T $ is 
\begin{equation}\label{interior1stationary}
	\Div\left(\int_0^\infty \frac{\nabla_x B_\nu(T(x))}{\alpha_\nu(x)}d\nu\right)=0.
\end{equation}
In order to obtain the behavior of $ I_\nu $ close to the boundary $ \bnd $ we now derive a boundary value problem that can be written in a single variable. This boundary layer equation is known in the literature as Milne problem. The matching of the solution of the Milne problem with the outer solution will provide the boundary value for the equation \eqref{interior1stationary} solved by the temperature $ T $.

We take $ p\in\bnd $. Assuming that near the boundary the radiation intensity and the temperature only depend on the distance to the boundary, we can further assume that they depend only on the distance to the boundary in direction $ n_p $. This is possible due to the smallness of the thickness of the boundary layer and the continuity of $ \alpha $. We hence define for $ x \in \Omega $ in a neighborhood of $ p $ the new scalar rescaled variable\begin{equation}\label{scalingMilne}
	y=-\frac{x-p}{\eps}\cdot n_p.
\end{equation} We recall that $ -(x-p)\cdot n_p $ is non-negative, since $ x-p $ points in the interior of the domain, and it is exactly the length of the cathetus with endpoint $ p $ of the triangle having as hypotenuse $ x-p $ (cf. Figure \ref{cathetus}).
\begin{figure}[H]\centering
	\includegraphics[height=4cm]{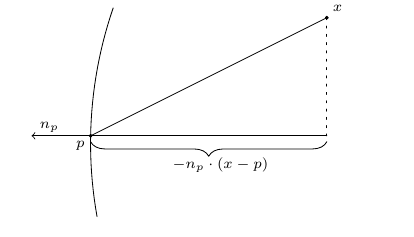}\caption{Representation of the change of variables.}\label{cathetus}
\end{figure}
Defining $ \Rot_p(x)={\text{Rot}}_p (x-p) $ as a rigid motion mapping $ p $ to zero with $ {\text{Rot}}_p(n_p)=-e_1 $ we see that we can also write $ y $ as the first component of $ y_1=\Rot_p\left(\frac{x}{\eps}\right)_1 $. Hence, as $ \eps\to 0 $ we obtain that both the absorption and scattering coefficients satisfy $ \alpha^j_\nu(x)=\alpha^j_\nu\left(\eps{\text{Rot}}_p(x)+p\right)\to \alpha^j_\nu(p) $, $ j\in\{a,s\} $.

We can now write the one-dimensional problem obtained by this new scaling and by the limit $ \eps\to 0$. Since $ \eps^{\gamma+1}\to 0 $ as $ \eps\to 0 $ the scattering term is negligible and we obtain for any $ p\in\bnd $
\begin{equation}\label{Milnecase1}
	\begin{cases}
	-(n\cdot n_p)\partial_yI_\nu(y,n;p)=\alpha_\nu^a(p)(B_\nu(T(y,p))-I_\nu(y,n;p))& y>0\;,n\in\Ss\\
	\Div\left(\int_0^\infty d\nu\int_\Ss dn\; (n\cdot n_p) I_\nu(y,n;p)\right)=0&y>0,\;n\in\Ss\\
	I_\nu(0,n;p)=g_\nu(n)& n\cdot n_p<0.
	\end{cases}
\end{equation}
The Milne equation \eqref{Milnecase1} is the equation describing the boundary layer for the diffusion approximation. In the pure absorption case the Milne problem was rigorously studied in \cite{Golse4}. The well-posedness of \eqref{Milnecase1} is shown there for constant absorption coefficients and also for coefficients depending only on the frequency $ \nu $, as well as for coefficients depending on both frequency and temperature of the form $ \alpha^a_\nu(p)=Q(\nu)\alpha(T(p)). $ Moreover, also the asymptotic behavior of $ I_\nu $ at infinity has been computed in this paper. It is indeed shown in \cite{Golse4} that as $ y\to \infty $ the solution of the Milne problem converges to the Planck distribution, i.e. $$ \lim\limits_{y\to\infty} I_\nu(y,n;p)=I^\infty_\nu(p)=B_\nu(T_\infty(p)), $$ for some $ T_\infty(p) $ depending only on $ g_\nu $ and $ p $. Notice that $ I_\nu^\infty(p) $ is independent of $ n\in\Ss $.

Moreover, since in this case the thermalization length and the Milne length are the same this is the only boundary layer appearing. The radiation intensity $ I_\nu $ becomes simultaneously isotropic and at equilibrium $ B_\nu(T) $ in the same length scale. This gives a matching condition for the temperature that has to be used as boundary condition for the new limit problem. In particular, the temperature and the radiation intensity solving the Milne problem \eqref{Milnecase1} are related by equation \eqref{stationary.temp}. In particular
\begin{equation}\label{Tinfty}
	 T_\infty(p)=\lim\limits_{y\to\infty}F^{-1}\left(\left(\int_0^\infty d\nu \alpha_\nu^a(p) I_\nu^\infty(p)\right),p\right),
\end{equation}  
where $ F $ is defined in \eqref{F} and $ g\mapsto I^\infty_\nu(p) $ is a functional that determines the limit intensity for each boundary point $ p\in\bnd $.

Summarizing, the limit problem for the stationary radiative transfer equation \eqref{rtestationaryeps} in the case $ \ell_M=\nolinebreak\ell_T\ll\nolinebreak\ell_S $ is given by the following boundary value problem 
\begin{equation*}\label{equilibtriumstationary1}
	\begin{cases}
	\Div\left(\int_0^\infty \frac{\nabla_x B_\nu(T(x))}{\alpha_\nu(x)}d\nu\right)=0& x\in\Omega\\
	T(p)=T_\infty(p)& p\in\bnd,
	\end{cases}
\end{equation*}
where $ T_\infty(p)$ is given by \eqref{Tinfty}.

\subsection{Case 1.2: $ \ell_M=\ell_T=\ell_S\ll L $. Equilibrium approximation}\label{subsec.2}

Due to the definitions $ \ell_M=\eps\ll 1 $ and $ \ell_T=\sqrt{\ell_A \ell_M} $ we have $ \ell_M=\ell_S=\ell_T=\ell_A=\eps $ in \eqref{rtestationaryeps}, i.e. $ \beta=\gamma=-1 $ in \eqref{rtestationaryeps1}.

We consider the expansion \eqref{expansion} for $ \delta=1 $, or equivalently without terms of order $ \eps^\delta $. We plug  \eqref{expansion} into \eqref{rtestationaryeps1} and we compare all terms of the same power of $ \eps $, namely $ \eps^{-1} $ and $ \eps^0 $. 
The term of order $ \eps^{-1} $ yields 
\begin{equation}\label{epsminus12}
0=\alpha_\nu^a(x)(B_\nu(T(x))-\phi_0(x,n,\nu))+\alpha_\nu^s(x)\left(\int_\Ss K(n,n') \phi_0(x,n',\nu)dn'-\phi_0(x,n,\nu)\right).
\end{equation} 
Notice that $ \phi_0(x,n,\nu)=B_\nu(T(x)) $ is a solution to \eqref{epsminus12}. This follows from Proposition \ref{prop.constant} and the isotropy of $ B_\nu(T) $. We show now that the solution to \eqref{epsminus12} is unique. 

To this end for every $ x\in\RR^3 $ and $ \nu>0 $ we define $ 0<\theta_{\nu,x}=\frac{\alpha_\nu^s(x)}{\alpha_\nu^a(x)+\alpha_\nu^s(x)}<1 $. Moreover, we define also the following operator which maps for every fixed $ x,\nu $ non-negative continuous functions to non-negative continuous functions and given by \begin{equation}\label{A}
	 A_{\nu,x}[\varphi](n)=\theta_{\nu,x}\int_\Ss K(n,n')\varphi(n')\;dn'. 
\end{equation} Then equation \eqref{epsminus12} can be rewritten as
\begin{equation}\label{epsminus1 2}
\phi_0(x,n,\nu)=A_{\nu,x}[\phi_0](x,n,\nu)+\frac{\alpha_\nu^a(x)}{\alpha_\nu^a(x)+\alpha_\nu^s(x)}B_\nu(T(x)).
\end{equation} 
Since the maps $ \phi_0\mapsto A_{\nu,x}(\phi_0) $ is a linear contraction, the Banach fixed-point theorem implies that \eqref{epsminus1 2} has a unique solution for every $ T(x)\in\RR_+ $. Hence, $ \phi_0=B_\nu(T) $. Therefore, also in this case we recover the equilibrium diffusion approximation.

We turn now to the terms of order $\eps^0 $. In this case we have 
\begin{equation*}\label{epszero2}
n\cdot \nabla_x B_\nu(T(x))=-\alpha_\nu^a(x)\phi_1(x,n,\nu)-\alpha_\nu^s(x)\left(\int_\Ss K(n,n')\phi_1(x,n',\nu)\;dn'-\phi_1(x,n,\nu)\right).
\end{equation*} 
Then, using the operator $ A_{\nu,x} $ defined as in \eqref{A}, we can rewrite this equation as
\begin{equation}\label{epszero2.2}
-\frac{1}{\alpha_\nu^a(x)+\alpha_\nu^s(x)}n\cdot \nabla_x B_\nu(T(x))=\left(Id-A_{\nu,x}\right)\phi_1(x,n,\nu).
\end{equation}
The same argument as for the term of order $ \eps^{-1} $ holds also in this case and Banach fixed-point theorem ensures the existence of a unique solution to \eqref{epszero2.2} given by 
\begin{equation*}\label{phi1.2}
	\phi_1(x,n,\nu)=-\frac{1}{\alpha_\nu^a(x)+\alpha_\nu^s(x)}\left(Id-A_{\nu,x}\right)^{-1}(n)\cdot \nabla_x B_\nu(T(x)),	
\end{equation*}
where for any $ x,\nu $ we used the notation $$ \left(Id-A_{\nu,x}\right)^{-1}(n)=\begin{pmatrix}
	\left(Id-A_{\nu,x}\right)^{-1}(n_1)\\\left(Id-A_{\nu,x}\right)^{-1}(n_2)\\\left(Id-A_{\nu,x}\right)^{-1}(n_3)
\end{pmatrix}, $$
which is well-defined due to the action of the linear operator $ A_{\nu,x} $ only on the variable $ n\in\Ss $.

Hence, we obtain the following expansion
\begin{equation}\label{case2expansion2}
I_\nu(x,n)=B_\nu(T(x))-\eps\frac{1}{\alpha_\nu^a(x)+\alpha_\nu^s(x)}\left(Id-A_{\nu,x}\right)^{-1}(n)\cdot \nabla_x B_\nu(T(x))+\eps^2\phi_2+\cdots
\end{equation}
Plugging \eqref{case2expansion2} into the second equation of \eqref{rtestationaryeps} and using that the Planck distribution is isotropic, we obtain the following limit problem solved in the domain $ \Omega $ that yields the the temperature $ T(x) $ to the leading order
\begin{equation}\label{interior2stationary}
\begin{split}
0=&\Div \left(\int_0^\infty d\nu\int_\Ss dn\; n\frac{1}{\alpha_\nu^a(x)+\alpha_\nu^s(x)}\left(Id-A_{\nu,x}\right)^{-1}(n)\cdot \nabla_x B_\nu(T(x))\right)\\=&\Div \left(\int_0^\infty d\nu\frac{1}{\alpha_\nu^a(x)+\alpha_\nu^s(x)}\left(\int_\Ss dn\; n\otimes \left(Id-A_{\nu,x}\right)^{-1}(n)\right)\nabla_x B_\nu(T(x))\right).
\end{split}
\end{equation}

The behavior of $ I_\nu $ close to the boundary $ \bnd $ is given again by a boundary layer equation which can be written in one variable. The derivation of the Milne problem for this case follows exactly the same steps as Subsection \ref{subsec.1} under the scaling \eqref{scalingMilne}. In this case both emission and scattering terms appear, since they are of the same order. Hence, for every $ p\in\bnd $ the Milne problem is given by 
\begin{equation}\label{Milnecase2}
\begin{cases}
-(n\cdot n_p)\partial_yI_\nu(y,n,p)=\alpha_\nu^a(p)(B_\nu(T(y,p))-I_\nu(y,n,p))\\\;\;\;\;\;\;\;\;\;\;\;\;\;\;\;\;\;\;\;\;\;\;\;\;\;\;\;\;\;\;\;\;\;\;\;\;\;\;\;+\alpha_\nu^s(p)\left(\int_\Ss K(n,n')I_\nu(y,n',p)\;dn'-I_\nu(y,n,p)\right)& y>0\;,n\in\Ss\\
\Div\left(\int_0^\infty d\nu\int_\Ss dn\; (n\cdot n_p) I_\nu(y,n,p)\right)=0&y>0,\;n\in\Ss\\
I_\nu(0,n,p)=g_\nu(n)& n\cdot n_p<0.
\end{cases}
\end{equation}
The mathematical properties of the Milne problem for both absorption and scattering processes have been considered in \cite{Sentis1}. Although the results provided in \cite{Sentis1} have been obtained only for the case of constant scattering kernel and constant scattering coefficient, the arguments there suggest that for more general choices of $ K $ and $ \alpha_\nu^s $ the solution $ I_\nu $ of \eqref{Milnecase2} converges to the Planck equilibrium distribution as $ y\to\infty $.

Notice that in this case, the thermalization length and the Milne length are the same, hence the boundary layers coincide.
Matching inner and outer solutions we obtain the following  boundary condition for equation \eqref{interior2stationary}
\begin{equation}\label{bnd2}
T_\infty(p)=\lim\limits_{y\to\infty}F^{-1}\left(\left(\int_0^\infty d\nu\fint_\Ss dn\; \alpha_\nu^a(p)I_\nu(y,n,p)\right),p\right),	
\end{equation}
with $ F $ as in \eqref{F}. Indeed as we have seen in Subsection \ref{subsect.2.2}, the temperature $ T $ and the radiation energy $ I_\nu $ satisfying the Milne problem \eqref{Milnecase2} are related by the identity \eqref{stationary.temp}.

Summarizing, the limit problem for the stationary radiative transfer equation \eqref{rtestationaryeps} in the case $ \ell_M=\ell_T=\ell_S $ is given by the following boundary value problem 
\begin{equation*}\label{equilibtriumstationary2}
\begin{cases}
\Div \left(\int_0^\infty \frac{d\nu}{\alpha_\nu^a(x)+\alpha_\nu^s(x)}\left(\int_\Ss dn\; n\otimes \left(Id-A_{\nu,x}\right)^{-1}(n)\right)\nabla_x B_\nu(T(x))\right)=0& x\in\Omega\\
T(p)=T_\infty(p)& p\in\bnd,
\end{cases}
\end{equation*}
where $ T_\infty $ is defined as in \eqref{bnd2} for the solution $ I_\nu(y,n,p) $ to the Milne problem \eqref{Milnecase2}.

\subsection{Case 2: $ \ell_M\ll\ell_T\ll L $. Equilibrium approximation} \label{Subsec.3}
The assumption $ \ell_T=\sqrt{\ell_M \ell_A}\ll \ell_M $ implies $ \ell_A>\ell_M $ and hence $ \eps=\ell_M=\ell_S $. We thus consider $ \ell_A=\eps^{-\beta} $ for $ \beta>-1 $. Moreover, since $ \ell_T=\eps^{\frac{1-\beta}{2}}\ll L=1 $ we restrict to the case $ \ell_A=\eps^{-\beta} $ for $ \beta\in(-1,1) $.
 
Since $ \ell_M=\ell_S\ll \ell_A $ the scattering process has a greater effect than the absorption-emission process. We expect hence the Milne problem to depend exclusively on the scattering process. In the bulk we expect also the scattering process to be present in the diffusive equation derived for the limit problem, but we will also show that at the interior the leading order of the radiation intensity is still the Planck distribution. Thus, we are again in the case of the equilibrium diffusion approximation. In this case the thermalization length is much larger than the Milne length but it is also still much smaller than the characteristic length of the domain. A second boundary layer, the so-called thermalization layer, will therefore appear. The equation describing this new layer will depend on both absorption-emission and scattering processes. Moreover, while the radiative energy becomes isotropic in the Milne layer, in the thermalization layer $ I_\nu $ will approach to the Planck distribution.

We use again the expansion \eqref{expansion} for the radiation intensity with $ \delta=|\beta| $
and we plug  it into the first equation in \eqref{rtestationaryeps}. We proceed as usual with the identification of the terms with the same power of $ \eps $, namely $ \eps^{-1} $, $ \eps^{|\beta|-1} $, $ \eps^\beta $ and $ \eps^0 $. 

Using the notation of \eqref{defH} the terms of order $ \eps^{-1} $ gives
\begin{equation*}\label{epsminus1.3}
	H[\phi_0(x,\cdot,\nu)](n)=\phi_0(x,n,\nu).
\end{equation*}
Proposition \ref{prop.constant} implies hence that $ \phi_0 $ is independent of $ n\in\Ss $ and hence $ \phi_0=\phi_0(x,\nu) $.

Next we consider for $ \beta<0 $ and $ \beta\ne -\frac{1}{2} $ the terms of power $ \eps^\beta $ and $ \eps^{|\beta|-1} $. The first case gives 
\begin{equation*}\label{epsbeta.3}
B_\nu(T(x))=\phi_0(x,\nu),
\end{equation*}
while the second case implies
\begin{equation}\label{epsbeta1.3}
H[\phi_1(x,\cdot,\nu)](n)=\phi_1(x,n,\nu).
\end{equation}
Hence, as for $ \phi_0 $ we conclude that $ \phi_1=\phi_1(x,\nu) $ is independent of $ n\in\Ss $. 

For $ \beta=-\frac{1}{2} $ the terms of order $ \eps^{-\frac{1}{2}} $ give 
\begin{equation}\label{epsbeta2.3}
0=\alpha_\nu^a(x)\left(B_\nu(T(x))-\phi_0(x,n)\right)+\alpha_\nu^s(x)\left(H[\phi_1](x,n,\nu)-\phi_1(x,n,\nu)\right).
\end{equation}
Using the invariance under rotations of $ K(n,n') $, the property $ \int_\Ss K(n,n')\;dn=1 $ and the isotropy of both $ B_\nu(T) $ and $ \phi_0 $, the integration over the whole $ \Ss $ of equation \eqref{epsbeta2.3} gives
\begin{equation}\label{epsbeta3.3}
B_\nu(T(x))=\phi_0(x,\nu).
\end{equation}
Moreover, plugging \eqref{epsbeta3.3} into \eqref{epsbeta2.3} we obtain again the equation \eqref{epsbeta1.3} and hence we conclude that $   \phi_1=\nolinebreak\phi_1(x,\nu) $.

We move now to $ \beta>0 $ and to the terms of order $ \eps^{\beta-1} $. We obtain hence once again equation \eqref{epsbeta1.3} and thus $   \phi_1=\phi_1(x,\nu) $.

We consider now the terms of order $ \eps^0 $. If $ \beta<0 $ these terms give 
\begin{equation}\label{eps0.3}
n\cdot \nabla_x B_\nu(T(x))=-\alpha_\nu^a(x)\phi_1(x,\nu)+\alpha_\nu^s(x)\left(H[\phi_2](x,n,\nu)-\phi_2(x,n,\nu)\right).
\end{equation}
Integrating equation \eqref{eps0.3} over the whole $ \Ss $ and using the isotropy of $ B_\nu(T) $ and $ \phi_1 $ we conclude $ \phi_1=0 $ and
\begin{equation}\label{phi2.3}
n\cdot \nabla_x B_\nu(T(x))=\alpha_\nu^s(x)\left(H[\phi_2](x,n,\nu)-\phi_2(x,n,\nu)\right).
\end{equation}

If $ \beta=0 $ the terms of order $ \eps^0 $ give equation
\begin{equation*}\label{eps01.3}
n\cdot \nabla_x \phi_0(x,\nu)=\alpha_\nu^a(x)\left(B_\nu(T(x))-\phi_0(x,\nu)\right)+\alpha_\nu^s(x)\left(H[\phi_2](x,n,\nu)-\phi_2(x,n,\nu)\right)
\end{equation*}
which integrated over $ \Ss $ implies, due to the isotropy of $ \phi_0 $, as for \eqref{epsbeta2.3},
$$ \phi_0=B_\nu(T),$$ which implies also in this case equation \eqref{phi2.3}. 

Finally if $ \beta>0 $ the terms of order $ \eps^0 $ give
\begin{equation*}\label{eps02.3}
n\cdot \nabla_x \phi_0(x,\nu)=\alpha_\nu^s(x)\left(H[\phi_2](x,n,\nu)-\phi_2(x,n,\nu)\right),
\end{equation*}
while the terms of order $ \eps^\beta $ give
\begin{equation}\label{eps03.3}
n\cdot \nabla_x \phi_1(x,\nu)=\alpha_\nu^a(x)\left(B_\nu(T(x))-\phi_0(x,\nu)\right)+\alpha_\nu^s(x)\left(H[\phi_3](x,n,\nu)-\phi_3(x,n,\nu)\right).
\end{equation}
Integrating \eqref{eps03.3} over the whole $ \Ss $ and using the isotropy of $ \phi_1 $, $ \phi_0 $ and $ B_\nu(T) $ we conclude once again $ \phi_0=B_\nu(T) $ and thus \eqref{phi2.3}.

Hence, for all $ \beta\in(-1,1) $ the identification of all terms of power $ \eps^{-1} $, $ \eps^{|\beta|-1} $, $ \eps^\beta $ and $ \eps^0 $ gives $ \phi_0=B_\nu(T) $, $ \phi_1=\phi_1(x,\nu) $ and 
\begin{equation}\label{phi2equiv.3}
-\frac{1}{\alpha_\nu^s(x)}n\cdot \nabla_x B_\nu(T(x))=(Id-H)[\phi_2(x,\cdot,\nu)](n).
\end{equation}
 We now study the equation \eqref{phi2equiv.3}. As we know from Proposition \ref{prop.constant} the kernel of the operator $ (Id-H) $ is given by the constant functions and its range are all functions with zero mean integral, i.e. $ \text{Ran}(Id-H)=\{\varphi\in L^\infty(\Ss): \int_\Ss \varphi=0\} $. Hence, the following linear operator is bijective $$ (Id-H):\bigslant{L^\infty(\Ss)}{\mathcal{N}(Id-H)}\to  \text{Ran}(Id-H),$$ where $ \bigslant{L^\infty(\Ss)}{\mathcal{N}(Id-H)}$ denotes the quotient space. Let $ e_i\in\RR^3 $ be the unit vector, we consider the equation \begin{equation}\label{ei}
 n\cdot e_i=(Id-H)\varphi(n).
 \end{equation}
 Since $ n\cdot e_i\in \text{Ran}(Id-H) $, for any $ c\in \RR $ the function $ \varphi(n)=(Id-H)^{-1}(n\cdot e_i)+c $ is a solution to \eqref{ei}. Therefore, using the notation \begin{equation*}\label{Id-H.inv}
 	\left(Id-H\right)^{-1}(n)=\begin{pmatrix}
 	\left(Id-H\right)^{-1}(n\cdot e_1)\\\left(Id-H\right)^{-1}(n\cdot e_2)\\\left(Id-H\right)^{-1}(n\cdot e_3)
 	\end{pmatrix}
 \end{equation*}
  and using the linearity of $ (Id-H) $ we see that $ \phi_2 $ is given by
  \begin{equation}\label{phi2.final.3}
  	\phi_2(x,n,\nu)=-\frac{1}{\alpha_\nu^s(x)}(Id-H)^{-1}(n)\cdot \nabla_x B_\nu(T(x))+c(x,\nu)
  \end{equation}
  where $ c(x,\nu) $ is independent of $ n\in\Ss $. The isotropic function $ c(x,\nu) $ does not contribute in the divergence free condition of \eqref{rtestationaryeps}, therefore we will not compute the exact value of $ c(x,\nu) $. Equation \eqref{phi2.final.3} implies that the first three terms in the expansion of $ I_\nu $ are given for all $ \beta\in(-1,1) $ by
  \begin{equation*}\label{case3expansion2}
  I_\nu(x,n)=B_\nu(T(x))+\eps^{|\beta|} \phi_1(x,\nu)-\frac{\eps}{\alpha_\nu^s(x)}(Id-H)^{-1}(n)\cdot \nabla_x B_\nu(T(x))+\eps c(x,\nu)+\eps^{|\beta|+1}\phi_3+\cdots
  \end{equation*}
  The divergence free condition in \eqref{rtestationaryeps} implies in the same manner as in the derivation of \eqref{interior2stationary} the following equation, which yields the limit problem in the interior of the domain $ \Omega $
  \begin{equation}\label{interior3stationary}
  \begin{split}
  \Div \left(\int_0^\infty d\nu\frac{1}{\alpha_\nu^s(x)}\left(\int_\Ss dn\; n\otimes \left(Id-H\right)^{-1}(n)\right)\nabla_x B_\nu(T(x))\right)=0.
  \end{split}
  \end{equation}  
  
 The behavior of $ I_\nu $ close to the boundary $ \bnd $ is described by two nested boundary layer equations. As anticipated at the beginning of Subsection \ref{Subsec.3}, since $ \ell_M\ll\ell_T\ll L $ we observe the formation of two distinct boundary layers. The first one, the Milne layer, has a thickness of size $ \ell_M $ and it is described by the Milne problem, whose derivation is similar to the derivation of the Milne problems \eqref{Milnecase1} and \eqref{Milnecase2}. The next boundary layer, which we will denote by thermalization layer, has a thickness of size $ \ell_T $ and it is described by a new boundary layer equation, which we will denote as thermalization equation and which we will construct immediately after deriving the Milne problem.
 
 Following the same procedure as in Subsection \ref{subsec.1} we can derive the Milne problem for this scaling limit under the rescaling \eqref{scalingMilne}. In this case we obtain a closed equation for $ I_\nu $ which depends only on the scattering process, since this is the largest term. Indeed, rescaling the space variable we obtain
 \begin{equation}\label{case3}
 	\begin{cases}
 		-(n\cdot n_p)\partial_yI_\nu(y,n,p)=\alpha_\nu^s\left(p+\mathcal{O}(\eps)\right)\left(\int_\Ss K(n,n')I_\nu(y,n',p)\;dn'-I_\nu(y,n,p)\right)\\\;\;\;\;\;\;\;\;\;\;\;\;\;\;\;\;\;\;\;\;\;\;\;\;\;\;\;\;\;\;\;\;\;\;\;\;\;\;\;+\eps^{\beta+1}\alpha_\nu^a\left(p+\mathcal{O}(\eps)\right)\left(B_\nu(T(y;p))-I_\nu(y,n;p)\right)& y>0\;,n\in\Ss\\
 		\int_0^\infty d\nu\int_\Ss dn\; (n\cdot n_p) \partial_yI_\nu(y,n,p)=0&y>0,\;n\in\Ss\\
 		I_\nu(0,n,p)=g_\nu(n)& n\cdot n_p<0.
 	\end{cases}
 \end{equation}

 Hence, for every $ p\in\bnd $ the Milne problem is given by
  \begin{equation}\label{Milnecase3}
  \begin{cases}
  -(n\cdot n_p)\partial_yI_\nu(y,n,p)=\alpha_\nu^s(p)\left(\int_\Ss K(n,n')I_\nu(y,n',p)\;dn'-I_\nu(y,n,p)\right)& y>0\;,n\in\Ss\\
  I_\nu(0,n,p)=g_\nu(n)& n\cdot n_p<0.
  \end{cases}
  \end{equation} 
On the other hand, we also obtain an equation for the temperature. Indeed, plugging the first equation of \eqref{case3} into the second one, we obtain to the leading order 
\begin{equation}\label{Milne3tempstationary}
	\int_0^\infty d\nu\int_\Ss dn\;\alpha_\nu^a(p)\left(B_\nu(T(y;p))-I_\nu(y,n;p)\right)=0.
\end{equation}
This equation has a steady distribution for the temperature $ T $ completely determined. At a first glance, this appears strange since in the limit equation \eqref{Milnecase3} the absorption coefficient $ \alpha^a_\nu(p) $ does not appear and the only processes able to modify the temperature are the absorption and emission of photons. However, the solution of this apparent paradox is that since \eqref{Milnecase3} describes a stationary solution, it is implicitly understood that the system was running during an infinite amount of time before and the absorption/emission process had time to bring the system to a steady state, even when this process is very small.

The Milne problem for the pure scattering case has been studied in several papers such as \cite{BardosSantosSentis,papanicolaou, GuoLei2d,Sentis1} in the context of neutron transport. Although all these results are actually obtained for functions $ \alpha^s $ independent of the frequency, since the one-speed approximation for the neutron transport (cf. \eqref{NTE}) was considered, they are expected to hold pointwise for every frequency $ \nu $. 
For example, in \cite{BardosSantosSentis} it is shown that there exists a unique solution to \eqref{Milnecase3} for strictly positive bounded and rotational symmetric scattering kernels. Moreover, as $ y\to\infty $ the solution approaches a function $ I(\nu;p) $ independent of $ n\in\Ss $. Hence, in the Milne layer the radiation intensity becomes isotropic.

We now turn to the thermalization layer. In this layer we expect the radiation intensity to approach the Planck equilibrium distribution. Moreover, the boundary value for the problem \eqref{interior3stationary} can be also found analyzing the thermalization layer. In order to construct the new boundary layer equation, i.e. the thermalization equation, we rescale the space variable according to 
the one-dimensional variable $ \eta=-\frac{x-p}{\ell_T}\cdot n_p $ for $ p\in\bnd $ and we obtain the following equation
  \begin{equation}\label{themaleq1.3}
  	\begin{cases}
  		-\eps^{\frac{1+\beta}{2}}(n\cdot n_p)\partial_\eta I_\nu(\eta,n,p)=\alpha_\nu^a\left(p+\eps^{\frac{1-\beta}{2}}\text{Rot}_p(\eta)\right)\eps^{1+\beta}(B_\nu(T(\eta,p))-I_\nu(\eta,n,p))\\\;\;\;\;\;\;\;+\alpha_\nu^s\left(p+\eps^{\frac{1-\beta}{2}}\text{Rot}_p(\eta)\right)\left(\left(\int_{\Ss} dn' K(n,n')I_\nu(\eta,n',p)\right)-I_\nu(\eta,n,p)\right)& \eta>0\;,n\in\Ss\\
  		\Div\left(\int_0^\infty d\nu\int_\Ss dn\; (n\cdot n_p) I_\nu(\eta,n,p)\right)=0&\eta>0,\;n\in\Ss\\
  		I_\nu(0,n,p)=I(\nu;p)& p\in\bnd,	
  	\end{cases}
  \end{equation}
where $ I(\nu;p)=\lim\limits_{y\to\infty}I^M(y,n,\nu;p) $ for $ I^M $ the solution to the Milne problem \eqref{Milnecase3}. In order to find the thermalization equation we proceed in a way similar to the derivation of the outer problem. We hence expand the radiation intensity according to $$ I_\nu(\eta,n;p)=\varphi_0(\eta,n,\nu;p)+\eps^{\frac{1+\beta}{2}}\varphi_1(\eta,n,\nu;p)+ \eps^{1+\beta}\varphi_2(\eta,n,\nu;p)+\cdots$$ and we identify in \eqref{themaleq1.3} all terms of the same power of $ \eps $, namely $ \eps^0 $, $ \eps^{\frac{1+\beta}{2}} $ and $ \eps^{1+\beta} $. We remark first that the functions $ \varphi_i $ for $ i\in\mathbb{N} $ could depend on $ \eps $. Moreover, the choice of the powers of $ \eps $ in the expansion of $ I_\nu $ is motivated by the order of magnitude of the sources in \eqref{themaleq1.3}.

The terms of order $ \eps^0 $ give $$ \int_\Ss K(n,n')\varphi_0 dn'=\varphi_0 $$ and hence by Proposition \ref{prop.constant} $ \varphi_0(\eta,n,\nu;p)=\varphi_0(\eta,\nu;p) $ is independent of the direction $ n\in\Ss $. The isotropy of $ \varphi_0 $ was expected as it is matched with the solution of the Milne problem, which becomes isotropic. Moreover, we see also that $ \varphi_0 $ does not depend on $ \eps $.

The terms of order $ \eps^{\frac{1+\beta}{2}} $ give $$ (n\cdot n_p)\partial_\eta \varphi_0=\alpha_\nu^s\left(p_\eps\right)(Id-H)(\varphi_1) ,$$ where we defined $ p_\eps= p+\eps^{\frac{1-\beta}{2}}\text{Rot}_p(\eta)$. Thus, arguing as in the derivation of \eqref{phi2.final.3}, Proposition \ref{prop.constant} implies $$ \varphi_1(\eta,n,\nu;p)=\frac{1}{\alpha_\nu^s\left(p_\eps\right)}(Id-H)^{-1}(n)\cdot n_p \partial_\eta \varphi_0 + c(\eta,\nu),$$ for some function $ c(\eta,\nu) $. Finally, identifying the terms of order $ \eps^{1+\beta} $ implies after an integration over $ \Ss $

\begin{equation*}\label{case3.11}
\begin{split}
		-\frac{1}{\alpha_\nu^s\left(p_\eps\right)}&\left(\fint_\Ss (n\cdot n_p)(Id-H)^{-1}(n)\cdot n_p \;dn\right)\partial_\eta^2\varphi_0(\eta,\nu;p)\\=&\alpha_\nu^a\left(p_\eps\right)\left( B_\nu(T(\eta;p))-\varphi_0(\eta,\nu;p)\right).
\end{split}
\end{equation*}
We now consider the limit as $ \eps\to 0 $ and we obtain by the continuity of the absorption and scattering coefficient
\begin{equation}\label{case3.1}	
		-\frac{1}{\alpha_\nu^s\left(p\right)}\left(\fint_\Ss (n\cdot n_p)(Id-H)^{-1}(n)\cdot n_p \;dn\right)\partial_\eta^2\varphi_0(\eta,\nu;p)=\alpha_\nu^a\left(p\right)\left( B_\nu(T(\eta;p))-\varphi_0(\eta,\nu;p)\right).
\end{equation} Moreover, the second equation in \eqref{themaleq1.3} yields  \begin{equation}\label{case3.2}
\int_0^\infty d\nu \frac{1}{\alpha_\nu^s(p)} \left(\fint_\Ss (n\cdot n_p)(Id-H)^{-1}(n)\cdot n_p \;dn\right)\partial_\eta^2\varphi_0(\eta,\nu;p)=0,
\end{equation}
where we again considered the limit $ \eps\to 0 $. Thus, the thermalization layer equation is given for every $ p\in\bnd $ by
\begin{equation}\label{themaleq2.3}
	\begin{cases}
		\varphi_0(\eta,\nu;p) -\frac{1}{\alpha_\nu^a(p)\alpha_\nu^s(p)}\left(\fint_\Ss (n\cdot n_p)(Id-H)^{-1}(n)\cdot n_p \;dn\right)\partial_\eta^2\varphi_0(\eta,\nu;p)=B_\nu(T(\eta;p))& \eta>0\\
	\int_0^\infty d\nu\;\alpha_\nu^a(p)B_\nu(T(\eta;p))=\int_0^\infty d\nu\;\alpha_\nu^a(p)\varphi_0(\eta,\nu;p)&\eta>0\\
		\varphi_0(0,\nu;p)=I(\nu;p)& p\in\bnd,	
	\end{cases}
\end{equation}
where the second equation is implied by \eqref{case3.2} taking the integral over the frequency of \eqref{case3.1}. As far as we know, the thermalization problem has not been studied so far in the literature and its well-posedness properties have not been described in detail. Nevertheless, we claim that the problem is well-posed under suitable assumptions and that the solution $ \varphi_0 $ to \eqref{themaleq2.3} converges to the Planck distribution, i.e.
$$ \lim\limits_{\eta\to\infty}\varphi_0(\eta,\nu;p)=\varphi(\nu,p)=B_\nu(T_\infty(p)).$$
From the second equation in \eqref{themaleq2.3} we recover the relation \eqref{stationary.temp} between the temperature and the radiation intensity $ \varphi_0 $. Hence, $ T(\eta;p)=F^{-1}\left(\left(\int_0^\infty d\nu\;\alpha_\nu^a(p)\varphi_0(\eta,\nu;p)\right),\eta;p\right) $ for $ F $ defined in \eqref{F}. In particular,
\begin{equation}\label{boundary.3}
T_\infty(p)=F^{-1}\left(\left(\int_0^\infty d\nu\;\alpha_\nu^a(p)\varphi(\nu,p)\right),p\right).
\end{equation}
We remark that since $ I(\nu;p) $, the limit as $ y\to\infty $ of the solution $ I^M(y,n,\nu;p) $ of the Milne problem \eqref{Milnecase3}, is a functional of the boundary condition $ g_\nu $, so are $ \varphi(\nu,p) $ and $ T_\infty(p) $ functionals of the boundary condition $ g_\nu $.
Summarizing, in the case of $ \ell_M\ll \ell_T\ll L $ the solution to \eqref{rtestationaryeps} is expected to solve in the limit problem the following equilibrium diffusion approximation given by the stationary boundary value problem
  \begin{equation*}\label{equilibtriumstationary3}
  \begin{cases}
  \Div \left(\int_0^\infty d\nu\frac{1}{\alpha_\nu^s(x)}\left(\int_\Ss dn\; n\otimes \left(Id-H\right)^{-1}(n)\right)\nabla_x B_\nu(T(x))\right)=0& x\in\Omega\\
  T(p)=T_\infty(p)& p\in\bnd,
  \end{cases}
  \end{equation*}
  where $ T_\infty $ is defined in \eqref{boundary.3} 

\subsection{Case 3: $ \ell_M\ll \ell_T=L $. Transition from equilibrium to non-equilibrium.}\label{Subsec.4}
Since $ \ell_M=\eps $ and $ \ell_T=\sqrt{\eps\ell_A}=L=1 $ we have to consider $ \ell_S=\eps $ and $ \ell_A=\eps^{-1} $.

This case is intriguing, because as we will see it yields the transition between the equilibrium approximation and the non-equilibrium approximation, i.e. the case where in the limit the radiation intensity is not given by the Planck distribution at the leading order in the bulk of the domain $ \Omega $. 

As usual we plug  the expansion \eqref{expansion} for $ \delta=1 $ into the first equation of \eqref{rtestationaryeps1} and we identify all terms of the same power of $ \eps $, namely $ \eps^{-1} $, $ \eps^0 $ and $ \eps^1 $.

The terms of order $ \eps^{-1} $ give
\begin{equation*}\label{epsminus1.4}
	\phi_0(x,n,\nu)=H[\phi_0(x,\cdot,\nu)](n),
\end{equation*}
and hence by Proposition \ref{prop.constant} the leading order is independent of $ n\in\Ss $, i.e. $ \phi_0=\phi_0(x,\nu) $.

The terms of order $ \eps^0 $ give
\begin{equation*}\label{eps0.4}
n\cdot\nabla_x\phi_0(x,\nu)=\alpha_\nu^s(x)\left(H[\phi_1(x,\cdot,\nu)](n)-\phi_1(x,n,\nu)\right).
\end{equation*}
Due to the isotropy of $ \phi_0 $, Proposition \ref{prop.constant} implies that $ \phi_1 $ is given by
\begin{equation*}\label{phi1.4}
\phi_1(x,n,\nu)=-\frac{1}{\alpha_\nu^s(x)}\left(Id-H\right)^{-1}(n)\cdot\nabla_x\phi_0(x,\nu)+c(x,\nu),
\end{equation*}
where $ c(x,\nu) $ is some function independent of $ n\in\Ss $. As in subsection \ref{Subsec.3} the isotropic function $ c(x,\nu) $ will not contribute to the divergence free condition, hence it will not be explicitly computed.

Finally, the terms of order $ \eps^1 $ yield, after an integration over $ \Ss $
\begin{equation}\label{eps1.4.2}
\begin{split}
4\pi\alpha_\nu^a(x) \phi_0(x,\nu)-\Div\left(\frac{1}{\alpha_\nu^s(x)}\left(\int_\Ss n\otimes\left(Id-H\right)^{-1}(n)\;dn\right)\nabla_x\phi_0(x,\nu)\right)=4\pi\alpha_\nu^a(x)B_\nu(T(x)),
\end{split}
\end{equation}
where we used the invariance under rotations of the scattering kernel $ K $ and the identity $ n\cdot \nabla_x f=\Div(nf) $.

Moreover, plugging the expansion $$ I_\nu(x,n)=\phi_0(x,\nu)-\frac{\eps}{\alpha_\nu^s(x)}\left(Id-H\right)^{-1}(n)\cdot\nabla_x\phi_0(x,\nu)+\eps c(x,\nu)+\eps^2\cdots  $$
into the divergence free equation in \eqref{rtestationaryeps1} we obtain at the leading order
\begin{equation*}
\Div\left(\int_0^\infty d\nu\;\frac{1}{\alpha_\nu^s(x)}\left(\int_\Ss n\otimes\left(Id-H\right)^{-1}(n)\;dn\right)\nabla_x\phi_0(x,\nu)\right)=0,
\end{equation*}
which implies integrating \eqref{eps1.4.2} the following equation for the temperature
\begin{equation*}\label{phi0T}
\int_0^\infty d\nu\;\alpha_\nu^a(x)\phi_0(x,\nu)=\int_0^\infty d\nu\;\alpha_\nu^a(x)B_\nu(T(x)).
\end{equation*}
Hence, using the definition of $ F $ in \eqref{F} we obtain the limit problem for $ \phi_0 $ in the interior, namely
\begin{equation*}\label{interior4stationary}
\begin{split}
 \phi_0(x,\nu)-\frac{1}{4\pi\alpha_\nu^a(x)}&\Div\left(\frac{1}{\alpha_\nu^s(x)}\left(\int_\Ss n\otimes\left(Id-H\right)^{-1}(n)\;dn\right)\nabla_x\phi_0(x,\nu)\right)\\=&B_\nu\left(F^{-1}\left(\left(\int_0^\infty d\nu\;\alpha_\nu^a(x)\phi_0(x,\nu)\right),x \right)\right).
\end{split}
\end{equation*}
Once more the boundary condition for the diffusion equation is given by the matching of the outer solution with the solution to a suitable boundary layer equation. Since $ \ell_T=L=1 $ the thermalization layer corresponds to the outer problem. Indeed, the radiation intensity is out of equilibrium in the limit as $ \eps\to 0 $. Hence, there is only one boundary layer, namely the Milne layer. The Milne problem describing the boundary layer for \eqref{rtestationaryeps1} as $ \ell_M\ll L=\ell_T $ is given once more by the \eqref{Milnecase3}. Indeed, the scattering term is the term of larger order with $ \ell_M=\ell_S $. Therefore, the computations in Subsection \ref{Subsec.3} hold in this case too. Summarizing, if $ \ell_M\ll \ell_T=L $ the radiation intensity and the temperature satisfy the following equation
\begin{equation*}\label{equilibtriumstationary4}
\begin{cases}
\phi_0(x,\nu)-\frac{1}{4\pi\alpha_\nu^a(x)}\Div\left(\frac{1}{\alpha_\nu^s(x)}\left(\int_\Ss n\otimes\left(Id-H\right)^{-1}(n)\;dn\right)\nabla_x\phi_0(x,\nu)\right)=B_\nu\left(T(x)\right)& x\in\Omega\\
\int_0^\infty d\nu\; \alpha_\nu^a(x)\left(B_\nu(T(x))-\phi_0(x,\nu)\right)=0& x\in\Omega\\
\phi_0(p,\nu)=I_\nu^\infty(p)& p\in\bnd,
\end{cases}
\end{equation*}
where $ I^\infty_\nu(p)=\lim\limits_{y\to\infty}I_\nu(y,n;p) $ for $ I_\nu(y,n;p) $ the solution to \eqref{Milnecase3} which converges to the isotropic function $ I_\nu^\infty $. It is important to remark here that in this case we are not obtaining an equilibrium diffusion regime. Indeed, the leading order $ \phi_0 $ is not the Planck distribution and therefore this case is an example of the non-equilibrium diffusion approximation.\\

\subsection{Case 4: $ \ell_M\ll L\ll\ell_T $. Non-Equilibrium approximation}\label{Subsec.5}
Since $ \ell_M=\eps $ the case where $ \ell_T=\sqrt{\eps \ell_A}\gg L=1 $ corresponds to $ \ell_S=\eps $ and $ \ell_A=\eps^{-\beta} $ for $ \beta>1 $. Under this assumption we obtain $ \ell_T=\eps^{\frac{1-\beta}{2}}\to \infty $ as $ \eps\to 0 $. Therefore, in this last subsection we study the case when the thermalization length $ \ell_T $ is growing to infinity as $ \eps\to 0 $. In this case we do not expect the solution to \eqref{rtestationaryeps} to approach at the interior the Planck distribution. We will indeed see that in this case we obtain the so called non-equilibrium diffusion approximation.

In order to derive the outer problem for \eqref{rtestationaryeps1} we plug  expansion \eqref{expansion} with $ \delta=\beta-1 $ into the first equation of \eqref{rtestationaryeps1} and we identify all terms of the same power of $ \eps $, namely $ \eps^{-1} $, $ \eps^{\beta-2} $, $ \eps^0 $, $ \eps^{\beta-1} $ and $ \eps^1 $.

The terms of order $ \eps^{-1} $ and $ \eps^{\beta-2} $ yield $ \int_\Ss K(n,n')\phi_i(n')dn'=\phi_i $ for $ i=0,2 $, respectively. Therefore, at the leading order the radiation intensity is isotropic, i.e. $ \phi_0=\phi_0(x,\nu) $. Moreover also $ \phi_1=\phi_1(x,\nu) $.

The terms of power $ \eps^0 $ give
\begin{equation*}\label{eps0.5}
	-\frac{1}{\alpha_\nu^s(x)}n\cdot\nabla_x\phi_0=(Id-H)[\phi_2(x,\cdot,\nu)](n).
\end{equation*}
Hence, Proposition \ref{prop.constant} implies the existence of some function $ c(x,\nu) $ independent of $ n\in\Ss $ such that
\begin{equation}\label{phi1.5}
\phi_2(x,n,\nu)=	-\frac{1}{\alpha_\nu^s(x)}(Id-H)^{-1}(n)\cdot\nabla_x\phi_0+c(x,\nu).
\end{equation}

Similar to the terms of order $ \eps^{0} $, the terms of power $ \eps^{\beta-1} $ give $ \phi_3=-\frac{1}{\alpha_\nu^s(x)}(Id-H)^{-1}(n)\cdot\nabla_x\phi_1+c(x,\nu) $. As in subsection \ref{Subsec.3} the isotropic function $ c(x,\nu) $ do not contribute to the divergence free condition and it will not be explicitly computed.

Finally, the terms of order $ \eps^1 $ imply
\begin{equation*}\label{eps1.5}
n\cdot\nabla_x\phi_2=\alpha_\nu^s(x)(H-Id)[\phi_4(x,\cdot,\nu)](n).
\end{equation*}
Hence, using \eqref{phi1.5} and integrating over $ \Ss $ we obtain the desired interior limit problem for $ \phi_0 $
\begin{equation*}\label{interior5stationary}
	\Div\left(\frac{1}{\alpha_\nu^s(x)}\left(\int_\Ss n\otimes\left(Id-H\right)^{-1}(n)\;dn\right)\nabla_x\phi_0(x,\nu)\right)=0.
\end{equation*}
Plugging now the first equation of \eqref{rtestationaryeps} into the second one we obtain also the following equation solved by the leading order of the temperature 
\begin{equation*}\label{interior5.1stationary}
	\int_0^\infty d\nu\; \alpha_\nu^a(x)\left(B_\nu(T(x))-\phi_0(x,\nu)\right)=0.
\end{equation*}
As in Subsection \ref{Subsec.3} the Milne problem for the Milne layer is given by \eqref{Milnecase3}. As in subsection \ref{Subsec.4} there is no thermalization layer since the radiation intensity does not approach the equilibrium distribution. Hence, denoting by $ I_\nu(y,n,p) $ the solution to \eqref{Milnecase3} and by $ I_\nu^\infty(p)=\lim\limits_{y\to\infty}I_\nu(y,n;p) $ we obtain for this case the following limit stationary boundary value problem
\begin{equation*}\label{equilibtriumstationary5}
\begin{cases}
\Div\left(\frac{1}{\alpha_\nu^s(x)}\left(\int_\Ss n\otimes\left(Id-H\right)^{-1}(n)\;dn\right)\nabla_x\phi_0(x,\nu)\right)=0& x\in\Omega\\
\int_0^\infty d\nu\; \alpha_\nu^a(x)\left(B_\nu(T(x))-\phi_0(x,\nu)\right)=0& x\in\Omega\\
\phi_0(p,\nu)=I_\nu^\infty(p)& p\in\bnd.
\end{cases}
\end{equation*}

\section{Time dependent diffusion approximation. The case of infinite speed of light ($ c=\infty $)}\label{Sec.5}
We turn now to the time dependent case. In physical applications the order of magnitude of the speed of light $ c $ is so large compared with the speed of heat transfer that it is often considered infinite (cf. \cite{Zeldovic}). This approximation is valid if the distance traveled by the light in the time scale in which meaningful changes of the temperature take place is much larger than the characteristic length of the body $ L $. We consider in this section the diffusion approximation for the time dependent radiative transfer equation \eqref{rtetimeps} when $ c= \infty $ and in the next sections we will consider other choices of $ c $. Under this assumption the initial-boundary value problem \eqref{rtetimeps} reduces to \eqref{rtetimepsc}. This is the case when the radiation is instantaneously transported in the domain $ \Omega $. Notice that since under this assumption in equation \eqref{rtetimepsc} there is no term containing $ \partial_tI_\nu $ we do not need to impose any initial value for $ I_\nu $.

We recall that the diffusion regime holds if $ \ell_M=\eps\ll 1 $. We will consider different choices of $ \ell_A $ and $ \ell_S $ given as powers of $\eps $. We will construct the resulting initial-boundary value limit problems as follows. We will first derive the outer problems valid in the interior of $ \Omega $. Afterwards we will construct the initial layer problems describing the transient behavior of the radiation intensity for very small times. We will formulate also boundary layer equations describing $ I_\nu $ near the boundary of $ \Omega $. It turns out that the latter are the Milne problems and the thermalization problems derived in Section \ref{Sec.4}. Finally, the matching between the outer, the boundary layer and the initial layer solutions will provide the initial value and the boundary conditions for the limit problem in the diffusion approximation under consideration.

\subsection{Outer problems}\label{4.1}
In this subsection we derive the outer problems arising from equation \eqref{rtetimepsc} under the assumption $ \ell_M=\eps\ll 1 $ and for different choices of $ \ell_A=\eps^{-\beta} $ and $ \ell_S=\eps^{-\gamma} $. As in the stationary case analyzed in Section \ref{Sec.4} there are five different cases to be considered which yield five different diffusive problems.

In order to determine the outer problems yielding the form of the solutions in the bulk of $ \Omega $ we use the expansion \begin{equation}\label{expansion1}
		I_\nu(t,x,n)=\phi_0(t,x,n,\nu)+\eps^\delta \phi_1(t,x,n,\nu)+\eps\phi_2(t,x,n,\nu)+\eps^{\delta+1}\phi_3(x,n,\nu)+...
\end{equation} for suitable choices of $ \delta $ depending on $ \ell_A $ and $ \ell_S $, plugging \eqref{expansion1} into \eqref{rtetimepsc} and identifying all terms of the same power of $ \eps $. It turns out that the diffusive problems are in this case the time dependent version of the stationary outer problems of Section \ref{Sec.4}. Indeed, since $ c=\infty $ the first equation in \eqref{rtetimepsc} is a stationary equation for the intensity $ I_\nu $. Therefore, the same computations of Section \ref{Sec.4} show that for any choice of $ \ell_A $ and $ \ell_S $ the first order term $ \phi_0 $ is isotropic and the next non-isotropic term arising in the expansion of $ I_\nu $ is of order $ \eps^1 $.

Hence, in the case $ \ell_T\leq 1 $, i.e. $\tau_{\mathit{h}}=\frac{1}{\eps}$, since the time derivative of the temperature in the second equation of \eqref{rtetimepsc} is a term of order $ \eps^0 $ which is balanced by the divergence of the flux of energy, we obtain the following outer problems
\begin{enumerate}[(i)]
	\item for $ \ell_M=\ell_T\ll\ell_S $ \begin{equation}\label{outerc1}
	\partial_tT(t,x)-\frac{4\pi}{3}\Div\left(\int_0^\infty \frac{\nabla_x B_\nu(T(t,x))}{\alpha_\nu(x)}d\nu\right)=0,
\end{equation}
\item for $ \ell_M=\ell_T=\ell_S $ \begin{equation}\label{outerc2}
	\begin{split}
		\partial_tT(t,x)=\Div \left(\int_0^\infty d\nu\frac{1}{\alpha_\nu^a(x)+\alpha_\nu^s(x)}\left(\int_\Ss dn\; n\otimes \left(Id-A_{\nu,x}\right)^{-1}(n)\right)\nabla_x B_\nu(T(t,x))\right),
	\end{split}
\end{equation}
\item for $ \ell_M\ll\ell_T\ll L $ \begin{equation}\label{outerc3}
	\partial_tT(t,x)=\Div \left(\int_0^\infty d\nu\frac{1}{\alpha_\nu^s(x)}\left(\int_\Ss dn\; n\otimes \left(Id-H\right)^{-1}(n)\right)\nabla_x B_\nu(T(t,x))\right),
\end{equation}
\item for $ \ell_M\ll L=\ell_T $ \begin{equation}\label{outerc4}
	\begin{cases}
		-\frac{1}{\alpha_\nu^a(x)}\Div\left(\frac{1}{\alpha_\nu^s(x)}\left(\fint_\Ss n\otimes\left(Id-H\right)^{-1}(n)\;dn\right)\nabla_x\phi_0(t,x,\nu)\right)=B_\nu\left(T(t,x)\right)-\phi_0(t,x,\nu) \\
		\partial_t T(t,x)-\Div\left(\int_0^\infty d\nu\;\frac{1}{\alpha_\nu^s(x)}\left(\fint_\Ss n\otimes\left(Id-H\right)^{-1}(n)\;dn\right)\nabla_x\phi_0(t,x,\nu)\right)=0.
	\end{cases}
\end{equation}
\end{enumerate}
In the case $ \ell_M\ll L\ll\ell_T $, namely when $ \tau_{\mathit{h}}=\frac{1}{\eps^\beta} $ for $ \beta>1 $ the outer problem is
\begin{equation}\label{outerc5}
	\begin{cases}
		\Div\left(\frac{1}{\alpha_\nu^s(x)}\left(\fint_\Ss n\otimes\left(Id-H\right)^{-1}(n)\;dn\right)\nabla_x\phi_0(t,x,\nu)\right)=0\\
		\partial_t T(t,x)-\int_0^\infty d\nu\;\alpha_\nu^a(x)\left(B_\nu\left(T(t,x)\right)-\phi_0(t,x,\nu) \right)=0.
	\end{cases}
\end{equation}
Indeed, plugging the expansion \eqref{expansion1} with $ \delta=\beta-1 $ into the first equation in \eqref{rtetimepsc} we obtain, arguing as in Section \ref{Subsec.5}, that the leading order $ \phi_0 $ is isotropic and solves the stationary equation
$$ \Div\left(\frac{1}{\alpha_\nu^s(x)}\left(\fint_\Ss n\otimes\left(Id-H\right)^{-1}(n)\;dn\right)\nabla_x\phi_0(t,x,\nu)\right)=0. $$
Moreover, plugging the second equation of \eqref{rtetimepsc} into the second one yields
$$ 	\partial_t T(t,x)-\int_0^\infty d\nu\;\alpha_\nu^a(x)\left(B_\nu\left(T(t,x)\right)-\phi_0(t,x,\nu) \right)=0.$$
These are the equations describing the radiation intensity and the temperature on the bulk away from the boundary and for positive times.

We remark that as for the stationary problem the regimes of equilibrium diffusion approximations are for $ \ell_T\ll L $ and correspond to the problems \eqref{outerc1}, \eqref{outerc2} and \eqref{outerc3} while the regimes of non-equilibrium approximations are for $ \ell_T\gtrsim L $ and are described by \eqref{outerc4} and \eqref{outerc5}.
\subsection{Initial layer equations and boundary layer equations}\label{4.2}
As in the case of the stationary diffusion approximation, the radiation intensity $ I_\nu $ and the temperature $ T $ can change abruptly near the boundaries, i.e. boundary layers might arise. In addition, in the time dependent case also the behavior of $ (T,I_\nu) $ could change quickly for small times. We will denote the latter as initial layers. In this subsection we construct the initial layers for distances to the boundary of order $ 1 $ and boundary layers for positive times of order $ 1 $. We denote by initial layer equations the problems derived for times$ t\ll 1 $ and solved at the interior of $ \Omega $. Similarly, the boundary layer equations are problems derived rescaling the space variable only and solved for any $ t>0 $.

In the considered case, i.e. $ c=\infty $, there are no initial layers for the temperature appearing on the bulk, i.e. for distances to the boundary of order $ 1 $. To see this we have to consider two different cases. We recall that the second equation in \eqref{rtetimepsc} is
\begin{equation}\label{second}
	\partial_t T(t,x)+\tau_{\mathit{h}} \Div\left(\int_0^\infty d\nu\int_\Ss dn\; nI_\nu(t,x,n)\right)=0.
\end{equation} Hence, if $ \ell_T\leq 1 $ the heat parameter is $ \tau_{\mathit{h}}=\frac{1}{\eps} $. Therefore in equation \eqref{second} the divergence of the flux of radiative energy is multiplied by $ \eps^{-1} $. As indicated before $ \phi_0 $ is isotropic. In addition to that, since the first non-isotropic term is of order $ \eps $, it follows that in \eqref{second} the term containing the divergence is of order $ 1 $ in the bulk. Therefore $ \partial_t T $ is of order $ 1 $ and as a consequence $ T\simeq T_0 $ for small times $ t\ll 1 $ and no initial layer appears. On the other hand, in the case $ \ell_T\gg 1 $ the heat parameter is $ \tau_{\mathit{h}}=\ell_A=\frac{1}{\eps^\beta}$ for $ \beta>1 $. In this case the leading term of the divergence of the total flux of energy is of order $ \eps^\beta $ and it is given by $$ \eps^\beta\Div\left(\int_0^\infty d\nu\int_\Ss dn\;n\phi_3\right)=\eps^\beta \int_0^\infty d\nu\int_\Ss dn\;\alpha_\nu^a\left(B_\nu(T)-\phi_0\right)  ,$$ where $ \phi_3 $ is the term of order $ \eps^{\beta} $ in the expansion \eqref{expansion1} obtained for $ \delta=\beta-1 $. This implies again that $ \partial_t T $ is of order $ 1 $ and hence, there are also in this case no initial layers.\\

We now examine the boundary layers appearing for times of order $ 1 $. In this case, similarly as in the stationary case, Milne and thermalization layers arise. It turns out that the equations describing the radiation intensity near the boundary are given either by the stationary Milne problems \eqref{Milnecase1}, \eqref{Milnecase2}, \eqref{Milnecase3}, or by the thermalization problem \eqref{themaleq2.3} or by a combination of both of them depending on the choice of $ \ell_A $ and $ \ell_S $.

We begin describing first the Milne layers. We rescale the space variable according to \linebreak$ y=-\frac{x-p}{\eps}\cdot n_p $, where $ \ell_M=\eps $ and $ p\in\bnd $. We express also the absorption and scattering lengths according to $ \ell_A=\eps^{-\beta} $, $ \ell_S=\eps^{-\gamma} $ with $\min\{\beta,\gamma\}=-1$. With this notation, \eqref{rtetimepsc} becomes
\begin{equation}\label{timec.boundary.2}
	\begin{cases}
		-(n\cdot n_p)\partial_y I_\nu(t,y,n;p)=\eps^{\beta+1} \alpha^a_\nu\left(p+\mathcal{O}\left(\eps\right)\right)(B_\nu(T)-I_\nu)\\\;\;\;\;\;\;\;\;\;\;\;\;\;\;\;\;\;\;\;\;+\eps^{\gamma+1}\alpha_\nu^s\left(p+\mathcal{O}\left(\eps\right)\right)\left(\int_\Ss K(n,n')I_\nu\;dn'-I_\nu\right)&y>0\\
		\partial_t T(t,y;p)-\frac{\tau_{\mathit{h}}}{\eps} \left(\int_0^\infty d\nu \int_\Ss dn\;(n\cdot n_p)\partial_y I_\nu\right)=0& y>0\\
		T(0,y;p)=T_0(y;p)& y>0\\
		I_\nu(t,0,n;p)=g_\nu(t,n)& n\cdot n_p<0.	\end{cases}
\end{equation}
Letting $ \eps\to0 $ we obtain different Milne problems for different choices of $ \beta $ and $ \gamma $. With similar arguments as in Section \ref{Sec.4} we can see that the Milne problems are the same as the one derived for the stationary case, except for the fact that the unknowns depend also on the variable $ t $. However, the variable $ t $ appears only as a parameter and the Milne problems are stationary. These are given by \eqref{Milnecase1} in the case $ \gamma>-1 $, by \eqref{Milnecase2} if $ \gamma=\beta=-1 $ and finally by \eqref{Milnecase3} if $ \beta>-1 $. Notice that we are assuming that, if the incoming radiation $ g_\nu $ depends on time, it does it only for times $ t $ of order one.

We remark that when $ \beta>-1 $ the Milne problem \eqref{Milnecase3} is a closed problem involving only the radiation intensity $ I_\nu $. If $ \ell_T\ll L $, in order to determine the temperature close to the boundary we have to solve the stationary equation 
\begin{equation*}\label{temp}
	\int_0^\infty d\nu\int_\Ss dn\;\alpha_\nu^a(p)\left(B_\nu(T(t,y;p))-I_\nu(t,y,n;p)\right) 	=0.
\end{equation*}
This is the same equation we obtained in the stationary case in \eqref{Milne3tempstationary}. On the other hand, if $ \ell_T\gtrsim L $ the temperature is related to the radiation intensity by a time dependent equation similar to the second one in \eqref{outerc4} and \eqref{outerc5}, namely the equations describing the temperature in the bulk, i.e.
%
\begin{equation}\label{milnec1}
	\partial_tT(t,y;p)+ \int_0^\infty d\nu\int_\Ss dn\;\alpha_\nu^a(p)\left(B_\nu(T(t,y;p))-I_\nu(t,y,n;p)\right) 	=0.
\end{equation}
Besides the Milne layer in the case $ \ell_M\ll\ell_T\ll L $ we observe also the formation of a thermalization layer at distance $ \ell_T $ to the boundary. The equation describing this layer is obtained with a change of variable $ \eta=-\frac{x-p}{\ell_T}\cdot n_p $ for $ p\in\bnd $. Recall that in this case we consider $ \ell_S=\eps $ and $ \ell_A=\eps^{-\beta} $ for $ \beta\in(-1,1) $ and hence $ \ell_T=\eps^{\frac{1-\beta}{2}} $ and $ \tau_{\mathit{h}}=\frac{1}{\eps} $. Thus \eqref{rtetimepsc} becomes under this rescaling
\begin{equation}\label{timec.boundary.1}
	\begin{cases}
		-\eps^{\frac{1+\beta}{2}}(n\cdot n_p)\partial_\eta I_\nu(t,\eta,n;p)= \alpha^a_\nu\left(p+\mathcal{O}\left(\eps^{\frac{1-\beta}{2}}\right)\right)\eps^{\beta+1}(B_\nu(T)-I_\nu)\\\;\;\;\;\;\;\;\;\;\;\;\;\;\;\;\;\;\;\;\;+\alpha_\nu^s\left(p+\mathcal{O}\left(\eps^{\frac{1-\beta}{2}}\right)\right)\left(\int_\Ss K(n,n')I_\nu\;dn'-I_\nu\right)&\eta>0\\
		\partial_t T(t,\eta;p)-\eps^{\frac{\beta-3}{2}} \left(\int_0^\infty d\nu \int_\Ss dn\;(n\cdot n_p)\partial_\eta I_\nu\right)=0& \eta>0.\\\end{cases}
	\end{equation}
We see once more that the thermalization layer equation is equation \eqref{themaleq2.3}, the equation constructed for the stationary problem in Section \eqref{Subsec.3}. 
Finally, matching the solution of the boundary layer equations with the outer problem we can construct the boundary condition for the diffusive initial-boundary limit problem. We will summarize these problems in the following subsection.
\subsection{Limit problems in the bulk}\label{4.3}
We summarize now the time dependent PDE problems that we obtain for the equation \eqref{rtetimepsc} as $ \ell_M\to 0 $ for all different choices of $ \ell $'s. They are given by the outer problems \eqref{outerc1}-\eqref{outerc5} valid in the bulk for positive times. Since there are no initial layers appearing for times $ t\ll 1 $, the initial condition is $ T(t,x)=T_0(x) $ for any choice of $ \ell_A $ and $ \ell_S $. Moreover, the boundary condition is given by the matching of the solution of the boundary layer problems with the outer solution.
\begin{enumerate}[(i)]
\item If $ \ell_M=\ell_T\ll\ell_S $ then the problem is given by
\begin{equation}\label{equilibtriumtime1.5}
	\begin{cases}
		\partial_tT(t,x)-\frac{4\pi}{3}\Div\left(\int_0^\infty \frac{\nabla_x B_\nu(T(t,x))}{\alpha_\nu(x)}d\nu\right)=0& x\in\Omega,\;t>0\\
		T(0,x)=T_0(x)& x\in\Omega\\
		T(t,p)=\lim\limits_{y\to\infty}F^{-1}\left(\left(\int_o^\infty d\nu \alpha_\nu^a(p) I_\nu(t,y,n;p)\right),y,p\right)& p\in\bnd,\;t>0,
	\end{cases}
\end{equation}
where $ I_\nu(y,n;p) $ is the solution to the Milne problem \eqref{Milnecase1}.
\item If $ \ell_M=\ell_T\ll L $, we obtain the following limit problem 
\begin{equation}\label{equilibtriumtime2.5}\hspace{-1cm}
	\begin{cases}
		\partial_tT(t,x)=\Div \left(\int_0^\infty d\nu\frac{1}{\alpha_\nu^a(x)+\alpha_\nu^s(x)}\left(\int_\Ss dn\; n\otimes \left(Id-A_{\nu,x}\right)^{-1}(n)\right)\nabla_x B_\nu(T(x))\right)& x\in\Omega,\;t>0\\
		T(0,x)=T_0(x)& x\in\Omega\\
		T(t,p)=\lim\limits_{y\to\infty}F^{-1}\left(\left(\int_0^\infty d\nu\fint_\Ss dn\; \alpha_\nu^a(p)I_\nu(t,y,n,p)\right),y,p\right)& p\in\bnd,\;t>0,
	\end{cases}
\end{equation}
where $ I_\nu(y,n,p) $ solves the Milne problem \eqref{Milnecase2}. 
\item We turn now to the case $\ell_M\ll \ell_T\ll L$, which corresponds to the case $ \ell_M=\eps=\ell_S $ and $ \ell_A=\eps^{-\beta} $ for $ \beta\in(-1,1) $. We obtain the following limit problem
\begin{equation}\label{equilibriumtime3.5}
	\begin{cases}
		\partial_t T-\Div \left(\int_0^\infty d\nu\frac{1}{\alpha_\nu^s(x)}\left(\int_\Ss dn\; n\otimes \left(Id-H\right)^{-1}(n)\right)\nabla_x B_\nu(T(x))\right)=0& x\in\Omega\\
		T(0,x)=T_0(x)& x\in\Omega\\
		T(t,p)=\lim\limits_{\eta\to\infty}F^{-1}\left(\left(\int_0^\infty d\nu \fint_\Ss dn\; \alpha_\nu^a(x)\varphi_0(t,\eta,\nu;p)\right),y,p\right)& p\in\bnd,\;t>0,
	\end{cases}
\end{equation}
where $\varphi_0(t,\eta,\nu;p)$ solves the thermalization equations \eqref{themaleq2.3} with boundary value \linebreak$ \varphi_0(t,0,\nu;p)= \lim\limits_{y\to\infty} I_\nu(t,y,n,\nu;p)$ for $ I_\nu $ the solution to the Milne problem \eqref{Milnecase3} with boundary value $ g_\nu(t,n) $.
\item We consider now the last two cases where $\ell_M\ll L\lesssim\ell_T $. The limit problem in the case $ \ell_T=L $ is 
\begin{equation}\label{equilibtriumtimec4}
	\begin{cases}
		-\frac{1}{\alpha_\nu^a(x)}\Div\left(\frac{1}{\alpha_\nu^s(x)}\left(\fint_\Ss n\otimes\left(Id-H\right)^{-1}(n)\;dn\right)\nabla_x\phi_0(t,x,\nu)\right)\\\;\;\;\;\;\;\;\;\;\;\;\;\;\;\;\;\;\;\;\;\;\;\;\;\;\;\;\;\;\;\;\;\;\;\;\;\;\;\;\;\;\;\;\;\;\;\;\;\;\;\;\;\;\;\;\;\;\;\;\;\;\;\;\;\;\;\;\;\;\;\;\;=B_\nu\left(T(t,x)\right)-\phi_0(t,x,\nu)& x\in\Omega,\;t>0\\
		\partial_t T(t,x)-\int_0^\infty d\nu\;\int_\Ss dn\; dn\;\alpha_\nu^a(x)\left(B_\nu\left(T(t,x)\right)-\phi_0(t,x,\nu)\right)=0& x\in\Omega,\;t>0\\
		T(0,x)=T_0(x)& x\in\Omega\\		\phi_0(t,p,\nu)=\lim\limits_{y\to\infty}\fint_\Ss I_\nu(t,y,n,p)& p\in\bnd,\;t>0,\\
	\end{cases}
\end{equation}
where $ I_\nu(t,y,n,p) $ solves the Milne problem \eqref{Milnecase3} for the boundary value $ g_\nu(t,n) $. Notice that in the problem \eqref{Milnecase3} the time $ t $ appears just as a parameter. 
\item Finally, if $ L\ll \ell_T $ with the same notation as above the limit problem in this case is
\begin{equation}\label{equilibtriumtimec5}
	\begin{cases}
		\Div\left(\frac{1}{\alpha_\nu^s(x)}\left(\fint_\Ss n\otimes\left(Id-H\right)^{-1}(n)\;dn\right)\nabla_x\phi_0(t,x,\nu)\right)=0& x\in\Omega,\;t>0\\
		\partial_t T(t,x)+\int_0^\infty d\nu \int_\Ss \alpha^a_\nu(x)\left(B_\nu(T)(t,x)-I_\nu(t,x,n)\right)=0& x\in\Omega,\;t>0\\\
		T(0,x)=T_0(x)& x\in\Omega\\		\phi_0(t,p,\nu)=\lim\limits_{y\to\infty}\fint_\Ss I_\nu(t,y,n,p)& p\in\bnd,\;t>0.\\
	\end{cases}
\end{equation}
Also for this case the boundary condition is obtained by the solution of the boundary layer described by the Milne problem \eqref{Milnecase3}. 
\end{enumerate}

\subsection{Initial-boundary layers}
It is important to notice that in regions very close to the boundary and for a times $ t\ll1 $ new layers could appear. These are the regions where the radiation intensity $ I_\nu $ and the temperature $ T $ change from the solution of the initial layer equation to the solution of the boundary layer equation. For this reason we denote these layers as initial-boundary layers. In this section we will derive the equations describing them for any choice of $ \ell_A $ and $ \ell_S $. In the following we will always denote by $ p $ a point belonging to the boundary, i.e. $ p\in\bnd $.
\begin{enumerate}[(i)]
\item If $ \ell_M=\ell_T\ll \ell_S $ we observe the formation of only one initial-boundary layer. It is described by an equation which can be constructed rescaling the space variable as $ y=-\frac{x-p}{\eps}\cdot n_p $ and the time by $ t=\eps^2 \tau $. Indeed, since in this case $ \beta=-1 $ (because $ \ell_A=\eps $) and $ \tau_h=\eps^{-1} $ we see that the leading term of divergence of the flux of energy is of order $ \eps^{-2} $ in the following equation
\begin{equation}\label{milnec}
	\partial_tT(t,y;p)+ \tau_{\mathit{h}} \eps^\beta \int_0^\infty d\nu\int_\Ss dn\;\alpha_\nu^a(p)\left(B_\nu(T(t,y;p))-I_\nu(t,y,n;p)\right) 	=0.
\end{equation}
This equation is obtained plugging the first equation in \eqref{timec.boundary.2} into the second one. We recall that equation \eqref{timec.boundary.2} is obtained after a rescaling of only the space variable. Hence, the time rescaling $ t=\eps^2 \tau $ gives a  non-trivial equation for the temperature. Thus, the radiation intensity $ I_\nu $ and the temperature $ T $ solve the following initial-boundary layer equation
\begin{equation*}\label{timec-boundary.1final}
	\begin{cases}
		-(n\cdot n_p)\partial_y I_\nu(\tau,y,n;p)= \alpha^a_\nu(p)(B_\nu(T(\tau,y))-I_\nu(\tau,y,n;p))&y>0,\;\tau>0\\
		\partial_tT(\tau,y)-\left(\int_0^\infty d\nu \int_\Ss dn\;(n\cdot n_p)\partial_yI_\nu(\tau,y,n;p)\right)=0& y>0,\;\tau>0\\
		T(0,y;p)=T_0(p)& y>0\\
		I_\nu(\tau,0,n;p)=g_\nu(0,n)&n\cdot n_p<0,\;\tau>0.	\end{cases}
\end{equation*}
\item In the case $ \ell_M=\ell_T=\ell_S $ under the scaling $ y=-\frac{x-p}{\ell_M}\cdot n_p $ and $ t=\eps^{2}\tau $ we obtain as above the following initial-boundary layer equation
\begin{equation*}\label{timec-boundary.2final}
	\begin{cases}
		-(n\cdot n_p)\partial_y I_\nu(\tau,y,n;p)= \alpha^a_\nu(p)(B_\nu(T(\tau,y))-I_\nu(\tau,y,n;p))\\\text{\phantom{nel mezzo del cammin di }}+\alpha_\nu^s\left(\int_\Ss K(n,n')I_\nu(\tau,y,n';p)\;dn'-I_\nu(\tau,y,n;p)\right)&y>0,\;\tau>0\\
		\partial_\tau T(\tau,y)+\Div \left(\int_0^\infty d\nu \int_\Ss dn\;nI_\nu\right)=0& y>0,\;\tau>0\\
		T(0,y;p)=T_0(p)& y>0\\
		I_\nu(\tau,0,n;p)=g_\nu(0,n)&n\cdot n_p<0,\;\tau>0.	\end{cases}
\end{equation*}
\item If $ \ell_A\ll\ell_T\ll L $ we obtain two different initial-boundary layers. This is consistent with the fact that there are two boundary layers appearing, namely the Milne layer, in which $ I_\nu $ becomes isotropic, and the thermalization layer, in which $ I_\nu $ approaches to the Planck distribution. We now notice that rescaling the space variable by $ y=-\frac{x-p}{\eps}\cdot n_p $ and the time variable according to $ t=\eps^{1-\beta} \tau$ equation \eqref{milnec} gives the following initial-boundary Milne layer equation
\begin{equation}\label{timec-boundary.3}
	\begin{cases}
		-(n\cdot n_p)\partial_y I_\nu(\tau,y,n;p)= \alpha_\nu^s(p)\left(\int_\Ss K(n,n')I_\nu(\tau,y,n';p)dn'-I_\nu(\tau,y,n,\nu;p)\right) &y>0,\;\tau>0\\
		\partial_t T(\tau,y)+\int_0^\infty d\nu \int_\Ss dn\; \alpha^a_\nu(p)\left(B_\nu(T)(\tau,y;p)-I_\nu(\tau,y,n;p)\right)=0& y>0,\;\tau>0\\
		T(0,y;p)=T_0(p)& y>0\\
		I_\nu(\tau,0,n;p)=g_\nu(0,n)&n\cdot n_p<0,\;\tau>0.
	\end{cases}
\end{equation}
Moreover, rescaling the space variable according to $ \eta=-\frac{x-p}{\ell_T}\cdot n_p $ and the time by $ t=\eps^{1-\beta} \tau$ from equation \eqref{timec.boundary.1} we obtain the following initial-boundary thermalization layer equation
\begin{equation*}\label{timec-boundary.5final}
	\begin{cases}
		\varphi_0(\tau,\eta,\nu;p) -\frac{1}{\alpha_\nu^a(p)\alpha_\nu^s(p)}\left(\fint_\Ss (n\cdot n_p)(Id-H)^{-1}(n)\cdot n_p \;dn\right)\partial_\eta^2\varphi_0(\tau,\eta,\nu;p)\\\;\;\;\;\;\;\;\;\;\;\;\;\;\;\;\;\;\;\;\;\;\;\;\;\;\;=B_\nu(T(\tau,\eta;p))& \eta>0,\;\tau>0\\
		\partial_\tau T-\int_0^\infty d\nu\int_\Ss dn\; \alpha^a_\nu(p)\left(B_\nu(T)(\tau,\eta;p)-I_\nu(\tau,\eta,n;p)\right) =0&\eta>0,\;\tau>0\\
		T(0,\eta;p)=T_0(p)& y>0\\
		\varphi_0(\tau,0,\nu;p)=I(0,\nu;p)& p\in\bnd,\;\tau>0.	
	\end{cases}
\end{equation*}
This is the initial-boundary layer equation describing the transition from the initial value to the boundary value in the limit problem \eqref{equilibriumtime3.5}.
\item[(iv)+(v)] Finally, in the last two considered case, namely when $ \ell_T\gtrsim L $ we do not obtain a initial-boundary layer. However, under the space variable rescale $ y=-\frac{x-p}{\eps}\cdot n_p $ for the Milne problem \eqref{Milnecase3} we obtained also an evolution equation for the temperature valid for all $ t>0 $ given as we saw in \eqref{milnec1} by 
\begin{equation*}\label{timec-boundary.4}
	\begin{cases}
		-(n\cdot n_p)\partial_y I_\nu(t,y,n;p)= \alpha_\nu^s(p)\left(\int_\Ss K(n,n')I_\nu(t,y,n';p)dn'-I_\nu(t,y,n,\nu;p)\right) &y>0, t>0\\
		\partial_t T(t,y)+\int_0^\infty d\nu \int_\Ss \alpha^a_\nu(p)\left(B_\nu(T)(t,y;p)-I_\nu(t,y,n;p)\right)=0& y>0, t>0\\
		T(0,y;p)=T_0(p)&y>0\\
		I_\nu(t,0,n;p)=g_\nu(t,n)&n\cdot n_p<0, t>0.
	\end{cases}
\end{equation*}
\end{enumerate}

\section{Time dependent diffusion approximation. The case of speed of light of order $ 1 $}\label{Sec.6}
In this section we construct the limit problem solved by the solution of the time dependent equation \eqref{rtetimeps} when $ \ell_M\to 0 $ and the speed of light is finite. Without loss of generality we consider first the case $ c=1 $. Physically this means that the characteristic time for the propagation of light is similar to the time of the heat transfer process. This situation can be expected to be relevant only in astrophysical applications. The strategy is the same as in Section \ref{Sec.5}. We will first formulate the limit problem valid at the interior of the domain $ \Omega $ for positive times. In Subsection \ref{subs5.2} we will consider the formation of initial and boundary layers. In this case we will obtain non-trivial initial layer equations. On the other hand, as in Section \ref{Sec.5} the boundary layer equations are stationary and are the same equations we constructed in Section \ref{Sec.4}. Finally, in Subsections \ref{subs5.3} and \ref{subs5.4} we will summarize the initial boundary value problem that we have obtained and we will construct the initial-boundary layer equations that we have to consider in order to describe the behavior of the solution in a small neighborhood of the boundary for times $ t\ll1 $. 
\subsection{Outer problems}\label{subs5.1}
We consider equation \eqref{rtetimeps} in the case $ c=1 $ and under the assumption $ \ell_M=\eps $ for the different choices of $ \ell_A=\eps^{-\beta} $ and $ \ell_S=\eps^{-\gamma} $. 
Expanding $ I_\nu $ according to \eqref{expansion1} and identifying in \eqref{rtetimeps} all terms of the same order we conclude as we computed in Section \ref{Sec.4} and Section \ref{Sec.5} that the first order $ \phi_0(t,x,n,\nu) $ of the intensity $ I_\nu $ is isotropic and the first non-isotropic term is of order $ \eps^1 $. Moreover, as long as $ \ell_T\ll L $ we have $ \phi_0(t,x,\nu)=B_\nu(T(t,x)) $. The outer problems in the case $ \ell_T\leq 1 $, i.e. $ \tau_{\mathit{h}}=\frac{1}{\eps} $ are given
\begin{enumerate}[(i)]
\item for $ \ell_M=\ell_T\ll\ell_S $ by \begin{equation*}\label{outer.1}
	\partial_tT(t,x)+4\pi\sigma \partial_tT^4(t,x)-\frac{4\pi}{3}\Div\left(\int_0^\infty \frac{\nabla_x B_\nu(T(t,x))}{\alpha_\nu(x)}d\nu\right)=0,
\end{equation*}
\item for $ \ell_M=\ell_T=\ell_S $ by \begin{equation*}\label{outer.2}
	\begin{split}
		\partial_t&T(t,x)+4\pi\sigma \partial_tT^4(t,x)\\&=\Div \left(\int_0^\infty d\nu\frac{1}{\alpha_\nu^a(x)+\alpha_\nu^s(x)}\left(\int_\Ss dn\; n\otimes \left(Id-A_{\nu,x}\right)^{-1}(n)\right)\nabla_x B_\nu(T(t,x))\right),
	\end{split}	
\end{equation*}
\item for $ \ell_M\ll\ell_T\ll L $ by \begin{equation*}\label{outer.3}
	\partial_tT(t,x)+4\pi\sigma \partial_tT^4(t,x)=\Div \left(\int_0^\infty d\nu\frac{1}{\alpha_\nu^s(x)}\left(\int_\Ss dn\; n\otimes \left(Id-H\right)^{-1}(n)\right)\nabla_x B_\nu(T(t,x))\right),
\end{equation*}
\item for $ \ell_M\ll L=\ell_T $ by \begin{equation}\label{outer.4}
	\begin{cases}
		\partial_t\phi_0(t,x,\nu)-\Div\left(\frac{1}{\alpha_\nu^s(x)}\left(\fint_\Ss n\otimes\left(Id-H\right)^{-1}(n)\;dn\right)\nabla_x\phi_0(t,x,\nu)\right)\\\;\;\;\;\;\;\;\;\;\;\;\;\;\;\;\;\;\;\;\;\;\;\;\;\;\;\;\;\;\;\;\;\;\;\;\;\;\;\;\;\;\;\;\;\;\;\;\;\;\;\;\;\;\;\;\;\;=\alpha_\nu^a(x)\left(B_\nu\left(T(t,x)\right)-\phi_0(t,x,\nu)\right) \\
		\partial_t T(t,x)+4\pi\int_0^\infty d\nu\;\alpha_\nu^a(x)\left(B_\nu(T(t,x))-\phi_0(t,x,\nu)\right)=0.
	\end{cases}
\end{equation}
\end{enumerate}
 In the case $ \ell_T \gg 1 $, i.e. $ \tau_{\mathit{h}}=\ell_A=\eps^{-\beta} $ for $ \beta>1 $, a similar computation to the one for the derivation of the problem \eqref{outerc5} yields 
\begin{equation}\label{outer.5}
	\begin{cases}
		\Div\left(\frac{1}{\alpha_\nu^s(x)}\left(\fint_\Ss n\otimes\left(Id-H\right)^{-1}(n)\;dn\right)\nabla_x\phi_0(t,x,\nu)\right)=0\\
		\partial_t T(t,x)+4\pi\int_0^\infty d\nu\;\alpha_\nu^a(x)\left(B_\nu(T(t,x))-\phi_0(t,x,\nu)\right)=0.
	\end{cases}
\end{equation}
Indeed, in the first equation of \eqref{rtetimeps} the leading order of the term containing the time derivative of $ I_\nu $ is of power $ \eps^0 $ as the emission-absorption term. On the other hand, the leading order $ \phi_0 $ of the radiation intensity is isotropic and the first non-isotropic term is of order $ \eps^1 $. Therefore, the identification in the first equation of \eqref{rtetimeps} of the terms of order $ \eps^{1-\beta}\gg \eps^0 $ gives the stationary equation in \eqref{outer.5} solved by $ \phi_0 $. Finally, plugging the first equation of \eqref{rtetimeps} into the second one yields the equation for the temperature as in \eqref{outer.5}.
\subsection{Initial layer equations and boundary layer equations}\label{subs5.2}
In this subsection we will describe the initial layers and the boundary layers appearing for time scales smaller than the heat parameter $ \tau_h $ and for regions close to the boundary, respectively. We start with the initial layers and we will see that similarly as for the boundary layers considered in Sections \ref{Sec.4} and \ref{Sec.5} there are two nested initial layers appearing. Indeed, in a first layer, i.e. for a very small time scale, the radiation intensity becomes isotropic, while in a second initial layer it becomes eventually the Planck distribution for the temperature. We will denote the first layer as initial Milne layer and the second one as initial thermalization layer, due to their analogy with the boundary layers considered in Sections \ref{Sec.4} and \ref{Sec.5}. We will also see that while the initial Milne layer appears for every choice of $ \ell_A $ and $ \ell_S $, the initial thermalization layer coincides with the initial Milne layer (if $ \ell_M=\ell_T $), appears after the initial Milne layer (if $ \ell_M\ll \ell_T\ll L $) or it is not present at all (if $ \ell_T\gtrsim L $).  

We recall that under the assumption $ \ell_A=\eps^{-\beta} $ and $ \ell_S=\eps^{-\gamma} $ for $ \min\{\beta,\gamma\}=-1 $ equation \eqref{rtetimeps} writes
\begin{equation}\label{rtetimeps2}
	\begin{cases}
		\partial_tI_\nu(t,x,n)+\tau_{\mathit{h}} n\cdot \nabla_x I_\nu(t,x,n)=\alpha_\nu^a(x)\eps^\beta \tau_{\mathit{h}}\left(B_\nu(T(t,x))-I_\nu(t,x,n)\right)\\\text{\phantom{nel mezzo del cammin}}+\alpha_\nu^s(x)\eps^\gamma \tau_{\mathit{h}}\left(\int_\Ss K(n,n')I_\nu(t,x,n')\;dn'-I_\nu(t,x,n)\right)&x\in\bnd,n\in\Ss,t>0\\
		\partial_tT+\partial_t \left(\int_0^\infty d\nu\int_\Ss dn \;I_\nu(t,n,x)\right)+\tau_{\mathit{h}}\Div\left(\int_0^\infty d\nu\int_\Ss dn \;nI_\nu(t,n,x)\right)=0&x\in\bnd,n\in\Ss,t>0\\
		I_\nu(0,x,n)=I_0(x,n,\nu)&x\in\bnd,n\in\Ss\\
		T(0,x)=T_0(x)&x\in\bnd\\
		I_\nu(t,n,x)=g_\nu(t,n)&x\in\bnd,n\cdot n_x<0,t>0.		
	\end{cases}
\end{equation}
Notice that the leading term in the first equation is of order $ \frac{\eps}{\tau_{\mathit{h}}} $. Therefore under a time rescaling $ t=\frac{\eps}{\tau_{\mathit{h}}} \tau $ the first equation writes
\begin{equation*}\label{rtetime3}
	\partial_\tau I_\nu=\eps^{\beta+1}\alpha_\nu^a(x)\left(B_\nu(T)-I_\nu\right)+\eps^{\gamma+1}\alpha_\nu^s(x)\left(\int_\Ss K(n,n')I_\nu\;dn'-I_\nu\right)+\eps n\cdot \nabla_x I_\nu
\end{equation*}
while the second one is 
\begin{equation*}\label{rtetime4}
		\partial_\tau T+\partial_\tau \left(\int_0^\infty d\nu\int_\Ss dn \;I_\nu\right)+\eps\Div\left(\int_0^\infty d\nu\int_\Ss dn \;nI_\nu\right)=0.
\end{equation*}
It is hence easy to see that for any choice of $ \ell_M $ and $ \ell_S $ there is an initial layer with thickness of order $ \frac{\eps}{\tau_{\mathit{h}}} $. Notice that as long as $ \ell_T\lesssim 1 $ (i.e. $ \tau_{\mathit{h}}=\eps^{-1} $) this initial layer has thickness of order $ \eps^2 $, while in the case $ \ell_T\gg 1 $ (i.e. $ \tau_{\mathit{h}}=\eps^{-\beta} $ for $ \beta>1 $) the order is $ \eps^{1+\beta} $. This layer plays the role of the Milne boundary layer in the time dependent case, as in this layer the radiation intensity becomes isotropic. For this reason we will denote it as the initial Milne layer.
\begin{enumerate}[(i)]
\item In the case $ \ell_M=\ell_T\ll \ell_S $ the initial Milne layer is described by the following initial Milne equation for the leading order of the radiation intensity
\begin{equation}\label{time.layer.1}
	\begin{cases}
		\partial_\tau \varphi_0(\tau,x,n,\nu)=\alpha^a_\nu(x)\left(B_\nu(T(\tau,x))-\varphi_0(\tau,x,n,\nu)\right)&\tau>0\\
		\partial_\tau T(\tau,x)+\int_0^\infty d\nu\int_\Ss dn\;\alpha^a_\nu(x)\left(B_\nu(T(\tau,x))-\varphi_0(\tau,x,n,\nu)\right)=0 &\tau>0\\
		\varphi_0(0,x,n,\nu)=I_0(x,n,\nu)\\
		T(0,x)=T_0(x).
	\end{cases}
\end{equation}
This equation plays the same role of the Milne problem and we expect $ T\to T_\infty $ and $ \varphi_0\to B_\nu(T_\infty) $ as $ \tau\to\infty $. Indeed, given a bounded solution to the equation \eqref{time.layer.1}, assuming $ T_\infty(x)=\lim\limits_{\tau\to\infty}T(\tau,x) $ and using simple ODE's arguments we have
\begin{equation}\label{limit.1}
	\varphi_0(\tau,x,n,\nu)=I_0e^{-\alpha_\nu^a(x)\tau}+\int_0^\tau \alpha_\nu^a(x)e^{-\alpha_\nu^a(x)(\tau-s)}B_\nu(T(s,x)\;ds\underset{\tau\to\infty}{\longrightarrow}B_\nu(T_\infty(x)).
\end{equation}
%
\item We turn now to the case $ \ell_M=\ell_T=\ell_S\ll L $. The initial Milne equation is
\begin{equation}\label{time.layer.2}
	\begin{cases}
		\partial_\tau \varphi_0(\tau,x,n,\nu)=\alpha^a_\nu(x)\left(B_\nu(T(\tau,x))-\varphi_0(\tau,x,n,\nu)\right)\\\text{\phantom{nel mezzo del cammin}}+\alpha_\nu^s(x)\left(\int_\Ss K(n,n')\varphi_0(\tau,x,n',\nu)\;dn'-\varphi_0(\tau,x,n,\nu)\right)&\tau>0\\
		\partial_\tau T(\tau,x)+\int_0^\infty d\nu\int_\Ss dn\;\alpha^a_\nu(x)\left(B_\nu(T(\tau,x))-\varphi_0(\tau,x,n,\nu)\right)=0 &\tau>0\\
		\varphi_0(0,x,n,\nu)=I_0(x,n,\nu)\\
		T(0,x)=T_0(x).
	\end{cases}
\end{equation}
Again, assuming $ T_\infty(x)=\lim\limits_{\tau\to\infty}T(\tau,x) $ for a bounded solution to \eqref{time.layer.2} we can write an explicit formula for $ \varphi_0 $ and we also obtain 
\begin{equation*}\label{varphi0.1}
	\begin{split}	\varphi_0(\tau,x,n,\nu)=&I_0e^{-(\alpha_\nu^a(x)+\alpha_\nu^s(x))\tau}+\int_0^\tau\alpha_\nu^a(x)e^{-(\alpha_\nu^a(x)+\alpha_\nu^s(x))(\tau-s)}B_\nu(T(s,x))\;ds\\
		&+\int_0^\tau\alpha_\nu^s(x)e^{-(\alpha_\nu^a(x)+\alpha_\nu^s(x))(\tau-s)}H[\varphi_0](\tau,x,n,\nu)\\
		=&e^{-(\alpha_\nu^a(x)+\alpha_\nu^s(x))\tau}\sum_{n=0}^\infty \frac{(\alpha_\nu^s(x)\tau)^n}{n!}H^n[I_0](x,n,\nu)\\&+\int_0^\tau\alpha_\nu^a(x)e^{-\alpha_\nu^a(x)(\tau-s)}B_\nu(T(s,x))\;ds\\
		&\underset{\tau\to\infty}{\longrightarrow}B_\nu(T_\infty(x)).
	\end{split}
\end{equation*}
\item For the case $ \ell_M\ll\ell_T\ll L$ similarly as for the boundary layers we expect the solution to the initial Milne layer equation to become isotropic but not necessarily to become the Planck distribution. In this case the initial Milne equation is 
\begin{equation}\label{time.layer.3}
	\begin{cases}
		\partial_\tau \varphi_0(\tau,x,n,\nu)=\alpha_\nu^s(x)\left(\int_\Ss K(n,n')\varphi_0(\tau,x,n',\nu)\;dn'-\varphi_0(\tau,x,n,\nu)\right)&\tau>0\\
		\partial_\tau T(\tau,x)=0 &\tau>0\\
		\varphi_0(0,x,n,\nu)=I_0(x,n,\nu)\\
		T(0,x)=T_0(x).
	\end{cases}
\end{equation}
On one hand we have $ T(\tau,x)=T_0(x) $ for all $ \tau>0 $, on the other hand we have $$ \varphi_0(\tau,x,n,\nu)=\exp(-\alpha_\nu^s(x)\tau (Id-H))I_0 .$$ Using standard spectral theory for the compact self-adjoint operator $ H\in\mathcal{L}\left(L^2(\Ss),L^2(\Ss)\right) $ we see that the greatest eigenvalue of $ H $ is $ 1 $ with eigenfunctions being the constants. Hence, an application of the spectral gap theory and of the continuous functional calculus (cf. \cite{reedsimon1}) yields the limit $$  \lim\limits_{\tau\to\infty}\varphi_0(\tau,x,n,\nu)=\varphi(x,\nu) ,$$
where $ \varphi $ is independent of $ n\in\Ss $. Moreover, $ \varphi(x,\nu)=\fint_\Ss I_0(x,n,\nu) $. Indeed, integratin over $ \Ss $ the first equation of \eqref{time.layer.3} we obtain using that $ \int_\Ss K(n,n')dn=1 $ the equation
\begin{equation*}
	\begin{cases}
		\partial_\tau \fint_\Ss\varphi_0(\tau,x,n,\nu)dn=0&\tau>0\\
		\fint_\Ss\varphi_0(0,x,n,\nu)dn=\fint_\Ss I_0(x,n,\nu)dn.
	\end{cases}
\end{equation*}
Hence, we conclude by the isotropy of $ \varphi $ $$\fint_\Ss I_0(x,n,\nu)dn= \fint_\Ss\varphi_0(\tau,x,n,\nu)dn\underset{\tau\to\infty}{\longrightarrow}\varphi(x,\nu) .$$
The study of the Milne initial layer described by \eqref{time.layer.3} has been rigorously studied in the context of the one-speed neutron transport equation in \cite{papanicolaou} and in \cite{Lei3d}, i.e when $ \alpha^s_\nu $ is independent of $ \nu $. While in \cite{papanicolaou} the behavior of the neutron distribution for small times is analyzed for general kernels using stochastic methods, in \cite{Lei3d} equation \eqref{time.layer.3} is solved for a very specific scattering kernel, namely the constant kernel $ K=\frac{1}{4\pi} $.\\

Moreover, there is also an initial thermalization layer. Indeed, under the rescaling $ t=\eps^{1-\beta}\tau $ for $ \beta~\in~(-1,1) $, $ \gamma=-1 $ and therefore $ \tau_{\mathit{h}}=\frac{1}{\eps} $ equation \eqref{rtetimeps2} becomes
\begin{equation}\label{time.layer.3-1}
	\begin{cases}
		\partial_\tau I_\nu(\tau,x,n,)+\eps^{-\beta}n\cdot \nabla_xI_\nu(\tau,x,n)=\alpha_\nu^a(x)\left(B_\nu(T(\tau,x))-I_\nu(\tau,x,n)\right)\\\;\;\;\;\;+\frac{\alpha_\nu^s(x)}{\eps^{1+\beta}}\left(\int_\Ss K(n,n')I_\nu(\tau,x,n')\;dn'-I_\nu(\tau,x,n)\right)&\tau>0\\
		\partial_\tau T(\tau,x)+\int_0^\infty d\nu\int_\Ss dn\;\partial_\tau I_\nu(\tau,x,n)+\eps^{-\beta}\Div\left(\int_0^\infty d\nu\int_\Ss dn\;I_\nu(\tau,x,n)n\right)=0 &\tau>0\\
	\end{cases}
\end{equation}
As we have seen several times, the leading order $ \varphi_0 $ of $ I_\nu $ in \eqref{time.layer.3-1} is isotropic. Moreover, for $ \beta\geq0 $ also the term of order $ \eps^\beta $ is isotropic. Hence, the initial thermalization layer equation for the leading order of the radiation intensity is given by is 
\begin{equation}\label{time.layer.3-2}
	\begin{cases}
		\partial_\tau \varphi_0(\tau,x,\nu)=\alpha_\nu^a(x)\left(B_\nu(T(\tau,x))-\varphi_0(\tau,x,\nu)\right)&\tau>0\\
		\partial_\tau T(\tau,x)+\int_0^\infty d\nu\int_\Ss dn\;\partial_\tau \varphi_0(\tau,x,\nu)=0 &\tau>0\\
		\varphi_0(0,x,n,\nu)=\varphi(x,\nu)=\fint_\Ss I_0(x,n,\nu)dn\\
		T(0,x)=T_0(x).
	\end{cases}
\end{equation}
As for equation \eqref{time.layer.1} arguing as in \eqref{limit.1} we expect $ \varphi_0(\tau,x,\nu)\to B_\nu(T_\infty(x)) $ as $ \tau\to\infty $ denoting by $ T_\infty(x)=\lim\limits_{\eta\to\infty}T(\tau,x) $.
\item[(iv)+(v)] Finally, in both cases $ \ell_M\ll \ell_T=L $ and $ \ell_M\ll L\ll \ell_T $, i.e. in the non-equilibrium diffusion case, we observe the formation of only the initial Milne layer in which the radiation intensity becomes isotropic. In both cases the initial Milne layer equation is once again \eqref{time.layer.3}.
\end{enumerate}
We study now the boundary layers. We notice that in \eqref{rtetimeps2} $ \partial_t I_\nu $ has relative order $ \tau_{\mathit{h}}^{-1} $ compared to $ n\cdot \nabla_x I_\nu $. Therefore, any rescaling of the space variable by $ \xi=-\frac{x-p}{\eps^\alpha}\cdot n_p $ for $ \eps^\alpha\in\{\ell_M=\eps,\ell_T\}\ll L $ and $ p\in\bnd $ yields the boundary layer equations constructed in Section \ref{4.2}. Indeed, under such procedure the system becomes
\begin{equation}\label{rtetimepsbnd}
	\begin{cases}
		\frac{\eps^\alpha}{\tau_{\mathit{h}}}\partial_tI_\nu(t,\xi,n;p)-(n\cdot n_p)\partial_\xi I_\nu(t,\xi,n;p)=\alpha_\nu^a(p+\mathcal{O}({\eps^\alpha}))\eps^{\beta+\alpha}\left(B_\nu(T(t,\xi;p))-I_\nu(t,\xi,n;p)\right)\\\text{\phantom{nel mezzo del cammin di nost}}+\alpha_\nu^s(p+\mathcal{O}({\eps^\alpha}))\eps^{\gamma+\alpha}\left(\int_\Ss K(n,n')I_\nu(t,\xi,n';p)\;dn'-I_\nu(t,\xi,n;p)\right)\\
		\partial_tT(t,\xi;p)+\partial_t \left(\int_0^\infty d\nu\int_\Ss dn \;I_\nu(t,\xi,n;p)\right)-\eps^{-\alpha}\tau_{\mathit{h}}\partial_\xi\left(\int_0^\infty d\nu\int_\Ss dn \;(n\cdot n_p)I_\nu(t,\xi,n;p)\right)=0\\
		I_\nu(0,\xi,n;p)=I_0(\xi,n,\nu;p)\\
		T(0,\xi)=T_0(\xi)\\
		I_\nu(t,0,n;p)=g_\nu(t,n)\text{\phantom{nel mezzo del cammin di nostra vita mi ritrovai in una selva osc}}n\cdot n_p<0.		
	\end{cases}
\end{equation}
Under these rescalings we obtain namely the Milne problems \eqref{Milnecase1} for $ \ell_M=\ell_T\ll\ell_S $ and \eqref{Milnecase2} for $ \ell_M=\ell_T=\ell_S\ll L $. In the case $ \ell_M\ll \ell_T\ll L $ there are two boundary layers appearing described by the Milne problem \eqref{Milnecase3} and by the thermalization equation \eqref{themaleq2.3}. Finally, if $ \ell_M\ll L \lesssim\ell_T $ the Milne boundary layer is described by \eqref{Milnecase3}.

\subsection{Limit problems in the bulk}\label{subs5.3}
We summarize now the PDEs which are expected to be solved by the solution of \eqref{rtetimeps} in the limit $ \ell_M=\eps\to0 $ for any different choice of $ \ell_T $ as the speed of light is finite, i.e. $ c=1 $. 
\begin{enumerate}[(i)]
\item In the case when $ \ell_M=\ell_T\ll \ell_S $, the limit problem is given  by 
\begin{equation*}\label{final.1}
	\begin{cases}
		\partial_tT(t,x)+4\pi\sigma \partial_tT^4(t,x)-\frac{4\pi}{3}\Div\left(\int_0^\infty \frac{\nabla_x B_\nu(T(t,x))}{\alpha_\nu(x)}d\nu\right)=0& t>0,\;x\in\Omega\\
		T(0,x)=T_\infty(x)&x\in\Omega\\
		T(t,x)=\lim\limits_{y\to\infty}\left(\int_0^\infty \alpha_\nu^a(p)I_\nu(t,y,n;p)\right)&p\in\bnd,
	\end{cases}
\end{equation*}
where $ I_\nu(t,y,n;p) $ is the solution to the Milne problem \eqref{Milnecase1} for the boundary value $ g_\nu(t,n) $ and\linebreak $ T_\infty(x)=\lim\limits_{\tau\to\infty}T(\tau,x) $ is defined as the limit of the solution to the initial layer \eqref{time.layer.1}.
\item If $ \ell_M=\ell_T=\ell_S\ll L $, i.e. $ \ell_S=\ell_A=\eps $ and $ \tau_{\mathit{h}}=\eps^{-1} $, the limit problem that describes the temperature in the interior of $ \Omega $ for positive times is
\begin{equation*}\label{final.2}
	\begin{cases}
		\partial_tT(t,x)+4\pi\sigma \partial_tT^4(t,x)\\\text{\phantom{nel mezzo del}}=\Div \left(\int_0^\infty d\nu\frac{1}{\alpha_\nu^a(x)+\alpha_\nu^s(x)}\left(\int_\Ss dn\; n\otimes \left(Id-A_{\nu,x}\right)^{-1}(n)\right)\nabla_x B_\nu(T(t,x))\right)& t>0,\;x\in\Omega\\
		T(0,x)=T_\infty(x)=\lim\limits_{\tau\to\infty}T(\tau,x)&x\in\Omega\\
		T(t,x)=\lim\limits_{y\to\infty}\left(\int_0^\infty \alpha_\nu^a(p)I_\nu(t,y,n;p)\right)&p\in\bnd,
	\end{cases}
\end{equation*}
where $ I_\nu(t,y,n;p) $ is the solution to the Milne problem \eqref{Milnecase2} for the boundary value $ g_\nu(t,n) $ and $ T(\tau,x) $  the solution to the initial layer \eqref{time.layer.2}.
\item We move now to the case $ \ell_M\ll \ell_T\ll L $, hence we consider $ \ell_S=\eps $ and $ \ell_A=\eps^{-\beta} $ for $ \beta\in(-1,1) $ and $ \tau_{\mathit{h}}=\eps^{-1} $. The limit problem is
\begin{equation*}\label{final.3}
	\begin{cases}
		\partial_tT(t,x)+4\pi\sigma \partial_tT^4(t,x)\\\text{\phantom{nel mezzo del}}=\Div \left(\int_0^\infty d\nu\frac{1}{\alpha_\nu^s(x)}\left(\int_\Ss dn\; n\otimes \left(Id-H\right)^{-1}(n)\right)\nabla_x B_\nu(T(t,x))\right)& t>0,\;x\in\Omega\\
		T(0,x)=T_\infty(x)=\lim\limits_{\tau\to\infty}T(\tau,x)&x\in\Omega\\
		T(t,p)=\lim\limits_{y\to\infty}\left(\int_0^\infty d\nu\; \alpha_\nu^a(p)\varphi_0(t,\eta,\nu;p)\right)&p\in\bnd,
	\end{cases}
\end{equation*}
where $ T(\tau,x) $ solves the initial layer \eqref{time.layer.3-2} and $ \varphi_0 $ is the solution to the thermalization problem \eqref{themaleq2.3}.
\item If $ \ell_M\ll L=\ell_T$ the limit problem is 
\begin{equation}\label{equilibtriumtime4.1}
	\begin{cases}
		\partial_t\phi_0(t,x,\nu)-\Div\left(\frac{1}{\alpha_\nu^s(x)}\left(\fint_\Ss n\otimes\left(Id-H\right)^{-1}(n)\;dn\right)\nabla_x\phi_0(t,x,\nu)\right)\\\text{\phantom{nel mezzo del cammin di nostra vita mi ritrovai}}=\left(B_\nu\left(T(t,x)\right)-\phi_0(t,x,\nu)\right)& x\in\Omega,\;t>0\\
		\partial_t T(t,x)+4\pi\int_0^\infty d\nu\;\alpha_\nu^a(x)\left(B_\nu(T(t,x))-\phi_0(t,x,\nu)\right)=0& x\in\Omega,\;t>0\\
		\phi(0,x,\nu)=\varphi(x,\nu)=\fint_\Ss I_0(x,n,\nu)dn\\
		T(0,x)=T_0(x)& x\in\Omega\\		\phi_0(t,p,\nu)=\lim\limits_{y\to\infty}\fint_\Ss I_\nu(t,y,n,p)& p\in\bnd,\;t>0,\\
	\end{cases}
\end{equation}
where $ I_\nu(t,y,n,p) $ solves the Milne problem \eqref{Milnecase3} for the boundary value $ g_\nu(t,n) $ and\linebreak $ \varphi(p,\nu)=\lim\limits_{\tau\to\infty}\varphi_0(\tau,p,n,\nu) $ for the solution to \eqref{time.layer.3}.
\item Finally, if $ \ell_M\ll L\ll\ell_T$ the limit problem is with the same notation as in \eqref{equilibtriumtime4.1}
\begin{equation}\label{equilibriumtime5.1}
	\begin{cases}
		\Div\left(\frac{1}{\alpha_\nu^s(x)}\left(\fint_\Ss n\otimes\left(Id-H\right)^{-1}(n)\;dn\right)\nabla_x\phi_0(t,x,\nu)\right)=0& x\in\Omega,\;t>0\\
		\partial_t T(t,x)+4\pi\int_0^\infty d\nu\;\alpha_\nu^a(x)\left(B_\nu(T(t,x))-\phi_0(t,x,\nu)\right)=0& x\in\Omega,\;t>0\\
		T(0,x)=T_0(x)& x\in\Omega\\		\phi_0(t,p,\nu)=\lim\limits_{y\to\infty}\fint_\Ss I_\nu(t,y,n,p)& p\in\bnd,\;t>0.\\
	\end{cases}
\end{equation}
\end{enumerate}
\subsection{Initial-boundary layers}\label{subs5.4}
We conclude Section \ref{Sec.6} considering the initial-boundary layer equations, which can be found studying \eqref{rtetimepsbnd}. This equation shows that on the one hand under the space rescale $ \xi=-\frac{x-p}{\eps^\alpha}\cdot n_p $ for $ p\in\bnd $ and $ \eps^\alpha\in\{\ell_M,\ell_T\} $ the time derivative term $ \partial_tI_\nu $ becomes of the same order of $ \partial_\xi I_\nu $ rescaling the time by $ t=\frac{\eps^{\alpha}}{\tau_{\mathit{h}}}\tau $, on the other hand it becomes of the same order of the absorption-emission term if we consider $ t=\frac{\tau}{\eps^\beta \tau_{\mathit{h}}} $. It is not difficult to see that rescaling the space variable according to the Milne length $ \ell_M=\eps $ we obtain a non-trivial equation of the leading order of $ I_\nu $ in both time and space variables only rescaling the time by $ t=\frac{\eps}{\tau_{\mathit{h}}}\tau $. In the case $ \ell_M\ll \ell_T\ll L $, i.e. when $ \ell_S=\eps $ and $ \ell_A=\eps^{-\beta} $ with $ \beta\in(-1,1) $ and $ \tau_{\mathit{h}}=\frac{1}{\eps} $, a thermalization layer also appears. It is described for small times and for $ x\in\Omega $ close to $ \bnd $ by the equation obtained rescaling the space variable by $ \ell_T=\eps^{\frac{1-\beta}{2}} $ and the time variable in a suitable way so that the resulting equation is non-trivial in both variables. This is the case when $ t=\eps^{1-\beta}\tau $. 
\begin{enumerate}[(i)]
\item If $ \ell_M=\ell_T\ll\ell_S $, i.e. if $ \beta=-1 $ and $ \gamma>-1 $ and $ \tau_{\mathit{h}}=\eps^{-1} $, rescaling the spatial variable by $ y=-\frac{x-p}{\eps}\cdot n_p $ for $ p\in\bnd $ and under the time rescaling $ t=\eps^{2}\tau $ we obtain the initial-boundary layer equation
\begin{equation*}\label{time-boundary.1final.1}
	\begin{cases}
		\partial_\tau I_\nu(\tau,y,n;p)-(n\cdot n_p)\partial_y I_\nu(\tau,y,n;p)= \alpha^a_\nu(p)(B_\nu(T(\tau,y))-I_\nu(\tau,y,n;p))&y>0,\;\tau>0\\
		\partial_tT(\tau,y)+\int_0^\infty d\nu \int_\Ss dn\;\partial_\tau I_\nu(\tau,y,n;p)-\partial_y \left(\int_0^\infty d\nu \int_\Ss dn\;(n\cdot n_p)I_\nu(\tau,y,n;p)\right)=0& y>0,\;\tau>0\\
		I_\nu(0,y,n;p)=I_0(p,n,\nu)& y>0\\
		T(0,y;p)=T_0(p)& y>0\\
		I_\nu(\tau,0,n;p)=g_\nu(0,n)&n\cdot n_p<0, \tau>0.	\end{cases}
\end{equation*}
\item In the case $ \ell_M=\ell_T=\ell_s $ we rescale again the variables according to $ y=-\frac{x-p}{\ell_M}\cdot n_p $ for $ p\in\bnd $ and $ t=\eps^{2}\tau $ and we obtain the following initial-boundary layer equation
\begin{equation*}\label{time-boundary.2final.1}
	\begin{cases}
		\partial_\tau I_\nu(\tau,y,n;p)-(n\cdot n_p)\partial_y I_\nu(\tau,y,n;p)= \alpha^a_\nu(p)(B_\nu(T(\tau,y))-I_\nu(\tau,y,n;p))\\\text{\phantom{nel mezzo del cammin di nostra}}+\alpha_\nu^s\left(\int_\Ss K(n,n')I_\nu(\tau,y,n';p)\;dn'-I_\nu(\tau,y,n;p)\right)&y>0,\;\tau>0\\
		\partial_\tau T(\tau,y)\int_0^\infty d\nu \int_\Ss dn\;\partial_\tau I_\nu(\tau,y,n;p)-\partial_y \left(\int_0^\infty d\nu \int_\Ss dn\;(n\cdot n_p)I_\nu(\tau,y,n;p)\right)=0& y>0,\;\tau>0\\
		I_\nu(0,y,n;p)=I_0(p,n,\nu)& y>0\\
		T(0,y;p)=T_0(p)& y>0\\
		I_\nu(\tau,0,n;p)=g_\nu(0,n)&n\cdot n_p<0,\tau>0.	\end{cases}
\end{equation*}
\item If $ \ell_M\ll \ell_T\ll L $ there are two initial-boundary layers appearing. In order to find the initial-boundary layer equation describing the transition from $ T_\infty $ to $ \lim\limits_{y\to\infty}\left(\int_0^\infty d\nu\; \alpha_\nu^a(p)\varphi_0(t,\eta,\nu;p)\right) $ we rescale first the space variable according to $ \eta=\frac{x-p}{\ell_T}\cdot n_p $ for $ p\in\bnd $ with $ \ell_T=\eps^{\frac{1-\beta}{2}} $ and the time variable according to $ t=\eps^{1-\beta}\tau $ and following the same computations as we did in Section \ref{Sec.5} in equation \eqref{timec-boundary.3} we obtain the initial-boundary layer equation \begin{equation*}\label{time-boundary.5final.1}
	\begin{cases}
		\partial_\tau \varphi_0(\tau,\eta,\nu;p)-\frac{1}{\alpha_\nu^s(p)}\left(\fint_\Ss (n\cdot n_p)(Id-H)^{-1}(n)\cdot n_p \;dn\right)\partial_\eta^2\varphi_0(\tau,\eta,\nu;p)\\\text{\phantom{nel mezzo del cammin di nostra vita mi rit}}=\alpha_\nu^a(p)\left(B_\nu(T(\tau,\eta;p))-\varphi_0(\tau,\eta,\nu;p)\right)& \eta>0,\;\tau>0\\
		\partial_\tau T(\tau,y;p)+ \int_0^\infty d\nu\int_\Ss dn\;\alpha_\nu^a(p)\left(B_\nu(T(\tau,\eta;p))-\varphi_0(\tau,\eta,\nu;p)\right) =0&\eta>0,\;\tau>0\\
		\varphi_0(0,\eta,\nu;p)=\varphi(p,\nu)=\fint_\Ss I_0(p,n,\nu)dn&\eta>0\\
		T(0,\eta;p)=T_0(p)&\eta>0\\
		\varphi_0(\tau,0,\nu;p)=I(0,\nu;p)& n\cdot n_p<0,\tau>0	,
	\end{cases}
\end{equation*}
where we used $ I(0,\nu;p)=\lim\limits_{y\to\infty} I_\nu(0,y,n;p) $ for the solution to the Milne problem \eqref{Milnecase3} and also` $ \varphi(p,\nu)=\lim\limits_{\tau\to\infty}\varphi_0(\tau,p,n,\nu) $ for the solution to \eqref{time.layer.3}.\\

Rescaling now both space and time variables according to $ y=\frac{x-p}{\eps}\cdot n_p $ for $ p\in\bnd $ and $ t=\eps^2\tau $ we obtain another initial-boundary layer equation which explains the transition from $ I(0,\nu;p) $ to $ \varphi(p,\nu) $. This is given by the following equation
\begin{equation}\label{time-boundary.3final.1}
	\begin{cases}
		\partial_\tau I_\nu(\tau,y,n;p)-(n\cdot n_p)\partial_y I_\nu(\tau,y,n;p)=\\\text{\phantom{nel mezzo del cammin di}}+\alpha_\nu^s\left(\int_\Ss K(n,n')I_\nu(\tau,y,n';p)\;dn'-I_\nu(\tau,y,n;p)\right)&y>0,\;\tau>0\\
		\partial_\tau T(\tau,y)=0& y>0,\;\tau>0\\
		I_\nu(0,y,n;p)=I_0(p,n,\nu)& y>0\\
		T(0,y;p)=T_0(p)& y>0\\
		I_\nu(\tau,0,n;p)=g_\nu(0,n)&n\cdot n_p<0,\tau>0.	\end{cases}
\end{equation} 
\item[(iv)+(v)]If $ \ell_T\gtrsim L $ under the rescaling $ y=\frac{x-p}{\eps}\cdot n_p $ for $ p\in\bnd $ and $ t=\eps^2\tau $ we obtain the problem \eqref{time-boundary.3final.1} as initial-boundary layer equation.
\item[(v)] Moreover, in the case $ \ell_T\gg L $ we notice in equation \eqref{equilibriumtime5.1} that the leading order $ \phi_0 $ of the radiation intensity solves a stationary equation. The transition from the solution of a time dependent equation, as the one of the original problem, to the solution of a stationary equation happens in times of order $ \eps^{\beta-1} $. Indeed, under a time rescaling $ t=\eps^{\beta-1}\tau=\frac{\tau}{\tau_{\mathit{h}} \eps} $ we obtain the following equation solved by the leading order $ \phi_0 $ in the bulk
\begin{equation}\label{timetostationary}
	\begin{cases}
		\partial_\tau\phi_0(\tau,x,\nu)-\Div\left(\frac{1}{\alpha_\nu^s(x)}\left(\fint_\Ss n\otimes\left(Id-H\right)^{-1}(n)\;dn\right)\nabla_x\phi_0(\tau,x,\nu)\right)=0& x\in\Omega,\;\tau>0\\
		\partial_\tau T(\tau,x)=0& x\in\Omega,\;\tau>0\\
		\phi(0,x,\nu)=\varphi(x,\nu)=\fint_\Ss I_0(x,n,\nu)dn& x\in\Omega\\
		T(0,x)=T_0(x)& x\in\Omega\\		\phi_0(\tau,p,\nu)=\lim\limits_{y\to\infty}\fint_\Ss I_\nu(\tau,y,n,p)& p\in\bnd,\;\tau>0,\\
	\end{cases}
\end{equation}
where $ \varphi(x,\nu) $ is defined by the initial layer equation \eqref{time.layer.3}. This equation can be derived in the same way as the outer problem \eqref{outer.5} taking into account that under this time scale the term containing $ \partial_\tau I_\nu $ is of order $ \eps^{1-\beta}\gg \eps^0 $. Moreover, also the second equation in \eqref{rtetimeps} gives $ \partial_\tau T=0 $ since the absorption emission terms are of order $ \eps^0\ll\eps^{1-\beta} $.
\end{enumerate}
\section{Time dependent diffusion approximation. The case of non-dimen\-sional speed of light scaling as a power law of the Milne length}\label{Sec.7}
In this last section we repeat all the procedures used in Sections \ref{Sec.4}, \ref{Sec.5} and \ref{Sec.6} and we construct the limit problem solved by the solution of the time dependent equation \eqref{rtetimeps} when $ \ell_M=\eps\to 0 $ and in the case in which the speed of light is a power-law of the form $ c=\eps^{-\kappa} $ for $ \kappa>0 $. The strategy is the same as in Section \ref{Sec.6}. It will turn out that the limit problems valid at the interior of the domain $ \Omega $ and for positive times are the same as the one we found in the case of infinite speed of light. On the other hand, differently from the case of infinite speed of light, in this case time layers appears also in regions far from the boundary. Similarly as in Section \ref{Sec.5} and \ref{Sec.6} the boundary layer equations are stationary and are the same equations constructed in Section \ref{Sec.4}. Finally, we will summarize the initial boundary value problems that we have obtained and we will construct the initial-boundary layer equations that we have to consider in order to describe the behavior of the solution for small times in regions close to the boundary. 

\subsection{Outer problems}
We consider equation \eqref{rtetimeps} in the case $ c=\eps^{-\kappa} $, $ \kappa>0 $. In order to find the outer problems solved in the limit we proceed as we did in the previous three sections.
It turns out that the outer problems are the same evolution equations obtaind for the infinite speed of light case. Indeed, under the assumption $ c=\eps^{-\kappa} $ and $ \ell_A=\eps^{-\beta} $, $ \ell_S=\eps^{-\gamma} $ with $ \min\{\alpha,\gamma\}=-1 $ equation \eqref{rtetimeps} becomes
\begin{equation}\label{rtetimepskappa}
	\begin{cases}
		\eps^\kappa\partial_tI_\nu(t,x,n)+\tau_{\mathit{h}} n\cdot \nabla_x I_\nu(t,x,n)=\alpha_\nu^a(x)\eps^\beta \tau_{\mathit{h}}\left(B_\nu(T(t,x))-I_\nu(t,x,n)\right)\\\text{\phantom{nel mezzo del cammin di}}+\alpha_\nu^s(x)\eps^\gamma \tau_{\mathit{h}}\left(\int_\Ss K(n,n')I_\nu(t,x,n')\;dn'-I_\nu(t,x,n)\right)&x\in\Omega,n\in\Ss,t>0\\
		\partial_tT+\eps^\kappa\partial_t \left(\int_0^\infty d\nu\int_\Ss dn \;I_\nu(t,n,x)\right)+\tau_{\mathit{h}}\Div\left(\int_0^\infty d\nu\int_\Ss dn \;nI_\nu(t,n,x)\right)=0&x\in\Omega,n\in\Ss,t>0\\
		I_\nu(0,x,n)=I_0(x,n,\nu)&x\in\Omega,n\in\Ss\\
		T(0,x)=T_0(x)&x\in\Omega\\
		I_\nu(t,n,x)=g_\nu(t,n)&x\in\bnd,n\cdot n_x<0,t>0.		
	\end{cases}
\end{equation}
Then, plugging the usual expansion \eqref{expansion1} for $ I_\nu $ into equation \eqref{rtetimepskappa} and identifying all terms of the same power of $ \eps $ give the same results as in Section \ref{Sec.5}. This is due to the fact that in the first equation of \eqref{rtetimepskappa} the term involving the time derivative of the radiation intensity is of order $ \eps^\kappa$ and hence it is much smaller than $ \eps^0\ll\eps^{-1}\ll \tau_{\mathit{h}}\eps^{-1}  $, i.e. the orders of magnitude which lead to the resulting first two terms in the expansion $ I_\nu(t,x,n)=\phi_0(t,x,\nu)+\eps\phi_1(t,x,n,\nu)+\cdots $. As we noticed in the previous sections, $ \phi_0 $ is isotropic and as long as $ \ell_T\ll L $ it is the Planck distribution $ B_\nu(T) $. Since also in the second equation of \eqref{rtetimepskappa} the leading term containing $ \partial_tT$ is of order $ 1 $, the term $ \eps^\kappa  \partial_t \int_0^\infty d\nu\int_\Ss dn \;I_\nu$ is negligible. The outer problems are hence as in Section \ref{Sec.5} equation \eqref{outerc1} for $ \ell_M=\ell_T\ll\ell_S $, equation \eqref{outerc2} for $ \ell_M=\ell_T=\ell_S $, equation \eqref{outerc3} for $ \ell_M\ll\ell_T\ll L $, the system \eqref{outerc4} for $ \ell_M\ll L=\ell_T $ and the system \eqref{outerc5} for $ \ell_M\ll L\ll\ell_T $.

\subsection{Initial layer equations and boundary layer equations}\label{subs6.2}
In contrast to Section \ref{Sec.5} (i.e. the case $ c=\infty $) besides the formation of boundary layers also time layers appear. The equations describing them can be obtained similarly as in Section \ref{Sec.6}. The first equation in \eqref{rtetimepskappa} has leading order $  \tau_{\mathit{h}}\eps^{-1} $, hence a time rescaling $ t=\frac{\eps^{1+\kappa}}{\tau_{\mathit{h}}} \tau$ gives
\begin{equation*}\label{rtetime3.1}
	\partial_\tau I_\nu=\eps^{\beta+1}\alpha_\nu^a(x)\left(B_\nu(T)-I_\nu\right)+\eps^{\gamma+1}\alpha_\nu^s(x)\left(\int_\Ss K(n,n')I_\nu\;dn'-I_\nu\right)-\eps n\cdot \nabla_x I_\nu
\end{equation*}
and
\begin{equation*}\label{rtetime4.1}
	\eps^{-\kappa}\partial_\tau T+\partial_\tau \left(\int_0^\infty d\nu\int_\Ss dn \;I_\nu\right)+\eps\Div\left(\int_0^\infty d\nu\int_\Ss dn \;nI_\nu\right)=0,
\end{equation*}
which implies $ \partial_\tau T=0 $ at the leading order. Hence, an initial layer of thickness of order $ \frac{\eps^{1+\kappa}}{\tau_{\mathit{h}}}  $ is appearing for any choice of $ \ell_A $ and $ \ell_S $. This is the so called initial Milne layer.
\begin{enumerate}[(i)]
\item If $ \ell_M=\ell_T\ll \ell_S $ the initial Milne layer is described by \begin{equation}\label{initial.1.1+}
	\begin{cases}
		\partial_\tau \varphi_0(\tau,x,n,\nu)=\alpha_\nu^a(x)\left(B_\nu(T_0(x))-\varphi_0(\tau,x,n\nu)\right)& \text{ if } \tau>0\\
		\varphi_0(0,x,n,\nu)=I_0(x,n,\nu).
	\end{cases}	
\end{equation}
Therefore, as $ \tau\to\infty $ we obtain using a simple ODE argument $ \lim\limits_{\tau\to\infty}\varphi_0(\tau,x,n,\nu)= B_\nu(T_0(x)) $.
\item In the case $ \ell_M=\ell_T=\ell_S$ and hence $ \tau_{\mathit{h}}=\frac{1}{\eps} $ with the scaling $ t=\tau \eps^{2+\kappa} $ we obtain on one hand $ \partial_\tau T=0 $ and on the other hand for the first order $ \varphi_0 $ the identity
\begin{multline*}\label{initial.1.2}
	\partial_\tau \varphi_0(\tau,x,n,\nu)=\alpha_\nu^a(x)\left(B_\nu(T_0(x))-\varphi_0(\tau,x,n\nu)\right)\\+\alpha_\nu^s(x)\left(\int_\Ss K(n,n')\varphi_0(\tau,x,n',\nu)\;dn'-\varphi_0(\tau,x,n,\nu)\right).
\end{multline*}
Again, using semigroup theory we can write the solution as $$ \varphi_0=e^{-\alpha_\nu^a(x)\tau} \left(e^{-\alpha_\nu^s\tau(Id-H)}I_0\right)+\left(1-e^{\alpha_\nu^a(x)\tau}\right)B_\nu(T_0) .$$ Hence, we have once more $ \lim\limits_{\tau\to\infty}\varphi_0(\tau,x,n,\nu)= B_\nu(T_0(x)) $.
\item For all cases $ \ell_M\ll\ell_T\ll L $, i.e. $ \ell_S=\eps $ and $ \ell_A=\eps^{-\beta} $ for $ \beta\in(-1,1) $ and $\tau_{\mathit{h}}=\frac{1}{\eps}$, under the scaling $ t=\tau \eps^{2+\kappa} $ we have the initial Milne layer equation \begin{equation}\label{initial.1.3}
	\begin{cases}
		\partial_\tau \varphi_0(\tau,x,n,\nu)=\alpha_\nu^s(x)\left(\int_\Ss K(n,n')\varphi_0(\tau,x,n',\nu)\;dn'-\varphi_0(\tau,x,n,\nu)\right)&\tau>0\\
		\partial_\tau T(\tau,x)=0&\tau>0\\
		\varphi_0(0,x,n,\nu)=I_0(x,n,\nu)\\
		T(0,x)=T_0(x).
	\end{cases}
\end{equation}
This is exactly the same equation as \eqref{time.layer.3}. Thus, an application of spectral theory implies again\\ $ \lim\limits_{\tau\to\infty}\varphi_0(\tau,x,n,\nu)=\varphi(x,\nu)=\fint_\Ss I_0(x,n,\nu)dn $.\\

However, as for the finite speed of light case, there is also a thermalization layer appearing. Indeed, with a time rescaling $ t=\eps^{1-\beta+\kappa}\tau $ the term involving $ \partial_tI_\nu $ becomes of the same order of the emission-absorption term according to
\begin{equation}\label{initial.2.3}
	\begin{cases}
		\partial_\tau I_\nu(\tau,x,n)+\eps^{-\beta}n\cdot \nabla_xI_\nu(\tau,x,n)=\alpha_\nu^a(x)\left(B_\nu(T(x))-I_\nu(\tau,x,n)\right)\\\text{\phantom{nel mezzo del cammin di nostra vita}}+\frac{\alpha_\nu^s(x)}{\eps^{1+\beta}}\left(\int_\Ss K(n,n')I_\nu(\tau,x,n')\;dn'-I_\nu(\tau,x,n)\right)\\
		\frac{1}{\eps^\kappa}\partial_\tau T(\tau,x)+\left(\int_0^\infty d\nu\int_\Ss dn\;\partial_\tau I_\nu(\tau,x,n)\right)+\eps^{-\beta}\Div\left(\int_0^\infty d\nu\int_\Ss dn\;nI_\nu(\tau,x,n)\right)=0.
	\end{cases}
\end{equation}
Hence, as we have seen in \eqref{time.layer.3-1} the leading order $ \varphi_0 $ of $ I_\nu $ in \eqref{initial.2.3} is isotropic, as well as the term of order $ \eps^\beta $ for $ \beta\geq 0 $. Moreover, once more the temperature $ T $ is just the initial temperature $ T_0 (x)$ to the leading order. This yields the initial thermalization layer equation
\begin{equation*}\label{initial.3.1}
	\begin{cases}
		\partial_\tau \varphi_0(\tau,x,\nu)=\alpha_\nu^a(x)\left(B_\nu(T_0(x))-\varphi_0(\tau,x,\nu)\right)&\tau>0\\
		\varphi_0(0,x,\nu)=\varphi(x,\nu) .
	\end{cases}
\end{equation*}
Hence, similarly to \eqref{initial.1.1+} we have $ \lim\limits_{\tau\to\infty}\varphi_0=B_\nu(T_0(x)) $ as $ \tau \to \infty $.
\item[(iv)+(v)] For the cases $ \ell_M\ll L\lesssim\ell_T $ the initial Milne layer equation is obtained again rescaling the time variable by $ t=\frac{\eps^{1+\kappa}}{\tau_{\mathit{h}}} \tau $ and it is given by equation \eqref{initial.1.3}.
\end{enumerate}
\
For the boundary layer equations we argue similarly as in the case $ c=\infty $ and $ c $ bounded. Rescaling the space variable by $ \xi=-\frac{x-p}{\eps^\alpha}\cdot n_p $ for $ \eps^\alpha\in\{\ell_M,\ell_T\} $ and $ p\in\bnd $ equation \eqref{rtetimepskappa} becomes
\begin{equation}\label{time.boundary.2}
	\begin{cases}
		-(n\cdot n_p)\partial_\xi I_\nu(t,\xi,n;p)= \alpha^a_\nu\left(p+\mathcal{O}(\eps^\alpha)\right)\frac{\eps^{\alpha}}{\ell_A}(B_\nu(T)-I_\nu)\\\;\;\;\;\;\;\;\;\;+\frac{\eps^{\alpha}}{\ell_S}\alpha_\nu^s\left(p+\mathcal{O}(\eps^\alpha)\right)\left(\int_\Ss K(n,n')I_\nu\;dn'-I_\nu\right)-\frac{\eps^{\alpha+\kappa}}{\tau_{\mathit{h}}}\partial_t I_\nu(t,\xi,n)+\eps^{2\alpha}\cdots &\xi>0\\
		\partial_t T(t,\xi;p)+\eps^\kappa\left(\int_0^\infty d\nu\int_\Ss dn\partial_t I_\nu\right)+\eps^{-\alpha}\tau_{\mathit{h}}\Div \left(\int_0^\infty d\nu \int_\Ss dn\;nI_\nu\right)=0& \xi>0\\
		I_\nu(0,\xi,n;p)=I_0(p,n,\nu)\\
		T(0,\xi;p)=T_0(p)\\
		I_\nu(t,0,n;p)=g_\nu(t,n)& n\cdot n_p<0.	\end{cases}
\end{equation}
Therefore, the boundary layers are described by the same stationary equation we constructed in Section \ref{Sec.4}. Indeed we obtain for $ \ell_M=\ell_T\ll\ell_S $ the Milne problem \eqref{Milnecase1} and for $ \ell_M=\ell_T=\ell_S\ll L $ the Milne problem \eqref{Milnecase2}. The two boundary layers appearing in the case $ \ell_M\ll \ell_T\ll L $ are described by the Milne problem \eqref{Milnecase3} and by the thermalization equation \eqref{themaleq2.3}. Finally, if $ \ell_M\ll L\lesssim\ell_T $ the Milne problems are given by \eqref{Milnecase3}.  \\

\subsection{Limit problems in the bulk}
We now summarize the PDE problems which are expected to be solved by the solution of \eqref{rtetimeps} when $ c=\eps^{-\kappa} $, $ \kappa>0 $ in the limit $ \ell_M=\eps\to0 $ for any choice of $ \ell_A $ and $ \ell_S $. Matching the solution to the outer problems valid in the bulk for positive times $ t>0 $ with the solution to the initial layer equations and boundary layer equations we obtain as limit equation exactly the same PDE problems in Section \ref{Sec.5}. Indeed, on one hand the boundary layer problems are exactly the Milne and thermalization problems constructed for the stationary problem and valid also for the time dependent problem. On the other hand, in the initial layer equations derived in the previous subsection \ref{subs6.2} the temperature is constant, hence it is $ T=T_0 $, the same result we obtained in the case $ c=\infty $ in Subsection \ref{4.2}. Therefore, since the outer problems coincides in both cases when $ c=\infty $ and $ c=\eps^{-\kappa} $ with $ \kappa>0 $ and $ \eps\to0 $, we conclude as in Section \ref{Sec.5} that the limit PDE problems are given by \eqref{equilibtriumtime1.5} if $ \ell_M=\ell_T\ll \ell_S $, by \eqref{equilibtriumtime2.5} if $ \ell_M=\ell_T=\ell_S $, by \eqref{equilibriumtime3.5} if $ \ell_M\ll \ell_T\ll L $, by \eqref{equilibtriumtimec4} if $ \ell_M\ll L=\ell_T$ and finally by \eqref{equilibtriumtimec5} if $\ell_M\ll L\ll \ell_T $. \\
\subsection{Initial-boundary layers}
As in Sections \ref{Sec.5} and \ref{Sec.6} we will derive the initial-boundary layer equations, which describe the behavior of the solutions for very small times and in regions close to the boundary. The initial-boundary layer equations are obtained rescaling in a suitable way the space and time variables. Considering equation \eqref{time.boundary.2} resulting from the space rescale according to the Milne length or the thermalization length we notice that the term involving the time derivative of the radiation intensity has order $ \frac{\eps^{\alpha+\kappa}}{\tau_{\mathit{h}}} $. Hence, the initial-boundary Milne layer equation is obtained by the rescaling $ y=-\frac{x-p}{\eps}\cdot n_p $ and $ t=\frac{\eps^{1+\kappa}}{\tau_{\mathit{h}}}\tau $ for $ p\in\bnd $. In the case $ \ell_M\ll \ell_T\ll L $ (i.e. when $ \ell_A=\eps^{\beta} $ for $ \beta\in(-1,1) $, $ \ell_S=\eps $ and $ \tau_{\mathit{h}}=\frac{1}{\eps} $) the initial-boundary thermalization equation is obtained rescaling $ \eta=-\frac{x-p}{\ell_T}\cdot n_p $ and $ t=\eps^{1-\beta+\kappa} $, where $ \ell_T=\eps^{-\frac{1-\beta}{2}} $ and $ p\in\bnd $.
\begin{enumerate}[(i)]
	\item If $ \ell_M=\ell_T\ll\ell_S $ rescaling the spatial variable by $ y=-\frac{x-p}{\eps}\cdot n_p $ for $ p\in\bnd $ and the time variable by $ t=\eps^{2+\kappa}\tau $ we see that the initial-boundary layer equation is given by
	\begin{equation*}\label{time-boundary.1final}
		\begin{cases}
			\partial_\tau I_\nu(\tau,y,n;p)-(n\cdot n_p)\partial_y I_\nu(\tau,y,n;p)= \alpha^a_\nu(p)(B_\nu(T_0(p))-I_\nu(\tau,y,n;p))&y>0,\;\tau>0\\		
			I_\nu(0,y,n;p)=I_0(p,n,\nu)&y>0,\;\tau>0\\
			I_\nu(\tau,0,n;p)=g_\nu(0,n)&n\cdot n_p<0,\;\tau>0,	\end{cases}
	\end{equation*}
	where we used that equation
	\begin{equation}\label{important.mixed}
		\begin{cases}
			\frac{1}{\eps^\kappa}\partial_tT(\tau,y)+\partial_\tau\left(\int_0^\infty d\nu \int_\Ss dn\;I_\nu(\tau,y,n;p)\right)=\partial_y \left(\int_0^\infty d\nu \int_\Ss dn\;(n\cdot n_p) I_\nu(\tau,y,n;p)\right)& y>0,\;\tau>0\\
			T(0,y;p)=T_0(p)& y>0,\\
		\end{cases}
	\end{equation}
	implies $ T(\tau,y;p)=T_0(p) $.
	\item If $ \ell_M=\ell_T=\ell_S $ rescaling the variables according to $ y=-\frac{x-p}{\eps}\cdot n_p $ for $ p\in\bnd $ and $ t=\eps^{2+\kappa}\tau $ we obtain the following initial-boundary layer equation
	\begin{equation*}\label{time-boundary.2final}
		\begin{cases}
			\partial_\tau I_\nu(\tau,y,n;p)-(n\cdot n_p)\partial_y I_\nu(\tau,y,n;p)= \alpha^a_\nu(p)(B_\nu(T_0(p))-I_\nu(\tau,y,n;p))\\\;\;\;\;\;\;\;\;\;\;\;\;\;\;\;\;\;\;\;\;\;\;\;\;+\alpha_\nu^s\left(\int_\Ss K(n,n')I_\nu(\tau,y,n';p)\;dn'-I_\nu(\tau,y,n;p)\right)&y>0,\;\tau>0\\
			I_\nu(0,y,n;p)=I_0(p,n,\nu)&y>0\\
			I_\nu(\tau,0,n;p)=g_\nu(0,n)&n\cdot n_p<0,\;\tau>0,	\end{cases}
	\end{equation*}
	where we used \eqref{important.mixed} again.
	\item If $ \ell_M\ll\ell_T\ll L $ there are again two different initial-boundary layers. We consider first the thermalization problem. We hence rescale the space variable according to $ \eta=\frac{x-p}{\eps^{\frac{1-\beta}{2}}}\cdot n_p $ for $ p\in\bnd $ and the time variable according to $ t=\eps^{\kappa+1-\beta}\tau $ and following the same computations as we did in Section \ref{Sec.5} in equation \eqref{timec-boundary.3} and using a similar argument as in \eqref{important.mixed} we obtain the initial-boundary layer equation as
	\begin{equation*}\label{time-boundary.5final}
		\begin{cases}
			\partial_\tau \varphi_0(\tau,\eta,\nu;p)-\frac{1}{\alpha_\nu^s(p)}\left(\fint_\Ss (n\cdot n_p)(Id-H)^{-1}(n)\cdot n_p \;dn\right)\partial_\eta^2\varphi_0(\tau,\eta,\nu;p)\\\text{\phantom{nel mezzo del cammin di nostra vita mi ritro }}=\alpha_\nu^a(p)\left(B_\nu(T_0(p))-\varphi_0(\tau,\eta,\nu;p)\right)& \eta>0,\;\tau>0\\
			\varphi_0(0,\eta,\nu;p)=\varphi(p,\nu)& \eta>0\\
			\varphi_0(\tau,0,\nu;p)=I(0,\nu;p)& p\in\bnd,\;\tau>0	,
		\end{cases}
	\end{equation*}
	where $ I(0,\nu;p)=\lim\limits_{y\to\infty} I_\nu(0,y,n;p) $ for the solution to the Milne problem \eqref{Milnecase3} for the boundary value $ g_\nu(t,n) $ and $ \varphi(p,\nu)=\lim\limits_{\tau\to\infty}\varphi_0(\tau,p,n,\nu) $ for the solution to \eqref{initial.1.3}.\\
	
	As we have seen in Section \ref{Sec.6} there is another initial-boundary value equation which describes the transition from the initial value $ \varphi(x,\nu) $ to the boundary value $ I(0,\nu;p) $. This is obtained rescaling the space variable according to $ y=\frac{x-p}{\eps}\cdot n_p $ for $ p\in\bnd $ and the time variable according to $ t=\eps^{\kappa+2}\tau $. Using $ \eqref{important.mixed} $ we obtain hence 
	\begin{equation*}\label{time-boundary.5final.2}
		\begin{cases}
			\partial_\tau I_\nu(\tau,y,n;p)-(n\cdot n_p)\partial_y I_\nu(\tau,y,n;p)\\\text{\phantom{nel mezzo del cammin di }}=\alpha_\nu^s\left(\int_\Ss K(n,n')I_\nu(\tau,y,n';p)\;dn'-I_\nu(\tau,y,n;p)\right)&y>0,\;\tau>0\\
			I_\nu(0,y,n;p)=I_0(p,n,\nu)&y>0\\
			I_\nu(\tau,0,n;p)=g_\nu(0,n)&n\cdot n_p<0,\;\tau>0.
		\end{cases}
	\end{equation*}
	\item In the case $ \ell_M\ll\ell_T=L $ rescaling $ \eta=\frac{x-p}{\eps}\cdot n_p $ for $ p\in\bnd $ and $ t=\eps^{2+\kappa}\tau $ we obtain also the initial boundary layer equation for this case
	\begin{equation}\label{time-boundary.4final}
		\begin{cases}
			\partial_\tau I_\nu(\tau,y,n;p)-(n\cdot n_p)\partial_y I_\nu(\tau,y,n;p)=\\\text{\phantom{nel mezzo del ca}}+\alpha_\nu^s\left(\int_\Ss K(n,n')I_\nu(\tau,y,n';p)\;dn'-I_\nu(\tau,y,n;p)\right)&y>0,\;\tau>0\\
			I_\nu(0,y,n;p)=I_0(p,n,\nu)&y>0\\
			I_\nu(\tau,0,n;p)=g_\nu(0,n)&n\cdot n_p<0,\;\tau>0.	\end{cases}
	\end{equation}
	Similar to the case where $ \ell_T\gg 1 $ and $ c=1 $ in Section \ref{Sec.6}, we notice that the radiation intensity $ I_\nu $ has a transition from a solution of a time dependent equation, as it was in the original problem \eqref{rtetimeps}, to a solution of a stationary equation, as it is in \eqref{equilibtriumtimec4}. This transition takes place at times of order $ \eps^\kappa $. Indeed, under the time rescaling $ t=\eps^\kappa \tau$ we obtain the following equation for the leading order $ \phi_0 $ of $ I_\nu $ for all $ x\in\Omega $
	\begin{equation}\label{timetostationary1}\hspace{-0.7cm}
		\begin{cases}
			\partial_\tau\phi_0(\tau,x,\nu)-\Div\left(\frac{1}{\alpha_\nu^s(x)}\left(\fint_\Ss n\otimes\left(Id-H\right)^{-1}(n)\;dn\right)\nabla_x\phi_0(\tau,x,\nu)\right)\\\text{\phantom{nel mezzo del cammin di nostra vita mi ritrovai }}=\left(B_\nu\left(T(\tau,x)\right)-\phi_0(\tau,x,\nu)\right)& x\in\Omega,\;\tau>0\\
			\partial_\tau T(\tau,x)=0& x\in\Omega,\;\tau>0\\
			\phi(0,x,\nu)=\varphi(x,\nu)& x\in\Omega\\
			T(0,x)=T_0(x)& x\in\Omega\\		\phi_0(\tau,p,\nu)=\lim\limits_{y\to\infty}\fint_\Ss I_\nu(0,y,n,p)& p\in\bnd,\;\tau>0,\\
		\end{cases}
	\end{equation}
	where $ I_\nu(0,y,n,p) $ solves the Milne problem \eqref{Milnecase3} for the boundary value $ g_\nu(0,n) $ and we used the notation $ \varphi(x,\nu)=\lim\limits_{\tau\to\infty}\varphi_0(\tau,x,n,\nu) $ for the solution to \eqref{initial.1.3}. In order to derive equation \eqref{timetostationary1} we notice that under the time rescale $ t=\eps^\kappa \tau$ the term in the first equation of \eqref{rtetimeps} containing $ \partial_\tau I_\nu $ becomes of order $ \eps^0 $ as the absorption-emission term. This implies the first equation in \eqref{timetostationary1} as we did in Section \ref{Sec.6} for \eqref{outer.4}. On the other hand, in the second equation of \eqref{timetostationary1} the leading term is $ \partial_\tau T $ of order $ \eps^{-\kappa}\gg \eps^0 $.
	\item Finally, if $ \ell_M\ll L\ll \ell_T $ the initial-boundary layer equation is again \eqref{time-boundary.4final}. Also for this last case we notice the leading order $ \phi_0 $ of $ I_\nu $, which solves a time-dependent equation \eqref{rtetimeps}, solves in the limit a stationary equation \eqref{equilibtriumtimec5}. The transition from time-dependent solution to stationary solution takes place at time of order $ \eps^{\beta-1+\kappa} $. Under the time rescale $ t=\eps^{\beta-1+\kappa}\tau $ we derive in the same way as for equation \eqref{timetostationary} the equation solved by $ \phi_0 $ in the bulk describing this transition. It turns out that it is exactly given by \eqref{timetostationary} for the initial condition $ \phi(0,x,\nu)=\varphi(x,\nu) $ given by the solution to \eqref{initial.1.3}.
\end{enumerate}
\section{Concluding remarks}
In this paper we considered the problem of describing the temperature distribution in a body where the heat is transported only by radiation. We considered the case where the mean free path of the radiative process tends to zero, i.e. $ \ell_M\to 0 $.
Therefore, we coupled the radiative transfer equation \eqref{RTE} with the energy balance equation \eqref{div-free} and we studied the diffusion approximation for the time dependent equations \eqref{rtetimeps} and \eqref{rtetimepsc} and the stationary equation \eqref{rtestationaryeps}.

For all different scaling limit regimes using the method of asymptotic expansions we derived the full limit models describing the temperature of the body and the radiation intensity. The resulting models have been classified depending on the form of the radiation intensity at the leading order on the bulk of the domain. The cases where the isotropic leading order of the radiation intensity is given by the Planck distribution for the temperature yield the diffusion equilibrium approximation, while the models in which the radiation intensity is not approximated by the Planck distribution are denoted by diffusion non-equilibrium approximation. Notice that the diffusion approximation is valid only on the bulk of the domain $ \Omega $ where the leading order of the radiation intensity is isotropic. On the other hand, at the boundary layers and at the initial layers the diffusion approximation fails. We also described for each considered case the boundary and initial layers appearing. Moreover, a summary of the available results about the diffusion approximation and the boundary layer problem for similar settings is included. Many of the derived problems in this article have to be still studied. 

For the time dependent problem we studied three different cases. First we analyzed the problem for the speed of light assumed to be $ c=\infty $, i.e. when the transport of radiation can be assumed to be instantaneous. We then considered the case where the speed of light is of order $ 1 $, i.e. when the time used by the light for spanning distances of order $ 1 $ is of the same order of the time needed by the temperature for having meaningful changes. Finally, we studied the case where the speed of light scales as a power law of the Milne length, i.e. $ c=\eps^{-\kappa} $ for $ \kappa>0 $ and $ \ell_M=\eps $.
\begin{appendices}
	\section{Proof of Proposition \ref{prop.constant}}\label{Appendix}
	We prove now Proposition \ref{prop.constant}. To this end we need the following auxiliary Lemma.
	\begin{lemma}\label{ergodicity}
		Let $ K\in C\left(\Ss\times\Ss\right) $, invariant under rotations, non-negative and satisfying $$ \int_\Ss K(n,n')dn=1 .$$ Let $ n,\omega\in\Ss $. Then there exists finitely many $ n_1,\cdots,n_N\in\Ss $ such that
		\begin{equation}\label{ergodicity.eq}
			K(n_{i-1},n_i)>0 \text{ for all }i\in\{1,\cdots,N+1\},
		\end{equation}
		where we defined $ n_0=n $ and $ n_{N+1}=\omega $.
		\begin{proof}
			Since $ K\geq 0 $ but it is not equal zero, there exists a pair $ n',n''\in\Ss $ such that $ K(n',n'')>0 $. Hence, applying the rotation $ \Rot_{n,n'} $ yields the existence of $ n_* $ such that $ K(n,n_*)>0 $. By continuity the set $ B_n=\{\tilde{n}\in\Ss:K(n,\tilde{n})>0\} $ is open. Hence, there exists $ \delta>0 $ and $ n_1\in\Ss $ such that $ B_\delta(n_1)\subset B_n $. We remark that $ \delta>0 $ is independent of the choice of $ n\in\Ss $. Indeed, for any $ n'\in\Ss $ there exists some $ n''\in\Ss $ such that $ B_\delta(n'')\subset B_{n'} $. This is a consequence of the invariance under rotations of $ K $. Indeed, it is not difficult to see that $ \Rot_{n,n'}({B_n})= B_{n'} $ and so $ \Rot_{n,n'}(B_\delta(n_1))=B_\delta(\Rot_{n,n'}(n_1))\subset B_{n'} $. 
			
			Let us consider the set $$ A_n=\{\tilde{n}\in\Ss: \text{ there exists }n_1,\cdots,n_N\in\Ss \text{ such that \eqref{ergodicity.eq} holds for }n_0=n \text{ and }n_{N+1}=\tilde{n}\}. $$ By the previous consideration we know that $ A_n $ is not empty. We claim now that $ B_\delta(n')\subset A_n $ for any $ n'\in A_n $. Indeed, let $ \delta>0 $ as above. Since $ n'\in A_n $, then $ B_{n'} $ is not empty and there exists some $ n_1\in\Ss $ such that $ B_\delta(n_1)\subset B_{n'} $. It is easy to see that $ n_1\in A_n $. Let now $ \tilde{n}\in B_\delta(n') $, then $ \Rot_{n',n_1}(\tilde{n})\in B_\delta(n_1) $. Hence $ K(n',\Rot_{n',n_1}(\tilde{n}))=K(n_1,\tilde{n})>0 $. Since also $K(n',n_1)>0$ we conclude that $ B_\delta(n')\subset A_n $ for all $ n'\in A_n $. Hence, $ A_n $ is open and it is the whole sphere $ \Ss $. Indeed, assume $ A_n \ne \Ss $. Then, since $ A_n $ is open, the boundary $ \partial A_n=\overline{A}_n\setminus A_n $ is not empty. Let $ n^*\in\partial A_n $ and let $ n_0\in A_n $ with $ d(n^*,n_0)<\frac{\delta}{3} $, where $ d(n^*,n_0) $  is the distance on $ \Ss $ between the two points $ n^*,n_0\in\Ss $. Since $ n^*\in \partial A_n $ it is true that
			$$ B_{\frac{\delta}{3}}(n^*)\cap A_n \ne \emptyset \text{ and }  B_{\frac{\delta}{3}}(n^*)\cap A^c_n \ne \emptyset. $$
			On the other hand we know that $ B_\delta(n_0)\subset A_n $ and therefore $$ B_{\frac{\delta}{3}}(n^*)\subset B_{\frac{\delta}{2}}(n^*)\subset B_\delta(n_0)\subset A_n. $$
			This contradiction concludes the proof of Lemma \ref{ergodicity}.
		\end{proof}
	\end{lemma}
	
	\begin{proof}[Proof of Proposition \ref{prop.constant}]
		We first show that $ \varphi $ is continuous. Let $ \eps>0 $. By the continuity of the kernel $ K $ there exists some $ \delta>0 $ such that $$ \left|K(n_1,n_1')-K(n_2,n_2')\right|<\frac{\eps}{4\pi\Arrowvert \varphi\Arrowvert_\infty} $$
		for all $ n_1,n_2,n_1',n_2'\in\Ss $ with $ d(n_1,n_2)+d(n_1',n_2')<\delta $. Let hence $ n_1,n_2\in\Ss $ with $ d(n_1,n_2)<\delta $ then it is easy to see that $ \varphi $ is continuous since
		\begin{equation*}
			\begin{split}
				\left|\varphi(n_1)-\varphi(n_2)\right|=&\left|H[\varphi](n_1)-H[\varphi](n_2)\right|\\\leq& \int_\Ss \left|K(n_1,n')-K(n_2,n')\right||\varphi(n')|\;dn'<\eps.
			\end{split}
		\end{equation*}
		
		We move now to the proof of claim $ (ii) $. Let $ M=\max_{n\in\Ss}(\varphi(n)) $. By continuity there exists some $ n_* \in\Ss $ such that $ M=\varphi(n_*) $. We define the set $ A_M=\left\{n\in\Ss:\: \varphi(n)=M\right\} $. Thus, $ A_M $ is not empty and by continuity it is also closed. We claim that $ A_M $ is also open, which implies claim $ (ii) $. Let $ n\in A_M $. Consider $ B_n=\{\tilde{n}\in\Ss:K(n,\tilde{n})>0\} $. Let $ \eps>0 $ and $ B_n^\eps=\{\tilde{n}\in B_n:\varphi(\tilde{n})<M-\eps\} $. We show $\varphi(\tilde{n})=M$ for all $ \tilde{n}\in B_n $. It is easy to see that this is true if $ B_n^\eps=\emptyset $ for all $ \eps>0 $. If not, let $ \eps>0 $ so that $ B_n^\eps\ne\emptyset $. Then
		\begin{equation*}
			\begin{split}
				M=\varphi(n)=\int_{B_n^\eps} K(n,n')\varphi(n')dn'+\int_{\left(B_n^\eps\right)^c} K(n,n')\varphi(n')dn'<M-\eps\int_{B_n^\eps} K(n,n')dn'<M.
			\end{split}
		\end{equation*}
		
		Arguing as in the proof of Lemma \ref{ergodicity} there exists a $ \delta>0 $ such that $ B_\delta(n_0)\subset B_n $ for some $ n_0\in B_n $. Hence, using the same argument, since $ n_0\in A_M $ it is also true that $\varphi(\tilde{n})=M$ for all $ \tilde{n}\in B_{n_0} $. Using the rotation invariance of the kernel analogously as we have done in Lemma \ref{ergodicity} we see that $$ \Rot_{n,n_0}\left(B_\delta(n_0)\right)=B_\delta(n)\subset \Rot_{n,n_0}(B_n)=B_{n_0}\subset A_M . $$
		We have just proved that closed non-empty set $ A_M $ is open and hence it must be the whole sphere $ \Ss $.\\
		
		Finally we prove claim $ (iii) $. To this end we notice that the linear operator $ H $ maps $L^p$-functions to continuous bounded functions. Analogously as in the proof of $ (i) $, this is a direct consequence of the H\"older inequality and the fact that the scattering kernel $ K $ is continuous. Hence, $ (Id-H)_1:L^1(\Ss)\to L^1(\Ss) $ given by $ (Id-H)_1\varphi=\varphi-H[\varphi] $ is a well-defined operator which maps integrable functions to integrable functions. Since $ H[\varphi]\in C(\Ss) $ for any $ \varphi\in L^1(\Ss) $, if $ (Id-H)_1\varphi=0 $ then also $ (ii) $ applies and hence $ \varphi=\text{const} $. This means that the null space of $ (Id-H)_1 $ as an operator acting on $ L^1(\Ss) $ is given by $$\mathcal{N}\left((Id-H)_1\right)=\text{span}\langle 1\rangle=\{f=c:\; c\in\RR\}. $$ It is not difficult to see that the dual operator $ (Id-H)_1^*:L^\infty(\Ss)\to L^\infty(\Ss) $ is exactly given by $ (Id-H) $. Indeed, let $ f\in L^1(\Ss) $ and $ g\in L^\infty(\Ss) $. We compute using the invariance under rotations of the kernel $ K $
		\begin{equation*}
			\begin{split}
				\int_\Ss dn \;g(Id-H)_1[f]=&	\int_\Ss dn \;g(n)f(n)-	\int_\Ss dn\int_\Ss dn'\;K(n,n')g(n)f(n')\\=&\int_\Ss dn\; g(n)f(n)-\int_\Ss dn'\int_\Ss dn\;K(n',n)g(n)f(n')=\int_\Ss dn \;(Id-H)[g]f.
			\end{split}
		\end{equation*}
		Therefore, by the orthogonality of the null-space to the range of the dual operator we conclude
		\begin{equation*}
			\begin{split}
				\text{Ran}(Id-H)=&\left\{\varphi\in L^\infty(\Ss): \int_\Ss \varphi(n)f(n)\;dn=0\; \forall f\in\mathcal{N}\left((Id-H)_1\right)\right\}\\
				=&\left\{\varphi\in L^\infty(\Ss): \int_\Ss \varphi(n)\;dn=0\right\}.
			\end{split}
		\end{equation*}
	\end{proof}
\end{appendices}

\bibliographystyle{siam}
\bibliography{literaturebibliography}

\begin{thebibliography}{10}

\bibitem{Allaire1}
{\sc G.~Allaire and K.~El~Ganaoui}, {\em Homogenization of a conductive and
  radiative heat transfer problem}, Multiscale Model. Simul., 7 (2008),
  pp.~1148--1170.

\bibitem{Allaire2}
{\sc G.~Allaire and Z.~Habibi}, {\em Homogenization of a conductive,
  convective, and radiative heat transfer problem in a heterogeneous domain},
  SIAM J. Math. Anal., 45 (2013), pp.~1136--1178.

\bibitem{Allaire3}
\leavevmode\vrule height 2pt depth -1.6pt width 23pt, {\em Second order
  corrector in the homogenization of a conductive-radiative heat transfer
  problem}, Discrete Contin. Dyn. Syst. Ser. B, 18 (2013), pp.~1--36.

\bibitem{MilneBoltzmann4}
{\sc C.~Bardos, R.~E. Caflisch, and B.~Nicolaenko}, {\em The {M}ilne and
  {K}ramers problems for the {B}oltzmann equation of a hard sphere gas}, Comm.
  Pure Appl. Math., 39 (1986), pp.~323--352.

\bibitem{Golse3}
{\sc C.~Bardos, F.~Golse, and B.~Perthame}, {\em The radiative transfer
  equations: existence of solutions and diffusion approximation under
  accretivity assumptions---a survey}, in Proceedings of the conference on
  mathematical methods applied to kinetic equations ({P}aris, 1985), vol.~16,
  1987, pp.~637--652.

\bibitem{Golse6}
\leavevmode\vrule height 2pt depth -1.6pt width 23pt, {\em The {R}osseland
  approximation for the radiative transfer equations}, Comm. Pure Appl. Math.,
  40 (1987), pp.~691--721.

\bibitem{Erratum}
\leavevmode\vrule height 2pt depth -1.6pt width 23pt, {\em Erratum to the
  article: ``{T}he {R}osseland approximation for the radiative transfer
  equations'' [{C}omm. {P}ure {A}ppl. {M}ath. {\bf 40} (1987), no. 6, 691--721;
  {MR}0910950 (88j:35134)]}, Comm. Pure Appl. Math., 42 (1989), pp.~891--894.

\bibitem{BardosSantosSentis}
{\sc C.~Bardos, R.~Santos, and R.~Sentis}, {\em Diffusion approximation and
  computation of the critical size}, Trans. Amer. Math. Soc., 284 (1984),
  pp.~617--649.

\bibitem{papanicolaou}
{\sc A.~Bensoussan, J.-L. Lions, and G.~C. Papanicolaou}, {\em Boundary layers
  and homogenization of transport processes}, Publ. Res. Inst. Math. Sci., 15
  (1979), pp.~53--157.

\bibitem{Chandrasekhar}
{\sc S.~Chandrasekhar}, {\em Radiative transfer}, Dover Publications, Inc., New
  York, 1960.

\bibitem{compton}
{\sc K.~Compton}, {\em L{X}{X}{I}{I}{I}. {S}ome properties of resonance
  radiation and excited atoms}, The London, Edinburgh, and Dublin Philosophical
  Magazine and Journal of Science, 45 (1923), pp.~750--760.

\bibitem{Davison}
{\sc B.~Davison and J.~B. Sykes}, {\em Neutron transport theory}, Oxford, at
  the Clarendon Press,, 1957.

\bibitem{dematte}
{\sc E.~Demattè}, {\em On a kinetic equation describing the behavior of a gas
  interacting mainly with radiation}, Journal of Statistical Physics, 190
  (2023), p.~124.

\bibitem{dematt42024compactness}
{\sc E.~Demattè, J.~W. Jang, and J.~J.~L. Velázquez}, {\em Compactness and
  existence theory for a general class of stationary radiative transfer
  equations}.
\newblock \url{https://doi.org/10.48550/arXiv.2401.12828}, 2024.

\bibitem{dematte2023diffusion}
{\sc E.~Demattè and J.~J.~L. Velázquez}, {\em On the diffusion approximation
  of the stationary radiative transfer equation with absorption and emission}.
\newblock \url{https://doi.org/10.48550/arXiv.2309.11437 }, 2023.

\bibitem{MilneBoltzmann2}
{\sc R.~Esposito, Y.~Guo, R.~Marra, and L.~Wu}, {\em Ghost {E}ffect from
  {B}oltzmann {T}heory}.
\newblock \url{https://doi.org/10.48550/arXiv.2301.09427}, 2023.

\bibitem{MilneBoltzmann1}
\leavevmode\vrule height 2pt depth -1.6pt width 23pt, {\em Ghost {E}ffect from
  {B}oltzmann {T}heory: {E}xpansion with {R}emainder}.
\newblock \url{https://doi.org/10.48550/arXiv.2301.09560}, 2023.

\bibitem{MilneBoltzmann3}
{\sc R.~Esposito and R.~Marra}, {\em Stationary non equilibrium states in
  kinetic theory}, J. Stat. Phys., 180 (2020), pp.~773--809.

\bibitem{friedlander2000smoke}
{\sc S.~Friedlander}, {\em Smoke, {D}ust, and {H}aze: {F}undamentals of
  {A}erosol {D}ynamics}, Topics in chemical engineering, Oxford University
  Press, 2000.

\bibitem{masmoudi2}
{\sc M.~Ghattassi, X.~Huo, and N.~Masmoudi}, {\em On the diffusive limits of
  radiative heat transfer system {I}: {W}ell-prepared initial and boundary
  conditions}, SIAM J. Math. Anal., 54 (2022), pp.~5335--5387.

\bibitem{masmoudi1}
\leavevmode\vrule height 2pt depth -1.6pt width 23pt, {\em Diffusive limits of
  the steady state radiative heat transfer system: boundary layers}, J. Math.
  Pures Appl. (9), 175 (2023), pp.~181--215.

\bibitem{Larsen7}
{\sc A.~G. Gibbs and E.~W. Larsen}, {\em {Neutron transport in plane geometry
  with general anisotropic, energy‐dependent scattering}}, Journal of
  Mathematical Physics, 18 (2008), pp.~1998--2007.

\bibitem{Golse4}
{\sc F.~Golse}, {\em The {M}ilne problem for the radiative transfer equations
  (with frequency dependence)}, Trans. Amer. Math. Soc., 303 (1987),
  pp.~125--143.

\bibitem{Golse1}
{\sc F.~Golse, F.~Hecht, O.~Pironneau, D.~Smets, and P.-H. Tournier}, {\em
  Radiative transfer for variable three-dimensional atmospheres}, J. Comput.
  Phys., 475 (2023), pp.~Paper No. 111864, 19.

\bibitem{Golse2}
{\sc F.~Golse and O.~Pironneau}, {\em Radiative transfer in a fluid}, Rev. R.
  Acad. Cienc. Exactas F\'{\i}s. Nat. Ser. A Mat. RACSAM, 117 (2023), pp.~Paper
  No. 37, 17.

\bibitem{GuoL2}
{\sc Y.~Guo}, {\em Decay and continuity of the {B}oltzmann equation in bounded
  domains}, Arch. Ration. Mech. Anal., 197 (2010), pp.~713--809.

\bibitem{GuoLei2d}
{\sc Y.~Guo and L.~Wu}, {\em Geometric correction in diffusive limit of neutron
  transport equation in 2{D} convex domains}, Arch. Ration. Mech. Anal., 226
  (2017), pp.~321--403.

\bibitem{paper}
{\sc J.~W. Jang and J.~J.~L. Vel\'{a}zquez}, {\em L{TE} and non-{LTE} solutions
  in gases interacting with radiation}, J. Stat. Phys., 186 (2022), pp.~Paper
  No. 47, 62.

\bibitem{jang}
\leavevmode\vrule height 2pt depth -1.6pt width 23pt, {\em On the {T}emperature
  {D}istribution of a {B}ody {H}eated by {R}adiation}, SIAM Journal on
  Mathematical Analysis, 56 (2024), pp.~3478--3508.

\bibitem{GuoL22}
{\sc J.~Kim, Y.~Guo, and H.~J. Hwang}, {\em An {$L^2$} to {$L^\infty$}
  framework for the {L}andau equation}, Peking Math. J., 3 (2020),
  pp.~131--202.

\bibitem{finnland2}
{\sc M.~Laitinen and T.~Tiihonen}, {\em Conductive-radiative heat transfer in
  grey materials}, Quart. Appl. Math., 59 (2001), pp.~737--768.

\bibitem{finnland1}
{\sc M.~T. Laitinen and T.~Tiihonen}, {\em Integro-differential equation
  modelling heat transfer in conducting, radiating and semitransparent
  materials}, Math. Methods Appl. Sci., 21 (1998), pp.~375--392.

\bibitem{Larsen6}
{\sc E.~Larsen, G.~Pomraning, and V.~Badham}, {\em Asymptotic analysis of
  radiative transfer problems}, Journal of Quantitative Spectroscopy and
  Radiative Transfer, 29 (1983), p.~285–310.

\bibitem{Larsen3}
{\sc E.~W. Larsen}, {\em A functional-analytic approach to the steady,
  one-speed neutron transport equation with anisotropic scattering}, Comm. Pure
  Appl. Math., 27 (1974), pp.~523--545.

\bibitem{Larsen2}
\leavevmode\vrule height 2pt depth -1.6pt width 23pt, {\em Neutron transport
  and diffusion in inhomogeneous media. {I}}, J. Mathematical Phys., 16 (1975),
  pp.~1421--1427.

\bibitem{Larsen9}
\leavevmode\vrule height 2pt depth -1.6pt width 23pt, {\em Solution of neutron
  transport problems in {$L\sb{1}$}}, Comm. Pure Appl. Math., 28 (1975),
  pp.~729--746.

\bibitem{Larsen10}
\leavevmode\vrule height 2pt depth -1.6pt width 23pt, {\em Distributional
  solutions of the transport equation as a weak limit of regular solutions of
  the discrete ordinates equations}, Transport Theory Statist. Phys., 19
  (1990), pp.~179--187.

\bibitem{LarsenKeller}
{\sc E.~W. Larsen and J.~B. Keller}, {\em Asymptotic solution of neutron
  transport problems for small mean free paths}, J. Mathematical Phys., 15
  (1974), pp.~75--81.

\bibitem{Larsen8}
{\sc E.~W. Larsen and P.~F. Zweifel}, {\em {Steady, one‐dimensional
  multigroup neutron transport with anisotropic scattering}}, Journal of
  Mathematical Physics, 17 (2008), pp.~1812--1820.

\bibitem{mihalas}
{\sc D.~Mihalas and B.~W. Mihalas}, {\em Foundations of radiation
  hydrodynamics}, Oxford University Press, New York, 1984.

\bibitem{Milne}
{\sc E.~A. Milne}, {\em The {D}iffusion of {I}mprisoned {R}adiation {T}hrough a
  {G}as}, J. London Math. Soc., 1 (1926), pp.~40--51.

\bibitem{mischchenko}
{\sc M.~I. Mischchenko, L.~D. Travis, and A.~A. Lacis}, {\em Scattering,
  {A}bsorption, and {E}mission of {L}ight by {S}mall {P}articles}, Cambridge
  University Press, 2002.

\bibitem{oxenius}
{\sc J.~Oxenius}, {\em Kinetic theory of particles and photons}, vol.~20 of
  Springer Series in Electrophysics, Springer-Verlag, Berlin, 1986.
\newblock Theoretical foundations of non-LTE plasma spectroscopy.

\bibitem{Pouso}
{\sc M.~M. Porzio and O.~L\'{o}pez-Pouso}, {\em Application of accretive
  operators theory to evolutive combined conduction, convection and radiation},
  Rev. Mat. Iberoamericana, 20 (2004), pp.~257--275.

\bibitem{reedsimon1}
{\sc M.~Reed and B.~Simon}, {\em I: Functional Analysis}, Methods of Modern
  Mathematical Physics, Elsevier Science, 1981.

\bibitem{Englishpeople}
{\sc C.~M. Rooney, C.~P. Please, and S.~D. Howison}, {\em Homogenisation
  applied to thermal radiation in porous media}, European J. Appl. Math., 32
  (2021), pp.~784--805.

\bibitem{rossani}
{\sc A.~Rossani, G.~Spiga, and R.~Monaco}, {\em Kinetic approach for two-level
  atoms interacting with monochromatic photons}, Mech. Res. Comm., 24 (1997),
  pp.~237--242.

\bibitem{Rutten}
{\sc R.~J. Rutten}, {\em Radiative {T}ransfer in {S}tellar {A}tmospheres}.
\newblock \url{https://robrutten.nl/rrweb/rjr-pubs/2003rtsa.book.....R.pdf},
  2003.
\newblock Utrecht University lecture notes, 8th edition.

\bibitem{Sentis1}
{\sc R.~Sentis}, {\em Half space problems for frequency dependent transport
  equations. {A}pplication to the {R}osseland approximation of the radiative
  transfer equations}, in Proceedings of the conference on mathematical methods
  applied to kinetic equations ({P}aris, 1985), vol.~16, 1987, pp.~653--697.

\bibitem{finnland3}
{\sc T.~Tiihonen}, {\em A non-local problem arising from heat radiation on
  non-convex surfaces}, in Integral methods in science and engineering, {V}ol.
  1 ({O}ulu, 1996), vol.~374 of Pitman Res. Notes Math. Ser., Longman, Harlow,
  1997, pp.~195--199.

\bibitem{Tiihonen1}
\leavevmode\vrule height 2pt depth -1.6pt width 23pt, {\em A nonlocal problem
  arising from heat radiation on non-convex surfaces}, European J. Appl. Math.,
  8 (1997), pp.~403--416.

\bibitem{Leiunsteady}
{\sc L.~Wu}, {\em Diffusive limit with geometric correction of unsteady neutron
  transport equation}, Kinet. Relat. Models, 10 (2017), pp.~1163--1203.

\bibitem{Lei3d}
\leavevmode\vrule height 2pt depth -1.6pt width 23pt, {\em Diffusive limit of
  transport equation in 3{D} convex domains}, Peking Math. J., 4 (2021),
  pp.~203--284.

\bibitem{wuguo}
{\sc L.~Wu and Y.~Guo}, {\em Geometric correction for diffusive expansion of
  steady neutron transport equation}, Comm. Math. Phys., 336 (2015),
  pp.~1473--1553.

\bibitem{annulus}
{\sc L.~Wu, X.~Yang, and Y.~Guo}, {\em Asymptotic analysis of transport
  equation in annulus}, J. Stat. Phys., 165 (2016), pp.~585--644.

\bibitem{Zeldovic}
{\sc Y.~B. Zel'dovich and Y.~P. Raizer}, {\em Physics of {S}hock {W}aves and
  {H}igh-{T}emperature {H}ydrodynamic {P}henomena}, Dover Publications, 2002.

\end{thebibliography}
\end{document}